\documentclass[11pt,a4paper]{article}\topmargin=-1.5cm
\newcounter{licznik}[section]

\newtheorem{Definition}[licznik]{Definition}
\newtheorem{Proposition}[licznik]{Proposition}
\newtheorem{Theorem}[licznik]{Theorem}
\newtheorem{Lemma}[licznik]{Lemma}
\newtheorem{Corollary}[licznik]{Corollary}
\newtheorem{Remark}[licznik]{Remark}

\newtheorem{anz}[licznik]{Ansatz}

\usepackage{sansmath, tikz, todonotes, mathrsfs}
\usepackage[hidelinks]{hyperref}
\usepackage{enumitem}
\usepackage{mathtools}
\mathtoolsset{showonlyrefs,showmanualtags}

\def\osta{\frac{1}{\sqrt{2 \alpha}}}
\def\dmla{\Big(\delta - \fla \Big)}
\newcommand{\fla}{\frac{\lambda}{\alpha}}

\newcommand{\las}{\frac{\lambda}{\alpha^2}}
\newcommand{\atl}{\frac \alpha {2\lambda}}

\newcommand{\ftla}{\frac{2\lambda}{\alpha}}
\def\flas{\frac \lambda {\alpha^2}}
\def\fltas{\frac \lambda {2 \alpha^2}}
\def\hd{\frac \delta 2}
\def\ha{\frac 1 2}

\newcommand{\R}{\mathbb{R}}
\newcommand{\F}{\mathbb{F}}

\newcommand{\e}{\epsilon}
\newcommand{\ol}{\overline}
\newcommand{\ul}{\underline}
\newcommand{\oc}{K}
\newcommand{\uc}{k}
\newcommand{\gcm}{greatest non-positive convex minorant of $H(\,\cdot\,;c)$}

\def\hye{y_r}
\def\hyc{y_m}

\def\fotd{\frac 1 {2\d}}

\def\a{\alpha}
\def\la{\lambda}

\def\sta{\sqrt{2\alpha}}
\def\cL{\mathcal{L}}
\def\cH{\mathcal{H}}
\def\ha{\frac 1 2 }
\def\lb{\left(}
\def\rb{\right)}
\def\Fb{\overline F}
\def\Gb{\overline G}

\def\hyo{\hat y_1}
\def\hyt{\hat y_2}

\def\hyoo{\hat y_3}
\def\hytt{\hat y_4}

\newcommand{\tild}{\widetilde}

\usepackage{amsfonts,amssymb,amsmath,latexsym,epsfig,verbatim,xcolor, hyperref}
{}
\newcommand{\pd}[2]{\frac{\partial#1}{\partial#2}}
\def\beas{\begin{align*}}
\def\eeas{\end{align*}}

\def\cF{\mathcal{F}}

\def\d{\delta}

\def\e{\epsilon}

\def\t{\tau}

\def\cS{\mathcal{S}}
\def\cT{\mathcal{T}}
\def\cA{\mathcal{A}}

\def\tx{\tau}

\def\RR{\mathbb R}
\def\EE{\mathsf E}
\def\PP{\mathsf P}

\def\cF{{\cal F}}
\def\cA{{\cal A}}

\def\cS{{\cal T}}

\numberwithin{equation}{section}

\renewcommand{\bar}{\overline}
\renewcommand{\hat}{\widehat}
\renewcommand{\epsilon}{\varepsilon}

\usepackage[normalem]{ulem}
\usepackage{yfonts}
\usepackage{float}

\title{The one-shot problem: Solution to an open question of \\ finite-fuel singular control with discretionary stopping}
\author{John Moriarty\thanks{School of Mathematical Sciences, Queen Mary University of London, Mile End Road,
London E1 4NS, United Kingdom; \texttt{j.moriarty@qmul.ac.uk}} 
\ and 
Neofytos Rodosthenous\thanks{Department of Mathematics, University College London, 
Gower Street, London WC1E 6BT, United Kingdom; \texttt{n.rodosthenous@ucl.ac.uk}}}
\date{}
\begin{document}
\maketitle

\begin{abstract} 
We resolve a long-standing open question posed by Karatzas, Ocone, Wang, and Zervos (Stochastics, 2000) on finite-fuel singular stochastic control with discretionary stopping. 
Our approach introduces a novel “one-shot” solution technique, based on reduction to an auxiliary optimal stopping problem. We prove that, despite the convexity of the objective function, the waiting region of the optimal strategy need not be connected. Furthermore, we identify a qualitative transition in the optimal policy: even with arbitrarily small positive fuel, the solution can differ from the zero-fuel (pure optimal stopping) limit. These results reveal a fundamental non-robustness in such control problems, showing that zero-fuel models need not approximate behaviour under small but positive fuel constraints. 
\\\\
\noindent
{\bf MSC2020 subject classifications:} Primary 93E20, 60G40; secondary 60J65.
\end{abstract}

\section{Introduction}
\label{sec:intro}

\subsection{Setup and related work}\label{sec:setup}

Monotone follower-type problems of singular stochastic control with discretionary stopping are motivated by a number of applications. In target tracking problems~\cite{Bather67}, for example, the challenge is to decide when a controlled diffusion is `sufficiently close' to a target such as a landing site. In consumption/investment problems of financial economics under transaction costs~\cite{DavisNorman90}, they are able to model the optimal exercise of an American option within a portfolio. In a variant of the goodwill problem~\cite{LonZervos11}, the question is to optimally raise a new product's image through advertising, while determining the best time to launch the product into the market. 

In the present paper, the problem (formally defined in Section~\ref{sec:method} below) has state space $\{(x,c):x \in \R, c \geq 0\}$ where $x$ represents a {\em position} while $c$ is the finite {\em fuel level} at the disposal of a decision maker. 
A controlled Brownian motion $(X_t)_{t \geq 0}$ started from position $x$ is pushed in both the left and right (that is, respectively, decreasing and increasing) directions by a bounded variation, adapted process $\xi:= \xi^+ - \xi^-$, where the increasing processes $\{\xi_t^-\}_{t\geq 0}$, $\{\xi_t^+\}_{t\geq 0}$ are the total push to the left and right, respectively, on $X_t$ up to time $t$ and should satisfy $\xi^+_t + \xi^-_t \in [0,c]$ for all $t \geq 0$. The remaining amount of fuel available to the decision maker at each given time $t\geq 0$ is given by the process $(C_t)_{t \geq 0}$, where $C_t = c - (\xi^+_t + \xi^-_t)$.
The aim is to keep the controlled Brownian motion as close to zero as possible, in the sense that there are  running and terminal costs which are quadratic in $X_t$, the former being proportional to $\lambda>0$ and the latter proportional to $\delta>0$. Given that controlling the Brownian motion is also costly, the decision maker must find a balance between acting (that is, expending fuel), waiting, and stopping the system at a stopping time $\tau$. The aim is to minimise the total expected cost discounted at rate $\a>0$, that is, the {\em performance function}
\begin{equation} \label{OCP}
\EE\bigg[ \int_0^{\tx}e^{-\alpha t}\lambda X^2_t dt + \int_{[0,\tx]} e^{-\alpha t} d \xi^+_t + \int_{[0,\tx]} e^{-\alpha t} d \xi^-_t + e^{-\alpha\t}\delta X^2_{\t}\cdot 1_{\{\tau < \infty\}} \bigg], 
\end{equation} 
over all possible singular controls $\xi$ and stopping times $\tau$. Note that \eqref{OCP} is evenly symmetric in $x$, which allows the problem to be solved for nonnegative $x$ using monotone decreasing controls, and finally extended by even symmetry.

This problem was formulated by Karatzas et al.~\cite{KOWZ00} and solved for relatively large and small values of the running cost coefficient $\lambda$ (that is, $\lambda \in [\alpha\delta, \infty)$ and $\lambda \in [0, \lambda^*]$ respectively, for the constant $\lambda^* < \alpha\delta$ defined by \eqref{lambda*} below). These solutions are represented respectively in Figures \ref{fig:boundaries0} and \ref{fig:boundaries1}. 
In contrast, the authors propose the intermediate regime $\lambda \in (\lambda^*, \alpha\delta)$ as an open problem (we refer to the latter as the {\em control problem} in the present paper). 
Davis and Zervos~\cite{DavisZervos94} had previously solved the special case when the decision maker has infinite fuel at their disposal, and Chen and Yi~\cite{ChenYi12} studied the finite time-horizon version of the infinite-fuel problem using PDE methods. 

Our new solutions are represented in Figures \ref{fig:boundaries4} and \ref{fig:boundaries2}.  
It is surprising that, even in a `nice' convex control problem such as \eqref{OCP}, the waiting region can be disconnected. 
It may be speculated that, for this reason, this control problem has remained open for more than two decades, despite the solution in the meantime of several interesting related problems.
The novel solution features uncovered in the present paper are present only for relatively small fuel levels and, because of this, they do not feature in the infinite-fuel studies \cite{ChenYi12, DavisZervos94}. 
Related problems have also been studied for other stochastic processes. For geometric Brownian motion, Morimoto \cite{Morimoto10} used a penalisation method to solve an infinite-fuel problem with one-sided singular control and discretionary stopping. A general one-dimensional driving diffusion process was considered by Lon and Zervos~\cite{LonZervos11}, and the case of drifted Brownian motion plus a compound Poisson process was analysed by Bayraktar and Egami~\cite{BayraktarEgami08}.

\begin{figure}[tb]
\begin{center}
\vspace{-1cm}
\begin{tikzpicture}
\node[anchor=south west,inner sep=0] (image) at (0,0) {\includegraphics[width=\textwidth]{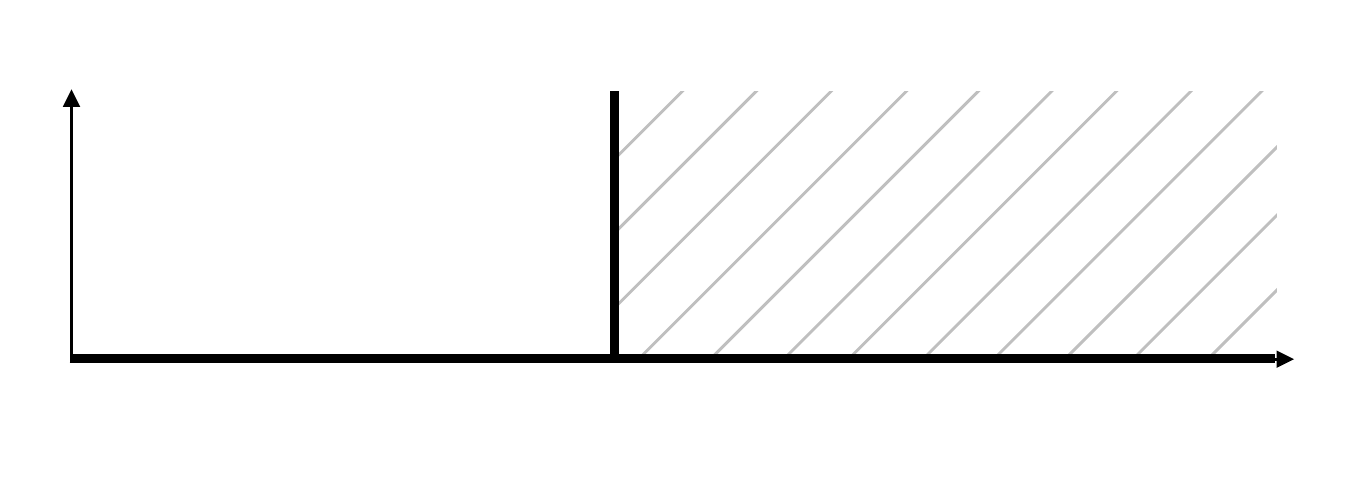}};
\begin{scope}[x={(image.south east)},y={(image.north west)}]
 \node[scale=1.2] at (0.98,0.27) {$x$};
 \node[scale=1.2] at (0.453,0.18) {$\fotd$};
 \node[scale=1.2] at (0.03,0.765) {$c$};
 \node[scale=1.2] at (0.7,0.55) {$A$};
 \node[scale=1.2] at (0.3,0.55) {$S$};
\end{scope}
\end{tikzpicture}
\vspace{-15mm}
\caption{Moving boundaries of the control problem when $\la \geq \a \d$ (obtained in \cite{KOWZ00}). In the stopping region $S$ the process is absorbed. The direction of control is south-west (grey lines). In the action region $A$, fuel is expended to drive the process in this direction towards its boundary $\partial A$, where it is then absorbed. There is no waiting region.}
\label{fig:boundaries0}
\end{center}
\end{figure}
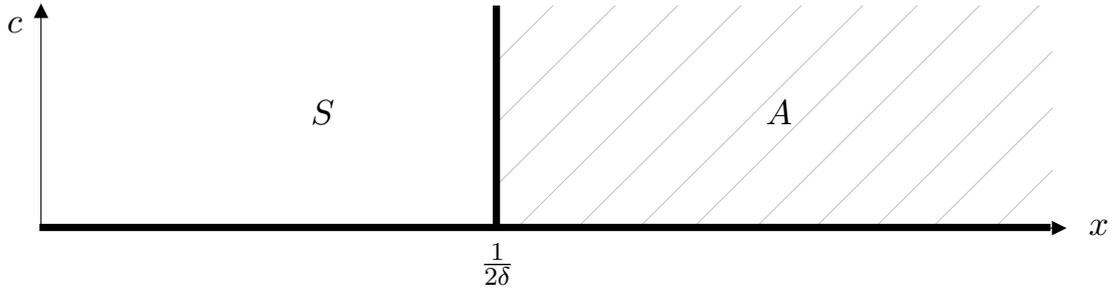

\begin{figure}[tb]
\begin{center}
\vspace{-1cm}
\begin{tikzpicture}
\node[anchor=south west,inner sep=0] (image) at (0,0) {\includegraphics[width=\textwidth]{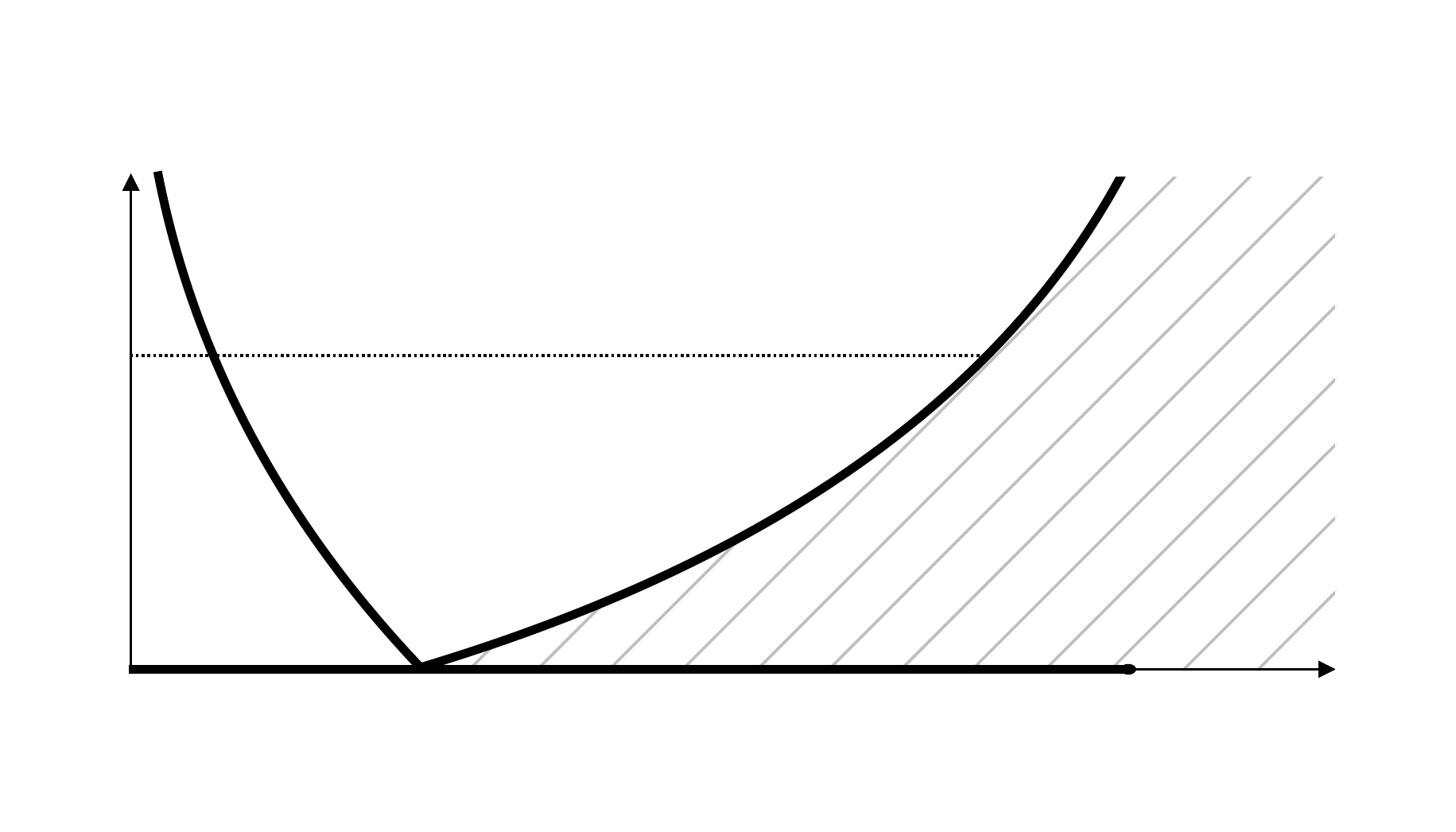}};
\begin{scope}[x={(image.south east)},y={(image.north west)}]
 \node[scale=1.2] at (0.95,0.198) {$x$};
 \node[scale=1.2] at (0.292,0.14) {$\fotd$};
 \node[scale=1.2] at (0.6,0.14) {$\atl$};
 \node[scale=1.2] at (0.055,0.75) {$c$};
 \node[scale=1.2] at (0.685,0.45) {$A$};
 \node[scale=1.2] at (0.35,0.45) {$W_1$};
 \node[scale=1.2] at (0.14,0.45) {$S$};
 \node[scale=1.2] at (0.055,0.58) {$\bar c$};
 \node[scale=1] at (0.15,0.82) {$x = F(c)$};
 \node[scale=1] at (0.78,0.82) {$x = G(c)$};
 \node[scale=1] at (0.78,0.14) {$F(0)\mathbin{=}f_0$};
\end{scope}
\end{tikzpicture}
\vspace{-17mm}
\caption{Moving boundaries of the control problem when $\la \in [\la^\dagger, \a \d)$, obtained in Section~\ref{sec:cpos} ($\la^\dagger$ is characterised by \eqref{ldag}). When the fuel level is $C_t>0$, the absorbing boundary is located at $X_t = F(C_t) < \fotd$. However when the fuel level $C_t = 0$, the absorbing boundary is located at $X_t = F(0) = f_0 > \fotd$, i.e.~$F$ is discontinuous. The right boundary $G$ is reflecting for $c>\bar c$ and repelling for $c \in (0,\bar c]$, and $W_1$ is a waiting region. The nature of the stopping region $S$ and action region $A$ are as in Figure~\ref{fig:boundaries0}.}
\label{fig:boundaries4}
\end{center}
\end{figure}

\begin{figure}[tb]
\begin{center}
\vspace{-1cm}
\begin{tikzpicture}
\node[anchor=south west,inner sep=0] (image) at (0,0) {\includegraphics[width=\textwidth, height=8cm]{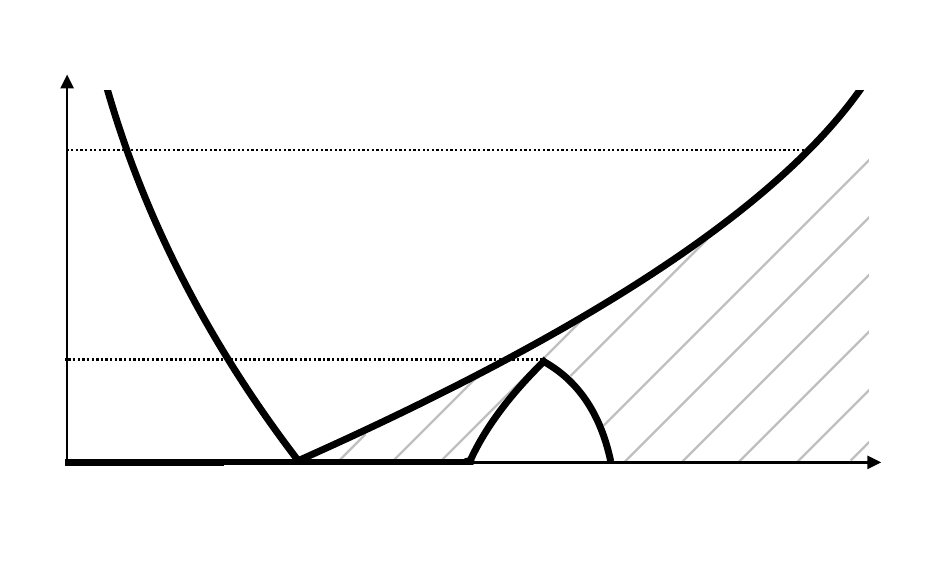}};
\begin{scope}[x={(image.south east)},y={(image.north west)}]
 \node[scale=1.2] at (0.04,0.37) {$c_\mathcal{I}$};
 \node[scale=1.2] at (0.95,0.205) {$x$};
 \node[scale=1.2] at (0.315,0.14) {$\fotd$};
 \node[scale=1.2] at (0.575,0.14) {$\atl$};
 \node[scale=1.2] at (0.488,0.14) {\footnotesize $F(0)\mathord{=}f_0$};
 \node[scale=1.2] at (0.04,0.83) {$c$};
 \node[scale=1.2] at (0.35,0.45) {$W_1$};
 \node[scale=1.2] at (0.58,0.31) {$W_2$};
 \node[scale=1.2] at (0.14,0.45) {$S$};
 \node[scale=1.2] at (0.04,0.75) {$\bar c$};
 \node[scale=1] at (0.15,0.88) {$x = F(c)$};
 \node[scale=1] at (0.89,0.88) {$x = G(c)$};
\node[scale=1] at (0.565,0.24) {$x = \bar F(c)$};
\node[scale=1] at (0.7,0.24) {$x = \bar G(c)$};
\node[scale=1.2] at (0.805,0.45) {$A$};
\end{scope}
\end{tikzpicture}
\vspace{-17mm}
\caption{Moving boundaries of the control problem when $\la \in (\la^*, \la^\dagger)$, obtained in Section~\ref{sec:newsol}.  Whereas the boundaries $F$ and $G$ border a component $W_1$ of the waiting region as in Figure~\ref{fig:boundaries4}, now an additional  repelling boundary $\bar F$ and reflecting boundary $\bar G$ create a separate waiting region $W_2$. The nature of the stopping region $S$ and action region $A$ are as in Figure~\ref{fig:boundaries0}.}
\label{fig:boundaries2}
\end{center}
\end{figure}
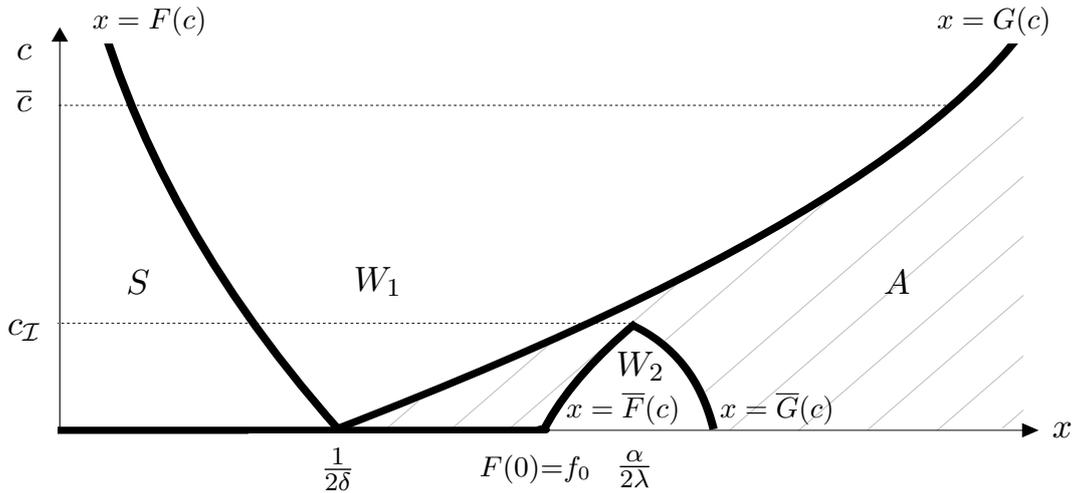

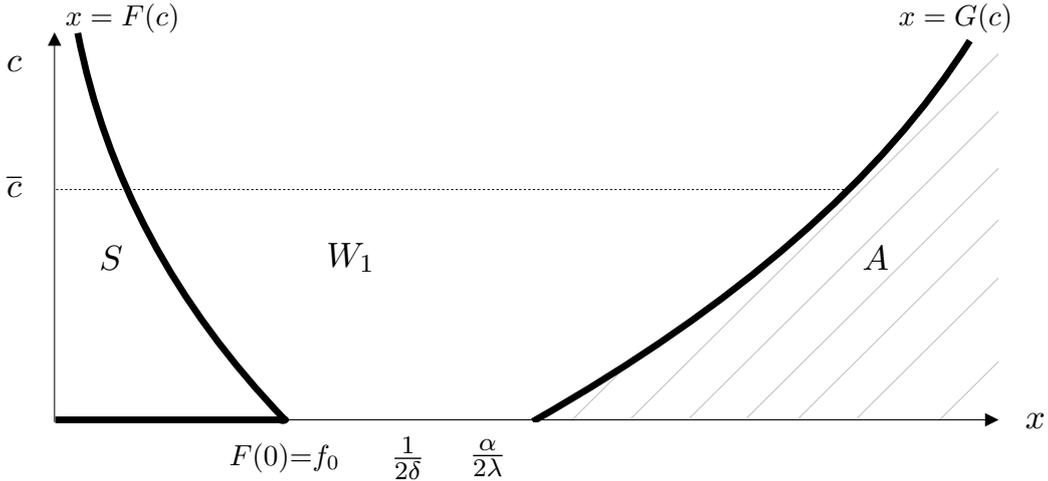
\begin{figure}[tb]
\begin{center}
\vspace{-1cm}
\begin{tikzpicture}
\node[anchor=south west,inner sep=0] (image) at (0,0) {\includegraphics[width=\textwidth]{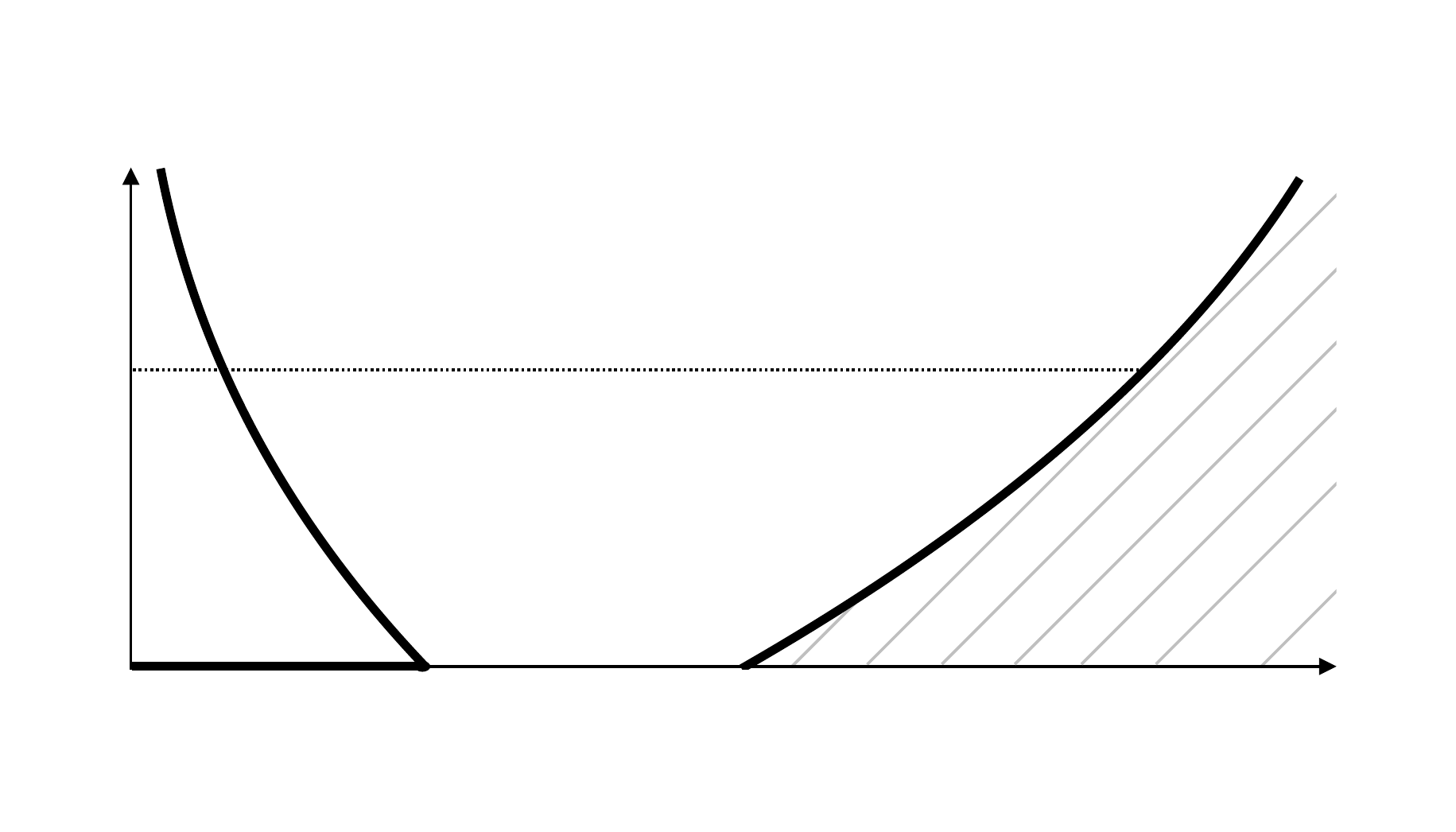}};
\begin{scope}[x={(image.south east)},y={(image.north west)}]
 \node[scale=1.2] at (0.95,0.198) {$x$};
 \node[scale=1.1] at (0.292,0.14) {\small $F(0)\mathord{=}f_0$};
 \node[scale=1.2] at (0.4,0.14) {$\fotd$};
 \node[scale=1.2] at (0.47,0.14) {$\atl$};
 \node[scale=1.2] at (0.055,0.75) {$c$};
 \node[scale=1.2] at (0.81,0.45) {$A$};
 \node[scale=1.2] at (0.35,0.45) {$W_1$};
 \node[scale=1.2] at (0.14,0.45) {$S$};
 \node[scale=1.2] at (0.055,0.56) {$\bar c$};
 \node[scale=1] at (0.15,0.82) {$x = F(c)$};
 \node[scale=1] at (0.88,0.82) {$x = G(c)$}; 
\end{scope}
\end{tikzpicture}
\vspace{-17mm}
\caption{Moving boundaries of the control problem when $f_0 \leq \fotd$ and $\la \in (\la_*,\la^*]$,  (obtained in \cite{KOWZ00}). In this case the absorbing boundary $F$ is continuous at $0$. The nature of the stopping region $S$ and action region $A$ are as in Figure~\ref{fig:boundaries0}.}
\label{fig:boundaries1}
\end{center}
\end{figure}

\subsection{The one-shot problem}
 
\label{sec:1s}

In this paper, we introduce a ``one-shot'' solution technique capturing the idea that, when fuel is scarce, the control problem may effectively reduce to a timing decision: choosing when to intervene rather than how to do so gradually. This leads to the associated one-shot problem, in which all remaining fuel is expended at a single stopping time and which is therefore essentially an optimal stopping problem. In this sense, the one-shot problem provides a natural approximation to the original control problem as the fuel level tends to zero. This reveals a fundamental non-robustness in the control problem: zero-fuel models need not approximate behaviour under small but positive fuel constraints, see Figure \ref{fig:boundaries4}. 

The solution structure of the one-shot problem serves as the key building block for tackling the unresolved cases of the finite-fuel problem. Specifically, the stopping region of the one-shot problem guides the construction of the intervention region in the finite-fuel setting. In one parameter regime this approach yields a direct solution (Section~\ref{sec:cpos}), while in another it provides the structural ansatz for the value function and free boundaries (Section~\ref{sec:newsol}). 
The primary contribution of this work is to introduce the technique and demonstrate its effectiveness in resolving the previously open cases in \cite{KOWZ00}. We expect the one-shot technique to extend naturally to more general methodology and to apply to a wider range of finite-fuel singular control 
problems with discretionary stopping; such developments are deferred to future work.

\subsection{Boundary types}\label{sec:btypes}

This section provides an overview of the boundary types found in the solutions below.
(Note that expending fuel acts on the process $(X_t,C_t)_{t \geq 0}$ in the south-west direction $(-1,-1)$ when we consider only monotone decreasing controls and $x \geq 0$.)

From Figure~\ref{fig:boundaries0}, when $\lambda \in [\alpha\delta, \infty)$ the optimal strategy has the vertical boundary $\{(\fotd, c) :c \geq 0\}$ which splits the state space into a stopping region $S$ on the left in which the problem is stopped, and an action region $A$ on the right. In region $A$, if $x \geq c + \fotd$ then all fuel is expended in a single shot and the problem is stopped, while if $x \in (\fotd, \fotd + c)$ then a `partial shot' $x - \fotd$ of fuel is expended, propelling the process to the vertical boundary, where it is  stopped (it is never optimal to wait). 
The boundary may be called {\em absorbing}, since the problem ends when it is reached. 

From Figure~\ref{fig:boundaries1}, for $\lambda \in (\lambda_*,\lambda^*]$
the optimal strategy divides the state space into regions $S$ (left), $W$ (centre) and $A$ (right) in which we respectively end the problem, do nothing, and expend fuel. 
The left moving boundary $F$ separates $S$ and $W$ and so is absorbing, since when reached (due in this case to the random fluctuations of $X$ while waiting) the problem ends. The right moving boundary $G$ has gradient 1 at the critical fuel level $\bar c$. In region $A$, below the critical fuel level $\bar c$, all fuel is expelled in a single shot and this part of $G$ is termed a {\em repelling} boundary. 
(More generally, a boundary point is called repelling if, whenever it is reached by the process $(X_t,C_t)_{t \geq 0}$, the fuel level coordinate $C$ jumps at that time; a boundary portion is repelling if all of its points are repelling.)
Above the critical fuel level $\bar c$, for $x \in (G(c),G(\bar c) + c - \bar c)$ the process is first moved to $G$ by expending a partial shot of fuel and the process then performs Skorokhod reflection at $G$ to remain in the closure $\bar{W}$ of $W$.  
Thus the part of $G$ above $\bar c$ is termed a {\em reflecting} boundary. These intuitive definitions of repelling and reflecting boundaries will be made precise in Definition \ref{def:rr} below.

\subsection{Main results}\label{sec:mainres}
Our main results are as follows. We obtain explicit analytical expressions for the value function of the control problem and completely characterise the decision maker's optimal control strategy in the unsolved parameter range. Compared to the existing results (Figures \ref{fig:boundaries0} and \ref{fig:boundaries1}), we discover two qualitatively different new parameter regimes, with waiting regions which may be called V-shaped and V--$\Lambda$-shaped, due to their two moving boundaries (Figure~\ref{fig:boundaries4}) and four moving boundaries (Figure~\ref{fig:boundaries2}) respectively.
The absorbing boundary is discontinuous, an interesting new feature which appears in both regimes. 
The additional $\Lambda$-shaped region in Figure~\ref{fig:boundaries2} is surprising, and can be understood as arising from the complex form of the gain function in the one-shot problem.

It is instructive to consider the transitions between the known and new parameter regimes. In the following we have $0 < \la_* \leq \la^* < \la^\dagger < \a\d$, where the constants $\lambda^*$ and $\la^\dagger$ are defined uniquely in Section~\ref{sec:method} and $\lambda_*$ is defined at the start of Section~\ref{sec:rewrite}.

As the quadratic running cost parameter $\lambda$ decreases from the known regime $\la \in [\a \d, \infty)$ of Figure~\ref{fig:boundaries0} to the new regime $\la \in [\la^\dagger, \a \d)$ of Figure~\ref{fig:boundaries4}, we see that:
\begin{enumerate}
\item[(a)] the single vertical boundary of Figure~\ref{fig:boundaries0} splits along its length into the two moving boundaries $F$ and $G$, creating a V-shaped waiting region $W_1$, as the cost $\lambda$ associated with waiting becomes smaller;
\vspace{-2.5mm}
\item[(b)] as in Figure~\ref{fig:boundaries0}, there is a jump in the absorbing boundary $F$ at $c=0$. However the jump is now finite, and its size is an increasing function of $\lambda$;
\vspace{-2.5mm}
\item[(c)] the boundary $G$ is reflecting for $c>\bar{c}$ and repelling for $c\leq \bar{c}$. This repulsion is full, i.e.~all fuel is expelled, and is immediately followed by absorption.
\end{enumerate}
Then as $\lambda$ decreases from the new regime $\la \in [\la^\dagger, \a \d)$ of Figure~\ref{fig:boundaries4} to the new regime $\la \in (\la^*, \la^\dagger)$ of Figure~\ref{fig:boundaries2}, we observe that:
\begin{enumerate}
\item[(a)] an additional, $\Lambda$-shaped component $W_2$ of the waiting region appears, formed by an additional repelling boundary $\Fb$ (where repulsion is full, and immediately followed by absorption) and reflecting boundary $\Gb$;
\vspace{-2.5mm}
\item[(b)] the jump in the absorbing boundary $F$ is still present but reduces in size;
\vspace{-2.5mm}
\item[(c)] the boundary $G$ is again reflecting for $c>\bar{c}$ and repelling for $c\leq \bar{c}$. Depending on the value of $c \in (0,\bar{c}]$, repulsion at $G$ may be either 
\begin{enumerate}
\vspace{-2.5mm}
\item [(i)] full, and immediately followed by absorption, or
\vspace{-1mm}
\item [(ii)] partial, and followed by waiting. More precisely, the process moves from the boundary $G$ of the V-shaped component $W_1$ of the waiting region to the boundary $\Gb$ of the $\Lambda$-shaped component $W_2$ of the waiting region. After this, the process is governed by the boundaries $\Fb, \Gb$. 
\end{enumerate}
\end{enumerate}
Finally, as $\lambda$ decreases from the new regime $\la \in (\la^*, \la^\dagger)$ of Figure~\ref{fig:boundaries2} to the known regime $\la \in (\la_*,\la^*]$ of Figure~\ref{fig:boundaries1}, we see that:
\begin{enumerate}
\item[(a)] the jump in the absorbing boundary vanishes, so that the absorbing boundary $F$ is now continuous at $c=0$;

\vspace{-2.5mm}
\item[(b)] the $\Lambda$-shaped waiting region disappears, leaving only the  boundaries $F$ and $G$; 

\vspace{-2.5mm}
    \item[(c)] the moving boundary $G$ is again reflecting for $c>\bar{c}$ and repelling for $c\leq \bar{c}$. Repulsion is again full, but is now followed by waiting until the process is absorbed at $X = f_0$;

\vspace{-2.5mm}
    \item[(d)] the boundaries $F$ and $G$ no longer converge as $c \to 0$. 
\end{enumerate}

\subsection{Structure of the paper} 
Following the background presented in Section~\ref{sec:method}, the one-shot problem is formulated in Section~\ref{sec:methods} as a family of optimal stopping problems parametrised by the initial fuel level $c$. 
To solve these stopping problems via the method of \cite{Dayanik2003}, we then present their relevant geometry, both with and without fuel.
Section~\ref{sec:cpos} carries out their solution in the open parameter regime $\la\in [\la^\dagger,\a\d)$, then uses it to characterise and develop an explicit candidate solution to the control problem for all fuel levels (V-shaped waiting region, Figure~\ref{fig:boundaries4}). 
Section~\ref{sec:newsol} performs the same steps in the remaining open regime $\la\in (\la^*, \la^\dagger)$, where the solution is more complex	 (V-$\Lambda$ shaped waiting region, Figure~\ref{fig:boundaries2}). 
Section~\ref{Sec:verification} establishes a verification theorem, applies it in both new regimes, and describes the optimal strategies. For convenience, three appendices contain a number of technical proofs, and lists of figures and symbols.

\section{Background and parameter regimes}
\label{sec:method}

For convenience we recall here the setup of the control problem. Consider a probability space $(\Omega, \cF, \PP)$ equipped with a filtration $\mathbb{F}=\{\cF_t, 0 \leq t < \infty \}$ satisfying the usual conditions of right continuity and augmentation by null sets, and let $\cT$ be the set of all $\F$-stopping times. 
Denote by $\cA$ the class of $\F$-adapted, right-continuous processes $\xi=\{\xi_t, 0 \leq t < \infty\}$ with finite total variation on any compact interval and with $\xi_{0-}=0$. 
A process $\xi \in \cA$ is considered in its minimal decomposition 
$$\xi_t=\xi^+_t-\xi^-_t, \qquad t \in [0,\infty),$$
as the difference of two non-decreasing processes $\xi^\pm \in \cA$, so that its total variation on the interval $[0,t]$ is 
$$\check \xi_t = \xi_t^+ + \xi_t^-, \qquad t \in [0,\infty],$$
and for $c \in [0,\infty)$ we write 
$$\mathcal{A}(c) = \{\xi \in \cA:\check \xi_\infty \leq c, \text{ a.s.}\}.$$ 
We assume also that $(\Omega, \cF, \PP)$ supports the $\F$-adapted Wiener process $W=\{W_t, 0 \leq t < \infty\}$. Given an {\em initial position} $x \in \R$, {\em initial fuel level} $c \geq 0$ and {\em control process} $\xi \in \cA(c)$, we define the {\em state process}
\begin{align*} 
(X_t, C_t) = (x + W_t + \xi_t, c-\check \xi_t), \quad t \geq 0.
\end{align*}
The control problem is then defined by the value function
\begin{align}
Q(x;c) 
:= \inf_{\xi \in \cA(c),\t \in \cT} \EE\left[ \int_0^{\tx}e^{-\alpha t}\lambda X^2_t dt + \int_{[0,\tx]} e^{-\alpha t} d \check \xi_t + e^{-\alpha\t}\delta X^2_{\t}\cdot 1_{\{\tau < \infty\}}
\right], \label{eq:defV}
\end{align}
where $\la$, $\a$ and $\d$ are positive constants. 

Further details on the interpretation and context of the control problem are given in \cite{KOWZ00}. Therein, solutions are obtained by constructing evenly symmetric candidate value functions using the associated variational inequality and applying a verification theorem.
This procedure is carried out in the cases $\la \in (0, \la^*]$ and $\la \in [\a \d, \infty)$, where
\begin{equation} \label{lambda*}
\la^* = \frac{\a \d}{1 + \frac{\d/\a}{(1/4\d)+(1/\sta)}} < \a \d.
\end{equation}
It is convenient to note that $f_0 \leq \frac 1 {2\d}$ holds in the previously solved case $\la \in (0,\la^*]$, while in the open case $\la \in (\la^*,\a\d)$ we have $f_0 > \frac 1 {2\d}$.
Here $f_0$ is the unique positive solution of the equation $\rho(f_0)=0$, where
\begin{align}
\rho(x) := x^2+\frac{ 2 x}{\sqrt{2\a}} -\frac{\la}{\a(\a\d-\la)} \qquad \begin{cases}
<0, &\text{for } x \in [0,f_0), \\
>0, &\text{for } x > f_0,
\end{cases}
\label{eq:defrho}
\end{align}
so that
\begin{align} \label{eq:fodef}
f_0 \equiv f_0(\la) = \frac 1 \sta \bigg(\sqrt{\frac{\a \d + \la}{\a \d - \la}}-1 \bigg)>0.
\end{align} 
The level $f_0$ is also the free boundary for the problem without fuel (that is, when $c=0$; see Theorem \ref{cor:czerosol}) and it plays a key role below. It may be checked that for fixed $\a>0$ and $\d>0$, 
\begin{equation} \label{ldag}
\text{there exists a unique value } \la^\dagger \in (\la^*,\a\d) 
\text{ satisfying } f_0(\la^\dagger)=\frac{\alpha}{2 \lambda^\dagger} 
\end{equation} 
and we therefore have the following two equivalences: 
\begin{align}
\la \in [\la^\dagger, \a \d) \quad \Leftrightarrow \quad \frac 1 {2\d} < \atl \leq f_0,
\label{eq:longineq1}\\
\la \in (\la^*, \la^\dagger) \quad \Leftrightarrow \quad \frac 1 {2\d} < f_0 < \atl .   
\label{eq:longineq2}
\end{align}

\begin{Remark}
\label{rem:halfspace}
It is not difficult to see that if $X_t = 0$ then it is optimal to immediately stop the problem, and that it is strictly suboptimal to expend fuel to make $X$ `jump across' zero. As also justified in \cite{KOWZ00} (and see \cite{ChenYi12} for a related problem), we aim to solve the control problem on the `half' state space $\{(x,c):x, c \geq 0\}$, which may be extended by even symmetry in $x$. 
We thus consider only starting values $x\geq 0$ and monotone controls by taking $\xi^+ \equiv 0$ so that $(X_t, C_t) =(x + W_t - \xi^-_t, c- \xi^-_t)$ for $x\geq 0$ and $t \geq 0$.
\end{Remark}

We conclude this section with the special case $c=0$, for which we necessarily have $\check \xi \equiv 0$ almost surely (thanks to the admissibility condition in $\mathcal{A}(0)$). 
In this case control cannot be exercised and the only available intervention is ending the problem at time $\tau$. The control problem $Q(x;0)$ in \eqref{eq:defV} then reduces to the one-dimensional optimal stopping problem
\begin{align}
Q(x,0) = \, \tild V_0(x) &:= \inf_{\t \in \cS} \EE\bigg[\int_0^{\t}e^{-\alpha t}\lambda X^2_t dt + e^{-\alpha\t} \delta X^2_{\t}\cdot 1_{\{\tau < \infty\}}\bigg]. \label{eq:defq0}
\end{align}

\section{The one-shot problem} 
\label{sec:methods}

In this section, our goals are to:
\begin{enumerate}
\item Outline the intuition behind the one-shot approach and define the one-shot problem (see Section~\ref{sec:Heuristics});
\vspace{-2mm}
\item Confirm that the small-fuel solution from \cite{KOWZ00} in the parameter range $\la \in (\la_*,\la^*]$ has this form (see Section~\ref{sec:rewrite});
\vspace{-2mm}
\item Provide preliminary results for the one-shot problem for all $\lambda$ using the method of \cite{Dayanik2003}: the complete solution without fuel, and the problem geometry with fuel (see Section~\ref{sec:anosp}). 
These will be used in the rest of the paper to solve the small-fuel problem.
\end{enumerate}

\subsection{One-shot problem definition}
\label{sec:Heuristics}

As noted in Section \ref{sec:1s},
we specify a candidate value function for the control problem by solving a related optimal stopping problem, that we call the ``one-shot problem". In the one-shot problem, all fuel is expended at the same time. Thus at any instant the decision maker can either (a) end the problem without expending any further fuel or, alternatively, they can (b) expend all fuel at the same time (choosing to end the problem sometime after that). The associated optimal stopping cost function is therefore the pointwise minimum of these two alternative costs, which may not be a smooth function.

More precisely, in the context of the control problem \eqref{eq:defV} and Remark \ref{rem:halfspace}, we consider a point $(x,c)$ lying in the waiting region $W$ with $x \geq 0$ and notice that: 
(a) ending the problem at time $\tau$ costs $\delta X_\tau^2$, 
whereas (b) expending all fuel $c$ at time $\tau$ and subsequently proceeding optimally without fuel has expected cost $c+ Q(X_\tau-c, 0)$. 
As there is also a running cost $\lambda X_t^2$ incurred during the time period $t \in [0, \tau]$, the one-shot problem $\tild V(x;c)$ is defined by 
\begin{align}
\tild V(x;c) &:= \inf_{\t \in \cS} \EE\bigg[\int_0^{\t}e^{-\alpha t}\lambda X^2_t dt + e^{-\alpha\t} \min \Big\{ \delta X^2_{\t},\, c+ \tild V_0(X_\t-c) \Big\}\cdot 1_{\{\tau < \infty\}}\bigg], \label{eq:defqc}
\end{align}
for all $\alpha, \lambda, \delta > 0$ and $c \geq 0$, where $\tild V_0$ is the value function without fuel defined in \eqref{eq:defq0}.

\subsection{Rewriting a known solution}
\label{sec:rewrite}

In this section, we rewrite the small-fuel part $c \in (0,\bar c)$ of a known solution from \cite[Section 10]{KOWZ00} for the subcase $\la \in (\la_*,\la^*]$, where $\la_* \in (0,\la^*]$ is uniquely defined in \cite[Proposition 8.7]{KOWZ00}. 
Our aim is to confirm that it has the form of the one-shot problem \eqref{eq:defqc}. 
The solution involves a unique fuel level $\bar c$ (cf.~Figure~\ref{fig:boundaries1}) such that 
\[
G'(c) > 1 \text{ for } c \in (0,\bar c) \text{ and }
G'(c) \in (0,1] \text{ for } c \in [\bar c, \infty).
\]
In light of Remark~\ref{rem:halfspace}, consider a point $(x,c)$ lying in the waiting region $W$ with $x \geq 0$ and $c \in (0,\bar c)$. The optimal policy is to wait until either the problem ends with absorption at the left boundary or until all fuel is expended instantaneously, in a single shot, upon reaching the right boundary. 
This suggests that for such initial values $(x,c)$ the control problem is in fact one of optimally stopping either at the left or right boundary. 
If the right boundary is hit first then the process is instantaneously `repelled' from the state $(X_\tau,c)$ to $(X_\tau-c, 0)$, after which the controller acts optimally without fuel. 
By the principle of optimality, the corresponding total expected cost from this repulsion action is $c+ \tild V_0(X_\tau-c)$, where $\tild V_0$ is the value function with no fuel \eqref{eq:defq0}.

In particular, the corresponding optimal stopping rule then has no direct dependence on the optimal policy for other fuel levels $\tild c \in (0,c) \cup (c,\bar c)$. 
This implies that the solution is derived by solving a family of optimal stopping problems parametrised by $c \in (0,\bar c)$, each taking the form of the one-shot problem $\tild V(x;c)$ in \eqref{eq:defqc}, and which may be solved independently of each other. 
We remark that this formulation as a family of optimal stopping problems, with $c$ acting only as a parameter, is consistent with the assumption made in \cite{KOWZ00} in the subcase $\la \in (\la_*,\la^*]$, that $Q$ is continuously differentiable in $x$ across both the right- and left-hand boundaries.

\subsection{Preliminaries to the one-shot approach} \label{sec:anosp}

In this section we rewrite problem \eqref{eq:defq0} for $c=0$, and the problems \eqref{eq:defqc} for $c > 0$, as a family of  parameter-dependent optimal stopping problems with terminal costs. 
As the driving process is a regular one-dimensional diffusion, the stopping problems can be solved constructively by the method of Dayanik and Karatzas \cite{Dayanik2003}, as we do below, or by any preferred method.
Then, in Sections \ref{sec:nofuel} and \ref{sec:h} 
we provide their associated geometry. Proofs of results not included in the section can be found in Appendix \ref{AppA3}.

Integrating the `running cost' term $\int_0^{\t}e^{-\alpha t}\lambda X^2_t dt$ by parts gives 
\begin{align}
V(x;c) &:= \inf_{\t \in \cT} \EE\big[ e^{-\alpha\tau}h(X_{\tau};c)\cdot 1_{\{\tau < \infty\}}\big] 
= \tild V(x;c) - \fla x^2 - \flas, \label{def-V} \\
V_0(x) &:= \inf_{\t \in \cT} \EE\big[ e^{-\alpha\tau}h_l(X_{\tau})\cdot 1_{\{\tau < \infty\}}\big]
= \tild V_0(x) - \fla x^2 - \flas, \label{eq:protonofuel}
\end{align}
where the obstacles $h$ and $h_l$ are respectively given by
\begin{align}
h(x;c) &:= h_l(x) \wedge h_r(x;c), \label{eq:hdef} \\
h_l(x) &:= \Big( \d - \fla \Big) x^2 -\flas, \label{eq:hldef} \\
h_r(x;c) &:= \tild V_0(x-c) + c - \fla x^2 -\flas.
\label{def-G0}
\end{align}
The equivalent forms \eqref{def-V}--\eqref{eq:protonofuel} enable the solution of problems \eqref{eq:defq0} and \eqref{eq:defqc} using the characterisation via concavity of excessive functions. To account for exponential discounting we appeal to the generalisation of this method  presented in \cite{Dayanik2003}, as follows. Define 
$$\phi_{\a}(x):=e^{-\sqrt{2\a}x} \quad \text{and} \quad 
\psi_{\a}(x):=e^{\sqrt{2\a}x}$$ 
to be the decreasing and increasing solutions respectively of the characteristic equation $(\cL - \a)u =0$, where $\cL:= \frac{1}{2}\frac{d^2}{dx^2}$ is the infinitesimal generator of Brownian motion. 
As in \cite[Eq.~(4.6)]{Dayanik2003}, we set
\begin{equation}
\label{def-F}
\Psi(x):=\frac{\psi_{\a}(x)}{\phi_{\a}(x)} = e^{2\sqrt{2\a}x}, \qquad x \in \RR.
\end{equation}
With an obvious terminology we will refer to the point $x$ as being in the {\em natural scale} and the point $y=\Psi(x)$ as being in the {\em transformed scale}.
Given a function $g: \R \times [0, \infty) \to \R$, define the transformation $\Phi$ by\footnote{With a slight abuse of notation, if the function $g$ does not depend on the $c$ variable then we omit the $c$ variable wherever it occurs in \eqref{def-H}.}
\begin{align}\label{def-H}
\Phi(g)(y;c):=
\left\{
\begin{array}{ll}
\frac{g(\Psi^{-1}(y);c)}{\phi_\a(\Psi^{-1}(y))}, & y>0,\\[+4pt]
0, & y=0,
\end{array}
\right.
\end{align}
and define the {\em transformed obstacles} (where $h_l(x;c) \equiv h_l(x)$)
\[
H = \Phi(h), \qquad H_r = \Phi(h_r), \qquad H_l = \Phi(h_l). 
\]

Recalling Remark \ref{rem:halfspace}, similarly, we study the optimal stopping problems \eqref{def-V}--\eqref{eq:protonofuel} on the domain $[0,\infty)$, with absorption at $x=0$. 
In view of \eqref{def-F}, this is equivalent to taking $y \in [1,\infty)$ in the transformed scale and the subsequent analysis focuses on this part of the state space. 
Since the optimal stopping problem is one of minimising costs, solutions are characterised via convexity, as follows. 

\vspace{3mm}
\begin{Proposition}[\cite{Dayanik2003}]\label{prop:DayKar} 
Fix $c > 0$ and let $W(\,\cdot\,;c), \; W_0(\cdot):[1,\infty) \to \R$ be the greatest non-positive convex minorants of $H(\,\cdot\,;c), \; H_l(\cdot):[1,\infty) \to \R$, respectively. 
Then
\begin{itemize}
\vspace{-1mm}
\item[\rm (i)] problem \eqref{def-V} has value function $V(x;c)=\phi_\alpha(x) W(\Psi(x);c)$ for all $x\in [0,\infty)$ and optimal stopping region $\cS_c=\Psi^{-1}(\cS^{W}_c)$, where $\cS^{W}_c:=\{y>1:W(y;c)=H(y;c)\}$; 

\vspace{-2mm}
\item[\rm (ii)] problem \eqref{eq:protonofuel} has value function $V_0(x)=\phi_\alpha(x) W_0(\Psi(x))$ for all $x\in [0,\infty)$ and the optimal stopping region is $\cS_0=\Psi^{-1}(\cS^{W}_0)$, where $\cS^{W}_0:=\{y>1:W_0(y)=H_l(y)\}$.
\end{itemize}
\end{Proposition}

These results will be used as follows. 
In Section~\ref{sec:nofuel} we write down the geometry of the transformed obstacle $H_l$, which is used to solve \eqref{eq:protonofuel}. 
In Section~\ref{sec:h} the value function $V_0$ is then used to study the obstacle $h$ (cf.\ \eqref{eq:hdef}--\eqref{def-G0}) and the geometry of its transformation $H$. 
This enables the solution of the one-shot problem \eqref{def-V}, which is used to create candidate solutions to the control problem in Sections \ref{sec:cpos} and \ref{sec:newsol}.

\subsubsection{Geometry without fuel}
\label{sec:nofuel}

As noted above, when the fuel level is 0, stopping the problem is the only available action, and the control problem reduces to the optimal stopping problem $\tild V_0$ of \eqref{eq:defq0}. Its solution can be obtained by solving problem $V_0$ of  \eqref{eq:protonofuel}, whose obstacle is the absorption cost $h_l$ specified in \eqref{eq:hldef}.
From \eqref{def-H}, its transformation $H_l = \Phi(h_l)$ is given by 
\begin{align}\label{eq:Hl}
H_l(y)=\sqrt{y}\Big(-\flas + \frac{\alpha\delta - \lambda}{8\alpha^2}(\ln y)^2\Big), & \quad y \geq 1.
\end{align} 
We present its geometry in the following result.

\begin{Lemma}[cf.~Figure~\ref{fig:1}]\label{lem:Hlprops} 
For any $\la \in (0,\a \d)$ 
the function $y \mapsto H_l(y)$ is:
\begin{enumerate}
\item[\rm (i)] continuous on $[1,\infty)$ with $H_l(1)=-\flas$ and $H_l'(1)=-\frac{\la}{2\a^2}$,

\vspace{-2mm}
\item[\rm (ii)] 
strictly convex on $\Big[1, \Psi\Big(\sqrt{\frac{\d}{\a\d-\la}}\Big) \Big]$ and strictly concave on $\Big[\Psi\left(\sqrt{\frac{\d}{\a\d-\la}}\right), \infty\Big)$,

\vspace{-1mm}
\item[\rm (iii)] \label{lempart:413} strictly decreasing on $[1,\Psi(f_0)]$ and strictly increasing on $[\Psi(f_0),\infty)$.
\end{enumerate}
\vspace{-1mm}
Note that $\sqrt{\frac{\d}{\a\d-\la}} > f_0(\la)$, for all $\la \in (0,\a \d)$.
\end{Lemma}

\begin{figure}[tb]
\begin{center}
\vspace{-1cm}
\begin{tikzpicture}
\node[anchor=south west,inner sep=0] (image) at (0,0) {\includegraphics[width=\textwidth]{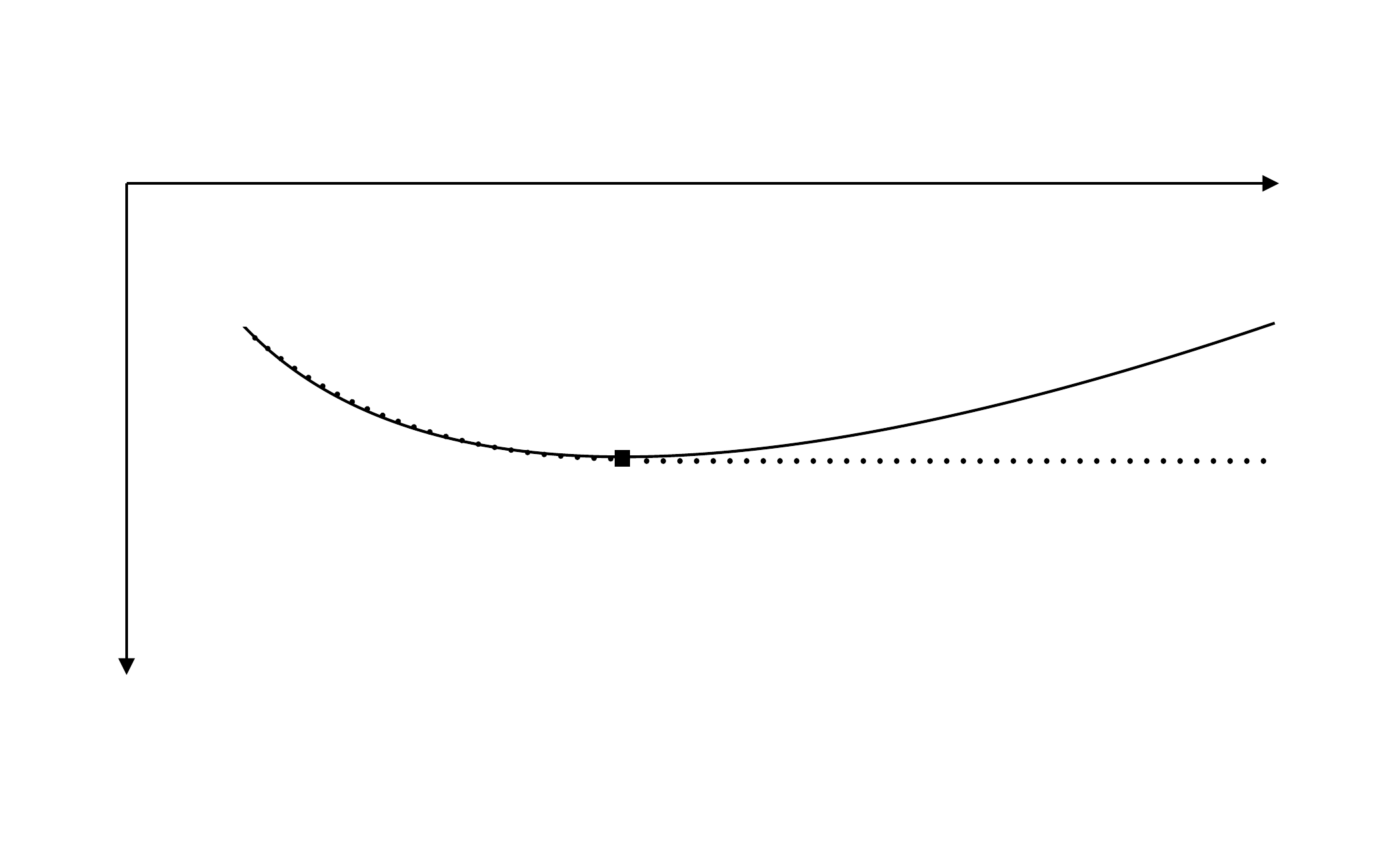}};
\begin{scope}[x={(image.south east)},y={(image.north west)}]
\draw (0.17,0.77) -- (0.17,0.79);
\draw (0.35,0.77) -- (0.35,0.79);
\draw (0.44,0.77) -- (0.44,0.79);
\draw (0.095,0.46) -- (0.085,0.46);
 \node[scale=1] at (0.06,0.46) {$B_0$};
 \node[scale=1] at (0.045,0.22) {$H_l,W_0$};
 \node[scale=1] at (0.17,0.73) {$1$};
 \node[scale=1] at (0.17,0.83) {$0$};
 \node[scale=1] at (0.35,0.83) {$\fotd$};
 \node[scale=1] at (0.44,0.83) {$f_0$};
 \node[scale=1] at (0.92,0.83) {$x$};
 \node[scale=1] at (0.92,0.73) {$y = \Psi(x)$};
\end{scope}
\end{tikzpicture}
\vspace{-25mm}
\caption{A geometric representation of the optimal stopping problem $V_0$ of \eqref{eq:protonofuel} when $\lambda \in (\lambda^*, \a \d)$, showing the transformed obstacle $H_l$ on $[1,\infty)$ (solid curve) and its greatest non-positive convex minorant $W_0$ (dotted curve). The two curves coincide for $y \in [1,\Psi(f_0)]$ and the tangency point is shown (square marker). For convenience the natural scale $x=\Psi^{-1}(y)$ is also given, and the stopping region in the natural scale is $[0,f_0]$.}
\label{fig:1}
\end{center}
\end{figure}
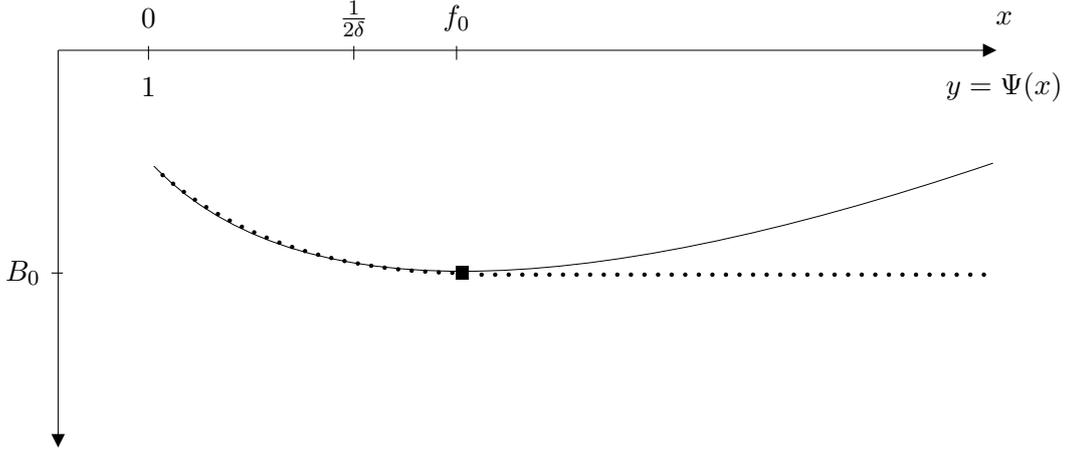

We may now solve the control problem without fuel. 

\begin{Theorem} \label{cor:czerosol}
For any $\la \in (0,\alpha\delta)$ the value function without fuel of \eqref{eq:defq0} is the continuously differentiable function 
\begin{align}\label{eq:Vtild0}
Q(x,0) = \tild V_0(x) &= \begin{cases}
\d x^2, & 0 \leq x \leq f_0, \\[+1mm]
B_0 e^{-x\sta} + \fla x^2 + \flas, & x > f_0,
\end{cases}
\end{align}
where the optimal stopping time is $\tau^* = \inf\{t \geq 0 : X_t \leq f_0\}$ and 
\begin{align}\label{eq:B0}
B_0 :=- \frac{2f_0}{\a \sta}(\a \d - \la)e^{f_0\sta} < 0.
\end{align}
\end{Theorem} 

\begin{proof}
Figure~\ref{fig:1} illustrates the geometry of the transformed obstacle $H_l:[1,\infty) \to \R$, which was obtained in Lemma \ref{lem:Hlprops}. Its greatest non-positive convex minorant is the continuously differentiable function
\begin{equation*} 
W_0(y) := 
\left\{
\begin{array}{ll}
H_l(y), & y \in [1,\Psi(f_0)], \\[+1mm]
B_0, & y > \Psi(f_0),
\end{array}
\right. 
\end{equation*}
where it follows from \eqref{eq:defrho} that $B_0$ is given by \eqref{eq:B0} and is strictly negative.
The required result now follows from Proposition \ref{prop:DayKar}.(ii) and \eqref{eq:protonofuel}.
\hfill$\Box$ 
\end{proof}

\vspace{3mm}
The no-fuel problem was originally solved analytically in \cite[Section 6]{KOWZ00} via a free-boundary problem. We provide this geometric solution for completeness, and because its geometry will be referenced below.

\subsubsection{Geometry with fuel}
\label{sec:h}

Recall now the one-shot problem $V$ of \eqref{def-V} with obstacle $h = h_l \wedge h_r$.
In this section, we study the geometry of its transformed obstacle $H=\Phi(h)=H_l \wedge H_r$. 
This will serve as the fundamental building block to specify the new moving boundaries $F$ and $G$ used in the candidate control problem solutions of Sections \ref{sec:cpos} and \ref{sec:newsol}.
As the geometry of $H_l$ was presented in Lemma \ref{lem:Hlprops}, we focus on $h_r$ and $H_r$ in what follows.

\begin{Lemma} \label{lem:littleh}
Fix $c>0$. The obstacle $x \mapsto h_r(x;c)$ is continuously differentiable and divides at $x=f_0+c$ into two parts:
\begin{align} 
h_r(x;c) = \left\{
\begin{array}{ll}
h_{r1}(x;c):=  \d c(c-2x) + \frac{\a\d - \la}{\a} x^2 + c -\flas, & x \in [0,f_0+c],\\
h_{r2}(x;c):= \fla c(c-2x) +  c+ B_0 e^{-(x-c) \sta}, & x \geq f_0+c.
\end{array}
\right.
\label{def-G1}
\end{align}
The obstacle $h(\cdot;c)$ of $V(\cdot;c)$ in \eqref{def-V} is given by
\begin{align}
h(x;c) = \left\{
\begin{array}{ll}
h_l(x) & \text{ for } x \leq x_c, \\
h_{r1}(x;c) & \text{ for } x \in [x_c, f_0 + c], \\
h_{r2}(x;c) & \text{ for } x \geq f_0 + c, \\
\end{array}
\right. \label{eq:hbigf0}
\end{align}
where
\begin{equation} \label{x_c}
x_c:=\frac 1 {2\d}+\frac c 2 \in \big(0, f_0+c \big),\quad \text{for every $c>0$}. 
\end{equation}
\end{Lemma}

Using this result we may establish the relevant geometry of the transformed obstacle $H_r(\cdot;c)$. Recalling the parameter ranges identified in \eqref{eq:longineq1}--\eqref{eq:longineq2}, this is illustrated in Figure~\ref{fig:Fig1synth} for $\la\in[\la^\dagger, \a\d )$ and in Figure~\ref{fig:Fig1synth2} for $\la\in(\la^*,\la^\dagger)$. It will be useful to define the critical $y$-values
\begin{equation} \label{ys}
y_c := \Psi(x_c) , \quad \hye(c) := \Psi(f_0+c) 
\quad \text{and} \quad 
\hyc(c) := \Psi\Big(\atl + \frac c2\Big), 
\end{equation}
and the critical fuel $c$-values
\begin{equation} \label{Kk}
\oc := 2 \Big(\sqrt{\frac{\d}{\a\d-\la}} - \frac 1{2\d} \Big) > 0
\quad \text{and} \quad 
\uc := 2 \Big(\atl - f_0 \Big).
\end{equation}
From \eqref{eq:longineq1}--\eqref{eq:longineq2}, $\uc$ is positive and becomes relevant only in the case $\la \in (\la^*, \la^\dagger)$. 
In this case, using \eqref{eq:defrho}--\eqref{eq:fodef} one can also show that $\uc < \oc$.

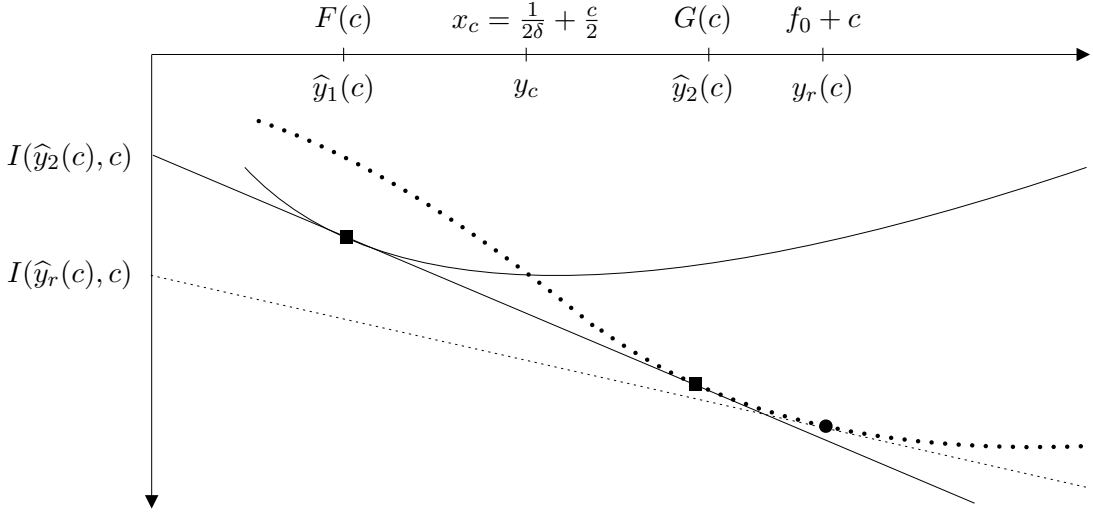
\begin{figure}[tb]
\begin{center}
\vspace{-1cm}
\begin{tikzpicture}
\node[anchor=south west,inner sep=0] (image) at (0,0) {\includegraphics[width=\textwidth]{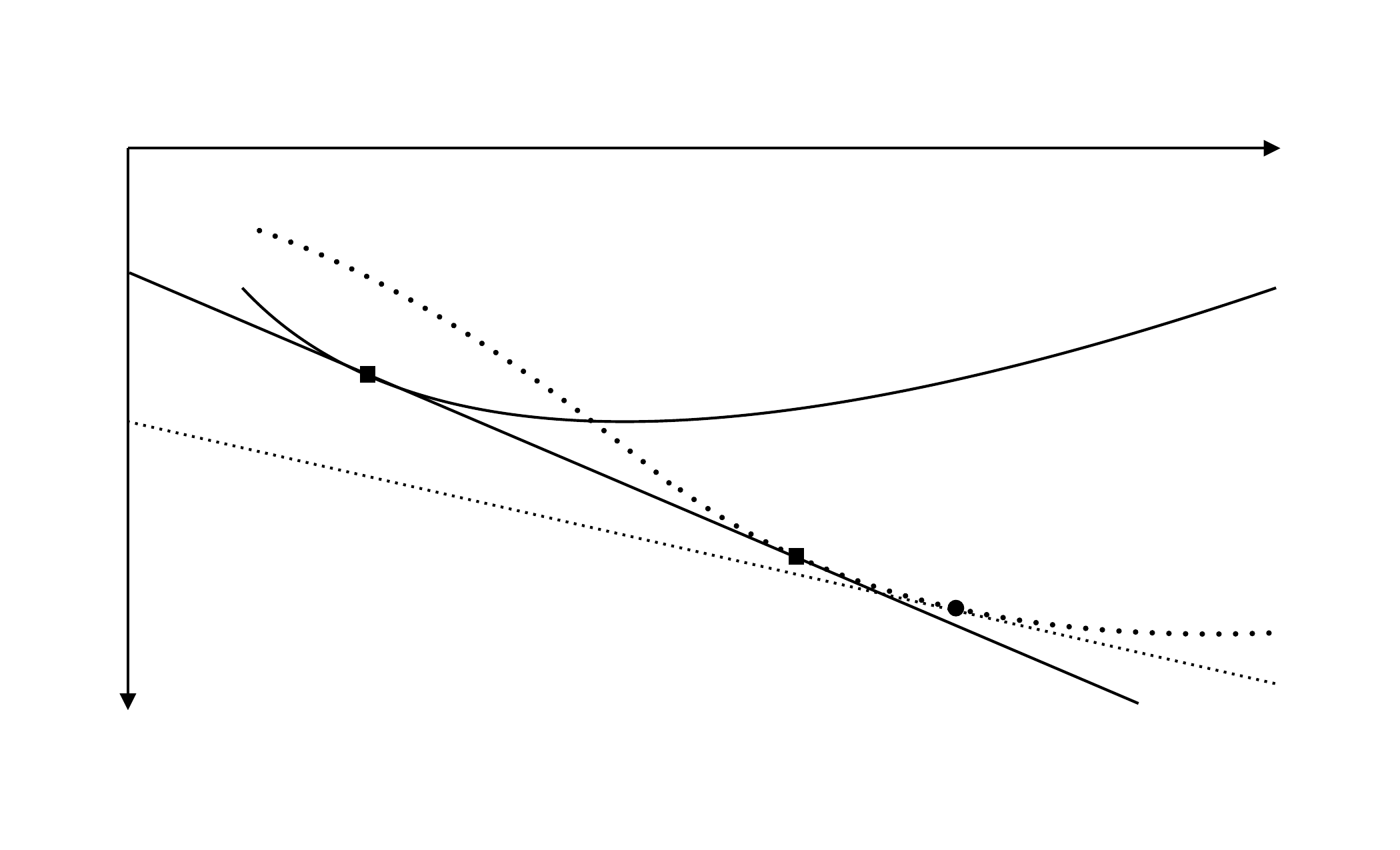}};
\begin{scope}[x={(image.south east)},y={(image.north west)}]
\draw (0.58,0.815) -- (0.58,0.835);
\draw (0.26,0.815) -- (0.26,0.835);
\draw (0.42,0.815) -- (0.42,0.835);
\draw (0.68,0.815) -- (0.68,0.835);
 \node[scale=1] at (0.02,0.675) {$I(\hat{y}_2(c),c)$};
 \node[scale=1] at (0.02,0.505) {$I(\hat{y}_r(c),c)$};
 \node[scale=1] at (0.26,0.775) {$\hat y_1(c)$};
 \node[scale=1] at (0.26,0.875) {$F(c)$};
  \node[scale=1] at (0.575,0.775) {$\hat y_2(c)$};
 \node[scale=1] at (0.575,0.875) {$G(c)$};
 \node[scale=1] at (0.68,0.875) {$f_0+c$};
 \node[scale=1] at (0.68,0.775) {$y_r(c)$};
 \node[scale=1] at (0.42,0.875) {$x_c = \fotd +\frac c 2$};
 \node[scale=1] at (0.42,0.775) {$y_c$};
  \node[scale=1] at (0.92,0.875) {$x$};
 \node[scale=1] at (0.92,0.775) {$y = \Psi(x)$};
 \node[scale=1] at (0.03,0.18) {$H_l, H_r, W$};
\end{scope}
\end{tikzpicture}
\vspace{-2cm}
\caption{A geometric representation of the optimal stopping problem $V(x;c)$ of \eqref{def-V} for $\la \in [\la^\dagger, \a\d)$ and $c>0$ fixed and sufficiently small. 
The transformed obstacle $y \mapsto H_r(y;c)$ is indicated by the dotted curve. 
To the left (resp.~right) of the circular marker at $y=y_r(c)$, $H_r$ is given by $H_{r1}$ (resp.~$H_{r2}$). 
Its common tangent $r_{\hyt(c)}$ (solid line) with the transformed obstacle $H_l$ (solid curve) is shown together with the points of common tangency at $x=F(c)$ and $x=G(c)$ (square markers). 
Intercepts at the vertical axis are labelled as in Definition \ref{def:tangents}. 
As $c \to 0$, $H(\cdot;c):=H_l(\cdot) \wedge H_r(\cdot;c)$, and hence also its greatest non-positive convex minorant, converge pointwise to $W_0(\cdot)$ shown in Figure~\ref{fig:1}.}
\label{fig:Fig1synth}
\end{center}
\end{figure}

\begin{figure}[tb]
\begin{center}
\vspace{-1cm}
\begin{tikzpicture}
\node[anchor=south west,inner sep=0] (image) at (0,0) {\includegraphics[width=\textwidth]{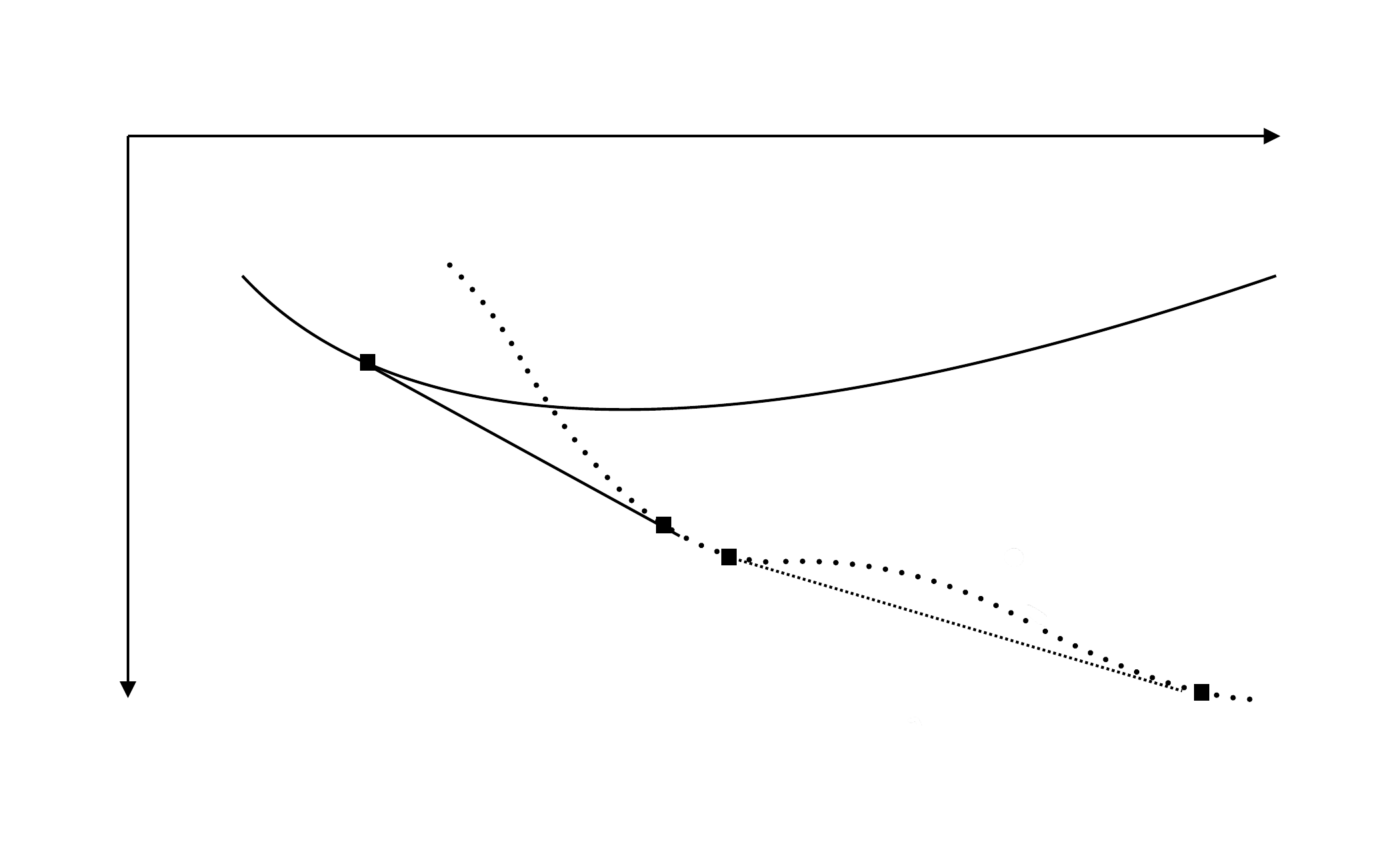}};
\begin{scope}[x={(image.south east)},y={(image.north west)}]
\draw (0.52,0.83) -- (0.52,0.85);
\draw (0.26,0.83) -- (0.26,0.85);
\draw (0.472,0.83) -- (0.472,0.85);
\draw (0.86,0.83) -- (0.86,0.85);
 \node[scale=1] at (0.26,0.79) {$\hat y_1(c)$};
 \node[scale=1] at (0.26,0.89) {$F(c)$};
  \node[scale=1] at (0.54,0.79) {$\hat{y}_3(c)$};
 \node[scale=1] at (0.85,0.79) {$\hat{y}_4(c)$};
 \node[scale=1] at (0.465,0.89) {$G(c)$};
 \node[scale=1] at (0.465,0.79) {$\hat y_2(c)$};
  \node[scale=1] at (0.03,0.19) {$H_l, H_r, W$};
    \node[scale=1] at (0.95,0.89) {$x$};
 \node[scale=1] at (0.95,0.79) {$y=\Psi(x)$};
 \draw[fill] (0.56,0.341) circle [radius=0.09cm];
\end{scope}
\end{tikzpicture}
\vspace{-2cm}
\caption{A geometric representation of the optimal stopping problem $V(x;c)$ of \eqref{def-V} for $\la \in (\la^*,\la^\dagger)$ and $c>0$ fixed and sufficiently small. The transformed obstacle $y \mapsto H_r(y;c)$ is indicated by the dotted curve.
To the left (resp.~right) of the circular marker, $H_r$ is given by $H_{r1}$ (resp.~$H_{r2}$). The common tangent between $H_l$ (solid curve) and $H_{r1}$, with tangency points $x=F(c)$ and $x=G(c)$, is shown by a solid line. The common tangent between $H_{r1}$ and $H_{r2}$, with tangency points $x=\Psi^{-1}(\hat{y}_3(c))$ and $x=\Psi^{-1}(\hat{y}_4(c))$, is shown by a fine dotted line.}
\label{fig:Fig1synth2}
\end{center}
\end{figure}

\begin{Lemma}\label{lem:Hconvexity}
The transformed obstacle $y \mapsto H_r(y;c)$ is continuously differentiable. 
\begin{itemize}
\vspace{-1mm}
\item[\rm (I)] For all $\la \in [\la^\dagger, \a\d)$, 
\begin{itemize}
\vspace{-2mm}
\item[\rm (i)] if $c>\oc$, there exists $y_v(c) \in [y_c, \hye(c))$ such that the function $y \mapsto H_r(y;c)$ is strictly concave on $(y_c,y_v(c))$ and is strictly convex on $(y_v(c),\infty)$,   

\item[\rm (ii)] if $c\leq\oc$, then the function $y \mapsto H_r(y;c)$ is strictly convex on $(y_c,\infty)$.
\end{itemize}

\vspace{-2mm}
\item[\rm (II)] For all $\la \in (\la^*,\la^\dagger)$, 
\begin{itemize}
\vspace{-2mm}
\item[\rm (i)] if $c>\oc$, there exists $y_v(c) \in [y_c, \hye(c))$ such that the function $y \mapsto H_r(y;c)$ is strictly concave on $(y_c,y_v(c))$ and is strictly convex on $(y_v(c),\infty)$,   

\item[\rm (ii)] if $c\in[\uc,\oc]$, the function $y \mapsto H_r(y;c)$ is strictly convex on $(y_c,\infty)$,

\item[\rm (iii)] if $c<\uc$, then the function $y \mapsto H_r(y;c)$ is strictly convex on $(y_c, \hye(c))$,  strictly concave on $(\hye(c), \hyc(c))$, and strictly convex on $(\hyc(c), \infty)$.  
\end{itemize}
\end{itemize}
\vspace{-2mm}
(In cases {(I).(ii)} and {(II).(ii)--(iii)} we simply set $y_v(c):=y_c$).
Further,
\begin{equation} \label{eq:Hneglim}
\lim_{y \to \infty}H_r(y;c) = -\infty, 
\quad \text{and} \quad 
\lim_{y \to \infty}\pd{H_r}{y}(y;c) = 0.
\end{equation}
\end{Lemma}

\section{Analysis of the case $\la\in [\la^\dagger,\a\d)$}
\label{sec:cpos}

In this section we solve the one-shot problem \eqref{def-V} for $\la\in [\la^\dagger,\a\d)$. This yields a candidate solution to the control problem, which will be verified in Section~\ref{Sec:verification}. 
Proofs of results not included below can be found in Appendix \ref{AppA4}.
\smallskip 

We apply Proposition \ref{prop:DayKar}.(i), constructing appropriate tangents to the transformed obstacles $H_l$ and $H_r$. The following definition will be useful:

\begin{Definition} \label{def:tangents}
For $y \neq y_c$, let $r_y(\,\cdot\, ; c)$ be the straight line tangent to $H(\cdot;c)$ at $y$, and let $r_{y_c}(\,\cdot\, ; c)$ be the straight line tangent to $H_r(\cdot;c)$ at $y_c$:\,\footnote{Note that since the obstacle $H_l$ does not depend on the fuel level $c$, neither does its tangent $r_y(\cdot;c)$ when  $y < y_c$; thus for $y < y_c$ we will simply write the tangent as $z \mapsto r_y(z)$.}
\begin{align*} 
r_y(z;c)&= \begin{cases} 
\pd H y (y;c)(z-y)+H(y;c), \quad y \neq y_c, \\[+5pt]
\pd {H_r} y (y;c)(z-y)+H(y;c), \quad y = y_c. 
\end{cases}
\end{align*}
Let $I(y;c)=r_y(0;c)$ be its intercept at the vertical axis, and let $P_r(y;c)$ be the following signed distance between $r_y(\,\cdot\, ; c)$ and $H_l(\cdot)$:
\begin{align} \label{def-P_r}
P_r(y;c)&=\sup_{z \in [1, y_c]}\big(r_y(z;c)-H_l(z)\big) .
\end{align} 
\end{Definition}

In the present section, for each $c>0$ the transformed obstacle $H$ has one doubly tangent line. 
As illustrated in Figure~\ref{fig:Fig1synth}, for small $c$ one of the tangency points lies on $H_l$ and the other on $H_{r1}$. 
This is the key difference with the known solution discussed in Section \ref{sec:rewrite}, where the right-hand tangency point instead lies on $H_{r2}$ for small $c$;  
this geometry being forced by an ansatz in \cite{KOWZ00} for the boundary of the action region (see Remark \ref{rem:noapply} for details). 
This leads to new equations for the moving boundaries of the control problem, and while it remains true that one boundary is absorbing and the other repelling for small $c$, the boundaries now converge at $x=\fotd$ as $c \downarrow 0$ (contrast Figures \ref{fig:boundaries4}--\ref{fig:boundaries2} with Figure~\ref{fig:boundaries1}). 
Recalling from Theorem \ref{cor:czerosol} that $f_0$ is the absorbing boundary in the no-fuel problem and that $\fotd < f_0$ in this regime, the absorbing boundary is therefore now discontinuous at $c=0$ in the sense that $\lim_{c \downarrow 0}F(c) = \fotd < f_0 =: F(0)$. 

We have referred above to `reflecting' and `repelling' boundaries, and now make this language more precise.
\begin{Definition}\label{def:rr}
Suppose $\mathcal{I} \subseteq [0,\infty)$. 
A moving boundary $c \mapsto R(c)$ between regions $W$ and $A$ of a control policy in problem \eqref{OCP} for $x \geq 0$ is called
\begin{itemize}
\vspace{-2mm}
\item {\em reflecting} for fuel levels $c\in\mathcal{I}$ if $(R(c) - \epsilon,c - \epsilon) \in W$ for all sufficiently small $\epsilon > 0$;
\vspace{-7mm}
\item {\em repelling} for fuel levels $c\in\mathcal{I}$, if $(R(c) - \epsilon,c - \epsilon) \in A$ for all sufficiently small $\epsilon > 0$.
\end{itemize}  
\end{Definition}

In Section \ref{sec:y12} we study the one-shot problem \eqref{def-V} in the unsolved parameter range $\la\in (\la^*,\a\d)$, using it to specify new moving boundary functions $F$ and $G$. 
Section \ref{sec:solV} presents the solution to the one-shot problem in the case $\la\in [\la^\dagger,\a\d)$. These new boundaries, which will be used in the candidate control problem solutions of both Sections \ref{sec:cpos} and \ref{sec:newsol}, are studied in Section~\ref{sec:FG}. 
In Section~\ref{Sec:verification2} they are used to construct the candidate value function for the control problem in the open regime $\la\in [\la^\dagger,\a\d)$.

\subsection{Small fuel: Construction of new boundaries} 
\label{sec:y12}

In this section we specify new candidate boundaries $F$ and $G$ using the obstacle geometry studied in Section \ref{sec:h}. 
It can be seen from Figures \ref{fig:Fig1synth}--\ref{fig:Fig1synth2} (and confirmed using Lemma \ref{lem:littleh}) that the function $y \mapsto P_r(y;c)$ in \eqref{def-P_r} of Definition \ref{def:tangents} satisfies
\begin{equation} \label{eq:Pyc}
P_r(y_c;c) > 0 , \quad \text{for all $\la\in(\la^*,\a\d)$ and $c>0$}. 
\end{equation} 
Moreover, from \eqref{def-H} and \eqref{def-G1} we obtain 
\begin{equation*}
H_{r2}(y;c)= \sqrt y \Big( \fla c\Big( c-\frac {\ln y}{\sta} \Big) +  c+ B_0 e^{c \sta}y^{-1/2}\Big),
\end{equation*}
so that the intercept $I(y;c)$ satisfies
\begin{equation*} 
I(y;c)=-\frac{\sqrt y \la c}{(2\a)^{3/2}}\ln y+ O(\sqrt y) \to -\infty, \quad \text{ as } y \to \infty.
\end{equation*} 
From Figures \ref{fig:Fig1synth}--\ref{fig:Fig1synth2} (see also Lemma \ref{lem:Hconvexity}) it is now clear that 
\begin{equation} \label{eq:Pinfty}
\lim_{y \to \infty}P_r(y;c)=-\infty .
\end{equation} 
These properties are used below to analyse the zeroes of $P_r(\cdot\,;c)$ which, for $y > y_c$, characterise lines tangent to both $H_l$ and $H_r$. 
The following result is also used several times to establish a separating line in Figures \ref{fig:Fig1synth}--\ref{fig:Fig1synth2}.

\begin{Lemma} \label{lem:minorant}
Let $1 < y_u < y_w \leq \Psi(f_0)$ with $\Psi(\fotd) \leq y_w$. Then for all $\la \in (\la^*, \a\d)$ and $c>0$ sufficiently small, the tangent $r_{y_u}(\cdot;0)$ to $H_l(\cdot)$ at $y_u$ lies strictly below $H(\cdot;c)$ on $[y_w,\infty)$.
\end{Lemma}

As a key step towards solving \eqref{def-V} we next verify that, for sufficiently small $c>0$, there exists a tangent to both $H_l$ and $H_{r1}$.

\begin{Proposition} \label{existencegeometricy2}
Suppose that $\la \in (\la^*, \a\d)$. 
There exists a unique couple $(\hat{y}_1(c),\hat{y}_2(c))$ with $1 \leq \hat{y}_1(c) < y_c \leq y_v(c) < \hat{y}_2(c) < \hye(c)$ solving the system:
\begin{equation} \label{eq:system-y1y2}
\left\{
\begin{array}{lr}
\pd {H_l}{y}(y_1)= \pd{H_{r1}} y (y_2;c), \\[+4pt]
H_l(y_1) - \pd {H_l} y (y_1)y_1 = H_{r1} (y_2;c) -\pd {H_{r1}} y (y_2;c)y_2 ,
\end{array}
\right. 
\end{equation}
for each fixed $c \in (0, c_1)$, where 
\begin{align}
c_1 &:= \hat{c} \wedge c_m, 
\label{eq:c2def}\\
\hat{c} & := \inf\{c \in (0,\infty): \hyt(c) \geq \hye(c) \} > 0, \label{eq:c2defb}\\
c_m & := \inf\{c \in (0,\infty): \hyo(c) = 1 \} > 0, \label{eq:cm}
\end{align}
and satisfying 
\begin{equation} \label{cor:1/2d0}
\lim_{c \to 0}\hyo(c)=\lim_{c \to 0}\hyt(c)=\Psi\Big(\frac 1 {2\d}\Big) 
\quad \text{and} \quad \hyo(c) \leq \Psi(f_0). 
\end{equation}
Also, $H_l$ is strictly convex at $\hyo(c)$ and $H_{r1}$ is strictly convex at $\hyt(c)$.
\end{Proposition}

Henceforth, for $c\in(0,c_1)$ we define the boundaries $F$ and $G$ in the natural scale by  
\begin{equation} \label{eq:FGdef}
F(c):=\Psi^{-1}(\hyo(c)) \quad \text{and} \quad 
G(c):= \Psi^{-1}(\hyt(c)),
\end{equation} 
where $\hyo(c)$ and $\hyt(c)$ are the tangency points given in Proposition \ref{existencegeometricy2}. 
This is illustrated in Figure~\ref{fig:Fig1synth}.

\subsection{Solution to the one-shot problem $\tild V$ for small fuel} 
\label{sec:solV}

In this section we confirm that, in the case $\la\in[\la^\dagger,\a\d)$, the boundaries $F$ and $G$ constructed in Section~\ref{sec:y12} are sufficient to solve \eqref{def-V}.

\begin{Theorem}
\label{existencegeometric}
If $\la \in [\la^\dagger, \a\d)$ then, for each fixed $c \in (0, c_1)$ (where $c_1$ is defined by \eqref{eq:c2def}), the function $y \mapsto W(y;c)$ given by
\begin{equation} \label{eq:Wdef}
W(y;c) := 
\left\{
\begin{array}{ll}
H_l(y), & 1 \leq y \leq \hyo(c),\\[+3pt]
A(c)y + B(c),
& \hyo(c) < y < \hyt(c), \\[+3pt]
H_{r}(y), & y \geq \hyt(c),\\
\end{array}
\right. 
\end{equation}
where $A(c)=\pd {H_l}{y}(\hyo(c))=\pd {H_{r1}}{y}(\hyt(c);c)$ 
and $B(c)=I(\hyo(c);c)=I(\hyt(c);c)$ (see Definition \ref{def:tangents} and Proposition \ref{existencegeometricy2}), is the greatest non-positive convex minorant of $H(\cdot;c)$ on $[1,\infty)$ and is of class $C^1(1,\infty)$.
The value function $\tild V$ of \eqref{eq:defqc} is then given by  
\begin{equation} \label{eq:VWdef}
\tild V(x,c) = \phi_\alpha(x) W(\Psi(x);c) + \fla x^2 + \flas, \quad (x,c) \in [0,\infty) \times (0, c_1),
\end{equation}
and the optimal stopping time is $\tau^* = \inf\{t \geq 0 : X_t \in [0,F(c)] \cup [G(c),\infty)\}$. 
\end{Theorem}

\begin{proof} 
Fixing $c \in (0,c_1)$, define $\hyo(c)$ and $\hyt(c)$ as in Proposition \ref{existencegeometricy2} and, hence, $A(c)$ and $B(c)$. 
We then know from Lemma \ref{lem:Hconvexity}.(I).(i)--(ii) that the strict convexity of $H_r$ further implies that 
\begin{equation*}
H(y;c)-r_{\hyt(c)}(y;c) > 0 \quad \text{ for all } y>\hyt(c).
\end{equation*}
On the other hand, by the convexity of $H_l$ from Lemma \ref{lem:Hlprops}, we have 
\begin{equation*} 
H_l(y)-r_{\hyt(c)}(y;c) > 0 \quad \text{ for all } y<\hyo(c).
\end{equation*}
We wish to conclude that the doubly tangent straight line $y \mapsto A(c)y + B(c)$ lies below $H(\cdot;c)$ for all $y>1$ and, hence, that the function $W(\cdot;c)$ of \eqref{eq:Wdef} is the greatest non-positive convex minorant of $H(\cdot;c)$ on $[1,\infty)$. For this, it remains only to check that it lies below $H(y;c)$ for $y \in (\hyo(c),\hyt(c))$.

Suppose for a contradiction that the common tangent coincides with $H(\cdot;c)$ when $y=y_a \in (\hyo(c),\hyt(c))$. It is obvious from the geometry of $H_l$ that then $y_a>y_c$; also $H_r$ must be concave at $y_a$. This in turn forces $H(\cdot;c)$ to lie strictly below the common tangent when $y=y_c$, a contradiction.

Applying Proposition \ref{prop:DayKar} we obtain the solution $V$ to \eqref{def-V} and conclude that~\eqref{eq:VWdef} solves~\eqref{eq:defqc}. 
The double tangency equations \eqref{eq:system-y1y2} confirm that $W(\cdot;c)$ is of class $C^1(1,\infty)$.~\hfill $\Box$
\end{proof} 

\vspace{3mm}
For $c\in(0,c_1)$, Theorem \ref{existencegeometric} establishes that the waiting region in the one-shot problem $\tild V$ of \eqref{eq:defqc} has a single connected component $(F(c),G(c))$, where 
$0 \leq F(c) < G(c) \leq f_0+c$.
Then heuristically, in order for the value function $\tild V$ to be an appropriate candidate value function in the control problem \eqref{eq:defV}, the boundary $G$ should be repelling on $(0,c_1)$.

\subsection{Monotonicity}
\label{sec:FG}

In this section we study properties of the moving boundaries $F$ and $G$ for all $\la\in(\la^*,\a\d)$, since they are used both when $\la\in[\la^\dagger,\a\d)$ in Section~\ref{Sec:verification2} (Figure~\ref{fig:boundaries4}), and  when $\la\in(\la^*,\la^\dagger)$ in Section~\ref{sec:newsol} (Figure~\ref{fig:boundaries2}).

We first show that they are continuously differentiable on $(0,c_1)$. This follows from the smoothness of the transformation $\Phi$ of \eqref{def-H} and the following result. 

\begin{Proposition}\label{cor:1/2d}
Suppose that $\la\in(\la^*,\a\d)$. 
The functions $\hyo(c)$, $\hyt(c)$ defined in Proposition \ref{existencegeometricy2} are both of class $C^1(0,c_1)$.
\end{Proposition}

Next we establish the existence of $c_0>0$ such that the functions $F$ and $G$, which are obtained from the tangency points $(\hyo(c), \hyt(c))$ via \eqref{eq:FGdef}, are monotone and continuously differentiable on $(0,c_0)$, with $G'(c)>1$ and $\lim_{c\to 0} F(c) \neq f_0$. 
For consistency with \cite{KOWZ00} we employ the functions $h_1$ to $h_4$, $q$ and $L$, which were defined in the latter paper, and whose definitions are reported again below, as appropriate.

It is straightforward to check from \eqref{eq:Hl} that the straight line tangent to $H_l(\cdot)$ at $y=\Psi(x)$ has gradient
\begin{align}
\label{eq:hidefs1}
h_1(x) &:= \frac{\a \d -\la}{2\a}e^{-x\sta}\rho(x) , 
\quad \text{with $\rho$ defined as in \eqref{eq:defrho}},  
\end{align}
and it intercepts the vertical axis (see e.g.\ Figure~\ref{fig:Fig1synth}) at the value 
\begin{align}\label{eq:hidefs2}
h_2(x) &:= \frac{\a \d -\la}{2\a}e^{x\sta}\Big(\rho(x) - \frac{4x}{\sta} \Big).
\end{align}
Recalling Lemma \ref{lem:Hlprops}, the function $h_1$ restricted to $[0,f_0]$ is non-positive and strictly increasing, and therefore has a well-defined inverse 
\begin{align}\label{eq:h1inv}
    h_1^{-1}:[-\frac{\la}{2\a^2},0]\to[0,f_0].
\end{align}
Moreover, for $y=\Psi(x)$ the transformation $\Phi$ of \eqref{def-H} gives $H_{r1}(y;c)=e^{x\sta} h_{r1}(x;c)$. It thus follows that the line tangent to $H_{r1}(\cdot;c)$ at $y=\Psi(x)$ has gradient $\tild \cH_3(x,c)$, where
\begin{align} \label{eq:th3}
\tild \cH_3(x,c)
&:= \pd {H_{r1}}y (y;c)\Big|_{y=\Psi(x)} 
=\frac{e^{-x\sta}}{2\sta} \Big( \pd {h_{r1}}x (x;c) + \sta {h_{r1}}(x;c)\Big) \\
&= e^{-x\sta} \lb \hd c^2 + c \lb \ha - \delta \Big( \osta + x \Big) \rb + \dmla x \Big( \frac x 2 + \osta \Big) - \fltas \rb \nonumber \\
&= ch_3(x) + e^{-x\sta} \lb \hd c^2 + \dmla \Big( \ha \rho(x) - c \Big( \osta + x \Big) \Big) \rb, 
\nonumber \\
\label{eq:h3}
h_3(x) 
&:= \fla \Big( \atl - \Big( x+\frac 1 \sta \Big) \Big) e^{-x\sta},
\end{align}
while its intercept at the vertical axis is 
\begin{align} \label{eq:th4}
\hspace{-3mm}\tild \cH_4(x,c)
&:= \Big( H_{r1}- y \pd {H_{r1}}y \Big)(y;c)\Big|_{y=\Psi(x)} 
= \frac{e^{x\sta}}{2\sta}\Big( -\pd {h_{r1}}x (x;c) + \sta {h_{r1}}(x;c)\Big) \\
&= e^{x\sta} \lb \hd c^2 + c \lb \ha + \delta \Big(\osta - x \Big) \rb + \dmla x \Big(\frac x 2 - \osta \Big) - \fltas \rb \nonumber\\
&= -ch_4(x) + e^{x\sta} \lb \hd c^2 + \dmla \Big( \ha \rho(x) - \frac{2x}{\sta} + c \Big(\osta - x \Big) \Big) \rb, \nonumber\\ 
\label{eq:h4}
h_4(x) &:= \fla \Big( x - \Big( \atl +\frac 1 \sta \Big) \Big) e^{x\sta}. 
\end{align}
Since $G(c) < f_0+c$ for all $c\in(0,c_1)$ (cf.~Proposition \ref{existencegeometricy2} and \eqref{eq:FGdef}), the common tangent line between $H_l$ and $H_r$ (Theorem \ref{existencegeometric}) is tangent to $H_{r1}$ rather than $H_{r2}$ (the latter was the case in \cite{KOWZ00}, see Remark \ref{rem:noapply}). We may therefore substitute \eqref{eq:hidefs1} and \eqref{eq:th3} into the first line of \eqref{eq:system-y1y2} to give (cf.\ Theorem \ref{existencegeometric}) that
\begin{equation} \label{H3=h1}
A(c) = h_1(F(c)) = \tild \cH_3(G(c),c),
\end{equation}
while substituting \eqref{eq:hidefs2} and \eqref{eq:th4} into the second line of \eqref{eq:system-y1y2} gives 
\begin{equation} \label{H4=h2}
B(c) = h_2(F(c)) = \tild \cH_4(G(c),c). 
\end{equation}
These quantities are respectively the gradient and intercept of the common tangent line, which can be seen in Figures \ref{fig:Fig1synth} and \ref{fig:Fig1synth2}.

Some useful relationships for the forthcoming analysis are obtained by differentiating both \eqref{eq:th3} and \eqref{eq:th4}, which yield
\begin{align} 
\label{H3H4}
e^{x\sta} \, \pd {\tild \cH_3} x (x,c) 
&= - e^{-x\sta} \, \pd {\tild \cH_4} x (x,c), \\
\label{eq:hall}
\begin{split}
\Big( \pd {\tild \cH_3} x + \pd {\tild \cH_3} c \Big)(x,c) &= h_3(x) - \sta \tild \cH_3 (x,c), \\
\Big( \pd {\tild \cH_4} x + \pd {\tild \cH_4} c \Big)(x,c) &=  \sta \tild \cH_4 (x,c)-h_4(x).
\end{split}
\end{align}

A combination of \eqref{cor:1/2d0}
with \eqref{eq:FGdef} implies that $F(c) = \Psi^{-1}(\hyo(c)) \leq f_0$. It then follows from \eqref{eq:h1inv} and \eqref{H3=h1} that for each $c \in (0,c_1)$ 
we have $\tild \cH_3(G(c),c) \in [-\frac{\la}{2\a^2},0]$. 
Thus, the system \eqref{H3=h1}--\eqref{H4=h2} gives
\begin{equation} \label{LG=0}
L(G(c),c) = 0, \quad \text{for $c \in (0,c_1)$},
\end{equation}
where we make the formal definition
\begin{equation} \label{eq:Ldot}
L(x,c):=
\tild \cH_4(x,c) - h_2\circ h_1^{-1} \big(\tild \cH_3(x,c)\big),
\end{equation} 
noting that the value $L(G(c),c)$ is well defined. 

\begin{Remark}\label{rem:noapply}
In the case $\la \in (\la_*,\la^*]$, the system of equations corresponding to \eqref{H3=h1}--\eqref{H4=h2} is \cite[Eq.~(10.8)]{KOWZ00}:
\begin{align} \label{eq:ifuel}
h_1(F(c)) = \cH_3(G(c),c)\quad \text{ and } \quad 
h_2(F(c)) = \cH_4(G(c),c) , 
\end{align}
where $\cH_3$ and $\cH_4$ are defined in \cite[Eq.~(10.9)]{KOWZ00}: 
\begin{equation} \label{(10.9)}
\begin{split}
\cH_3(x,c) &:= c\, h_3(x) + \frac{\lambda}{2\alpha} c^2 \, e^{-x\sqrt{2\alpha}} , \\
\cH_4(x,c) &:= -c\, h_4(x) + \frac{\lambda}{2\alpha} c^2 \, e^{x\sqrt{2\alpha}} 
+ h_2(f_0) \, e^{c\sqrt{2\alpha}}.
\end{split}
\end{equation} 
This system (which differs from \eqref{H3=h1}--\eqref{H4=h2} since $\tild \cH_3$ and $\tild \cH_4$ defined in \eqref{eq:th3} and \eqref{eq:th4} are different from \eqref{(10.9)}) was established via the ansatz $G(c)>\atl + c$, which implies that $G(c)>f_0 + c$ since $F(c) \leq f_0 \leq \fotd < \atl$ when $\la \leq \la^*$ (see \cite[Sec.~10]{KOWZ00} for more details). In the geometric perspective of the current paper, this would correspond to tangency points lying on $H_l$ and $H_{r2}$ (cf.~Figures \ref{fig:Fig1synth}--\ref{fig:Fig1synth2}). The system \eqref{eq:ifuel} is studied in \cite[Prop.~10.2]{KOWZ00} and characterises the repelling boundary $\{G(c):c \in (0,\bar c]\}$ illustrated in Figure~\ref{fig:boundaries1}, via 
\begin{equation} \label{eq:Ldotold}
L_{old}(G(c),c) = 0,  
\quad \text{where} \quad 
L_{old}(x,c) := \cH_4(x,c) - h_2\circ h_1^{-1} \big(\cH_3(x,c)\big) ,
\end{equation} 
which implies in turn that $G$ satisfies the ODE
\begin{align} \label{G'c0}
\hspace{-1cm}
G'(c) = 1 - \frac{q(G(c);F(c))}{\pd{}{x} L_{old}(G(c),c)}, 
\end{align}
and \eqref{eq:ifuel} specifies $F$ as the following transformation of $G$:
\begin{align}\label{eq:Funi}
F(c) = h_1^{-1} \big(\cH_3(G(c),c)\big). 
\end{align}

In contrast, in the present case $\la\in(\la^*, \alpha\delta)$, the tangency points lie on $H_l$ and $H_{r1}$, and we have $G(c) < \atl + c$ for small $c>0$, hence the ansatz of \cite{KOWZ00} does not apply.
\end{Remark}

The function $L$ is used repeatedly in the analysis below, together with the function
\begin{align} 
\label{eq:qdef}
q(x;z) 
&:= \sta \big(h_2(z)-h_1(z)e^{2z\sta}\big) + h_3(x)e^{2z\sta} - h_4(x) \\
&= \fla e^{x\sta} \Big[\Big(\atl -x -\frac 1 \sta \Big)e^{2(z-x)\sta} 
+ \Big(\atl -x +\frac 1 \sta \nonumber \Big) \Big] - 2z\frac{\a\d-\la} \a e^{z\sta}, \nonumber
\end{align}
which was defined in \cite[Eq.~(8.22)]{KOWZ00} for $0 < z \leq x < \infty$, and which satisfies 
\begin{align}
q(z;z)&=e^{z\sta}(1-2\d z), \label{eq:qprop1}\\
\pd{q}{x}(x;z)&= \la \sqrt{\frac 2 \a}\Big( \atl - x \Big)e^{x\sta} \big(1-e^{2(z-x)\sta} \big), \quad z \leq x < \infty. \label{eq:qprop2}
\end{align}

\begin{Lemma} \label{derivL}
On its domain of definition, the function $L(\cdot,\cdot)$, defined formally in \eqref{eq:Ldot} satisfies
\begin{align} 
&\pd {L}{x} (x,c) = \big(e^{2z\sta}-e^{2x\sta} \big) \, \pd {\tild \cH_3} x (x,c), \quad \text{where } z:=h_1^{-1}\big(\tild \cH_3(x,c) \big), \label{eq:lx}\\
&\Big( \pd {L}{x}+\pd {L}{c} \Big)(G(c),c) = q\big(G(c);F(c) \big), \quad \text{where $q$ is defined by \eqref{eq:qdef}.} 
\label{eq:hall3}
\end{align}
\end{Lemma}
Combining the property \eqref{LG=0} of $L$ with Lemma \ref{derivL}, we obtain the following expression for the slope of the moving boundary $G$. 
\begin{Lemma}\label{cor:FGC1}
For all $\la\in(\la^*,\a\d)$ and $c \in (0,c_1)$ we have $\pd{L}{x}(G(c),c)<0$ and 
\begin{equation}\label{eq:Gderiv}
G'(c) = 1 - \frac{q(G(c);F(c))}{\pd{L}{x}(G(c),c)}.
\end{equation}
\end{Lemma}
 Thus information on the derivative of the boundary $G$ can be obtained from the sign of $q(G(c);F(c))$ via Lemma \ref{cor:FGC1},  
and the next result locates the boundary $F$.
\begin{Lemma} \label{lem:Fcsmall}
For all $\la\in(\la^*,\a\d)$ and sufficiently small $c>0$ we have $F(c)<\frac 1 {2\d}$.
\end{Lemma}
Finally, recalling the constant $c_1$ from \eqref{eq:c2def}, we define 
\begin{equation}\label{lem:Gderv}
c_0:= c_1 \wedge \ol{c} 
\quad \text{and} \quad 
\ol{c}:= \inf \{c \in (0,c_1) : G'(c) \leq 1\}, 
\end{equation} 
and establish bounds on the derivatives of the moving boundaries for all $c\in(0,c_0)$.

\begin{Proposition} \label{prop:monotonicity}
For all $\la\in(\la^*,\a\d)$, we have $c_0>0$ and the moving boundaries $F$ and $G$ satisfy $F'(c)<0$ and $G'(c)>1$, respectively, for all $c \in (0,c_0)$.
\end{Proposition}
The next result simplifies the definition \eqref{lem:Gderv} of $c_0$.

\begin{Corollary} \label{cor:c1}
For all $\la\in(\la^*,\a\d)$, the constant $c_0$ defined by \eqref{lem:Gderv} takes the form
$$
c_0 = \bar c \wedge \hat{c} 
=\bar c \wedge \inf\{c \in (0,\infty): G(c) \geq f_0+c \}.
$$
\end{Corollary}

\subsection{A candidate value function for the control problem $Q$} 
\label{Sec:verification2}

In this section we construct a candidate value function $(x,c) \mapsto \tild{Q}(x, c)$ for $(x,c)\in[0,\infty)^2$ and $\la\in [\la^\dagger,\a\d)$. Considering first the case $c=0$, Theorem \ref{cor:czerosol} and \eqref{eq:defq0} imply that the value of the absorbing boundary when $c=0$ is $F(0) := f_0$ and the value function is the function $x \mapsto Q(x,0)$ from \eqref{eq:Vtild0}.
Thus it remains to consider fuel levels $c>0$.

\subsubsection{Small fuel: $c \in (0,c_0]$} 
\label{Sec:verification2a}

Our candidate value function $\tild{Q}(x;c)$ for all $(x,c)\in [0,\infty) \times (0,c_0]$ will be the value function of the one-shot problem $\tild V(\cdot,c)$ of \eqref{eq:defqc}, which is an optimal stopping problem parametrised by the initial fuel level $c$. 
Thus from Theorem \ref{existencegeometric} and the fact that $c_0 \leq c_1$ (definition \eqref{lem:Gderv}), for all $c \in (0,c_0]$ we have
\begin{equation} \label{eq:tildQ}
\tild{Q}(x;c) := 
\left\{
\begin{array}{ll}
\delta x^2, & 0 \leq x \leq F(c),\\[+3pt]
{h_1(F(c))} e^{x\sta} + {h_2(F(c))} e^{-x\sta} + \fla x^2 + \flas, & F(c) < x < G(c), \\[+3pt]
\tild{V}_0(x-c) + c, & x \geq G(c).\\
\end{array}
\right. 
\end{equation}

Recalling Definition \ref{def:rr}, this candidate is reasonable since Proposition~\ref{prop:monotonicity} and Corollary~\ref{cor:c1}
establish that the boundary $G$ is well-defined and repelling for all $c\in (0,c_0]$. 
We have from \eqref{cor:1/2d0}--\eqref{eq:FGdef} that $\lim_{c\to 0}F(c) = \fotd$. However, as noted at the start of Section~\ref{Sec:verification2}, the value of the absorbing boundary when $c=0$ is $F(0) = f_0$. 
Since $\fotd < f_0$ (cf.~\eqref{eq:longineq1}), the absorbing boundary is therefore discontinuous at $c=0$.
Next we extend this candidate to all fuel levels, proceeding in two stages (intermediate and large fuel). 

\subsubsection{Intermediate fuel: $c \in (c_0, \ol{c}]$} 
\label{Sec:verification2b}

This fuel regime appears only
when $c_0 = \hat{c} < \ol{c}$ (with $\bar c$ defined in \eqref{lem:Gderv}).
In this case, we have from Section~\ref{Sec:verification2a} that $G(c_0) = f_0+c_0$ (Corollary \ref{cor:c1}) and $G'(c_0-)>1$. 

Our candidate value function for $c \in (c_0, \ol{c}]$ is
\begin{equation} \label{eq:tildQc0}
\tild{Q}(x;c) := \left\{
\begin{array}{ll}
\delta x^2, & 0 \leq x \leq F(c),\\[+3pt]
{h_1(F(c))} e^{x\sta} + {h_2(F(c))} e^{-x\sta} + \fla x^2 + \flas, & F(c) < x < G(c), \\[+3pt]
\tild{V}_0(x-c) + c, & x \geq G(c),
\end{array}
\right. 
\end{equation}
where the moving boundaries $F$ and $G$ are now specified by the equations~\eqref{eq:ifuel} of Remark~\ref{rem:noapply}.
For all $c \in (c_0,\bar c]$ the moving boundary $G$ is thus uniquely characterised by the boundary condition $G(c_0) = f_0 + c_0$ and the ODE \eqref{G'c0} for $c \in (c_0, \bar c]$. This extension of $G$ is continuous at $c=c_0$. 
The moving boundary $F$, which is then specified uniquely by \eqref{eq:Funi} for $c \in (c_0, \ol{c}]$, is also continuous at $c=c_0$; 
this follows since $F(c_0)= h_1^{-1} \big(\cH_3(f_0+c_0,c_0)\big)= h_1^{-1} \big(\tild \cH_3(f_0+c_0,c_0)\big))$ 
by Lemma~\ref{lem:littleh}, \eqref{eq:th3}, \eqref{(10.9)} and Remark \ref{rem:noapply}.
Note also that in the case $\la\in [\la^\dagger,\a\d)$ we have
\begin{itemize}
\vspace{-1mm}
\item[(i)] $q(G(c_0),F(c_0)) > 0$, which follows from Lemma~\ref{cor:FGC1} and the fact that $G'(c_0-)>1$;
\vspace{-2mm}
\item[(ii)] $G(c_0) \geq \atl+ \frac{c_0}{2}$, since \eqref{eq:longineq1} guarantees that $\atl + c_0 \leq f_0 +c_0$.
\end{itemize}
These properties ensure that the proof of \cite[Prop.~10.2]{KOWZ00}, whose results are summarised for convenience in Remark~\ref{rem:noapply}, remains valid in the present regime $\la\in [\la^\dagger,\a\d)$. 
\subsubsection{Large fuel: $c \in (\ol{c}, \infty)$}
\label{Sec:verification2c}

In both cases (that is, when $c_0 = \hat{c}< \ol{c}$ or when $c_0 = \ol{c}$), it follows from the regularity properties of $G$ obtained above that $G'(\ol{c}) = 1$ 
and $q(G(\ol{c}), F(\ol{c})) = 0$ (cf.~\eqref{G'c0} and~\eqref{eq:Gderiv}). Our candidate value function for $c \in (\ol{c}, \infty)$ is
\begin{equation} \label{eq:tildQolc}
\tild{Q}(x;c) := 
\left\{
\begin{array}{ll}
\delta x^2, & 0 \leq x \leq F(c),\\[+3pt]
{h_1(F(c))} e^{x\sta} + {h_2(F(c))} e^{-x\sta} + \fla x^2 + \flas, & F(c) < x < G(c), \\[+3pt]
\tild Q(x-\zeta, c-\zeta) + \zeta, & G(c) < x < G(\bar c) + (c- \bar c),
\\
\tild{V}_0(x-c) + c, & x \geq G(\bar c) + (c- \bar c),
\end{array}
\right. 
\end{equation}
where $\zeta = \zeta(x,c)$ is defined uniquely by $x-\zeta = G(c-\zeta)$ 
and $F$ solves the ODE of \cite[Eq.~(10.16)]{KOWZ00}, namely
\begin{align} \label{F'olc}
\begin{split}
&\begin{cases}
F'(c) = \frac{h_3(\chi(F(c))) - \sta h_1(F(c))}{h_1'(F(c))}, \qquad c > \ol{c}, \\
F(\ol{c}) = h_1^{-1} \big(\cH_3(G(\ol{c}),\ol{c})\big) 1_{\{c_0 = \hat{c}\}} + \Psi^{-1}(\hyo(\ol{c})) 1_{\{c_0 = \ol{c}\}},
\end{cases} \\
&\text{ and} 
\quad G(c) = \chi(F(c)), \text{ for all } c > \bar c,
\end{split}
\end{align}
where $\chi:[0,f_0] \to (\atl,\infty)$ is defined implicitly via $q(\chi(z),z)=0$, giving also
\begin{align} \label{eq:longineq}
F(c) < \frac1{2\d} < \atl < G(c) < \atl + \frac1{\sqrt{2\a}}, \qquad c\in [\ol{c}, \infty),
\end{align}
thanks to \cite[Lem.~8.2]{KOWZ00}, whose proof applies also in this case since $F(\ol{c}) < \fotd$.  (When $c_0 = \ol{c}$, the inequality $F(\ol{c}) < \fotd$ follows from Lemma \ref{lem:Fcsmall} and Proposition~\ref{prop:monotonicity}; when $c_0 = \hat{c}$, it follows from the construction of $F$ in Sections \ref{Sec:verification2a}--\ref{Sec:verification2b}).
The continuity of $F$ and $G$ at $c=\bar c$ follows by construction. 
We conclude by observing that the argument that $G'(c)<1$ for $c \in (\ol{c},\infty)$ found in the proof of \cite[Thm.~10.6]{KOWZ00} does not depend on the particular value of $\lambda$, 
so we may apply it in the present regime to conclude that $G$ is reflecting on $[\ol{c}, \infty)$ (cf.~Definition~\ref{def:rr}).

\section{Analysis of the case $\la\in (\la^*, \la^\dagger)$}
\label{sec:newsol}

We begin this section by solving the one-shot problem \eqref{def-V} in the case $\la\in (\la^*, \la^\dagger)$. This motivates the following ansatz, which we use below to develop a candidate solution to the control problem. The method contrasts with Section~\ref{sec:cpos}, where the one-shot solution itself was the candidate solution to the small-fuel control problem. Proofs of results in this section which are not included below can be found in Appendix \ref{AppA5}.

\begin{anz}\label{bigass}
For $\la\in (\la^*, \la^\dagger)$, the strategy has four continuous moving boundaries, namely $F,G:(0,\infty) \mapsto (0,\infty)$ and $\Fb,\Gb: \mathcal{I} \mapsto (0,\infty)$
where $\mathcal{I} = (0,c_\mathcal{I}]$ for some $c_\mathcal{I} > 0$, 
satisfying 
$G(c) < \Fb(c) < \Gb(c)$ for all $c \in (0,c_\mathcal{I})$ and
\begin{equation} \label{olF00}
\lim_{c\to 0} F(c) = 
\lim_{c\to 0} G(c) = \frac1{2\d} <  
\lim_{c\to 0} \Fb(c) < \lim_{c\to 0} \Gb(c). 
\end{equation} 
Note that, for all $C_t>0$, such a strategy prescribes discretionary stopping for $X_t \in [0,F(C_t)]$, waiting for $X_t \in (F(C_t), G(C_t)) \cup (\Fb(C_t),\Gb(C_t))$, and action for all other pairs $(X_t,C_t)$.
\end{anz}

\begin{Remark} \label{rem:Ansatz}
In Ansatz \ref{bigass}, the interval $(\Fb(C_t),\Gb(C_t))$ is understood to be empty for $C_t \geq c_\mathcal{I}$. Also, as noted in Section~\ref{Sec:verification2}, when $C_t=0$ the absorbing boundary is $F(0) := f_0$, so it is optimal to stop when $X_t \in [0,F(0)]$ and wait when $X_t \in (F(0), \infty)$.
The continuity of $F$ and $G$ on $(0,\infty)$ and their limiting values as $c \to 0$ were established in Section~\ref{Sec:verification2} (showing also that $F$ is discontinuous at $c=0$). The remaining parts of Ansatz \ref{bigass} are verified in Remark \ref{rem:a1ver}, after the construction of the candidate value function.
\end{Remark}

The motivation for Ansatz \ref{bigass} is provided in Section~\ref{sec:sf}, which suggests the existence of an additional connected component of the waiting region, bordered by two new boundaries $(\Fb,\Gb)$. 
As shall be seen, the reason for the failure of the approach of Section~\ref{sec:cpos} (which had only absorbing and repelling boundaries) in the present regime $\la\in (\la^*, \la^\dagger)$ is that the new boundary $\Gb$ will be reflecting, no matter how small the fuel level. 
Interestingly, although the one-shot problem $\tild V$ of \eqref{eq:defqc} (equivalently, the optimal stopping problem $V$ of \eqref{def-V}) is no longer a direct solution to the control problem, nevertheless it will play an important indirect role in the novel methodological step performed in Section~\ref{sec:sf}. 
We therefore also provide its solution in the present regime.

\begin{Proposition} \label{existencegeometricy4}
Suppose that $\la \in (\la^*, \la^\dagger)$. Then recalling $\hat{y}_2(c)$ from Proposition \ref{existencegeometricy2}, for sufficiently small $c>0$ there exists a unique couple $(\hat{y}_3(c),\hat{y}_4(c))$ with $y_c < \hat{y}_2(c) < \hyoo(c) < \hye(c) \leq \hyc(c) < \hytt(c)$ that solves the system:
\begin{equation*} 
\left\{
\begin{array}{lr}
 \pd{H_{r1}} y (y_3;c) = \pd{H_{r2}} y (y_4;c), \\[+4pt]
 H_{r1} (y_3;c) - \pd {H_{r1}} y (y_3;c) y_3 = 
 H_{r2} (y_4;c) - \pd {H_{r2}} y (y_4;c) y_4,
\end{array}
\right. 
\end{equation*}
and satisfies $\lim_{c \to 0}\hyoo(c)= \Psi(f_0)$. Thus the \gcm \, has the geometry of Figure~\ref{fig:Fig1synth2}.
\end{Proposition}

The optimal stopping problem $V$ of \eqref{def-V} now differs qualitatively from that in the regime $\la\in [\la^\dagger,\a\d)$, as can be seen by comparing Figures \ref{fig:Fig1synth} and \ref{fig:Fig1synth2}. 
More precisely, for all $c<\uc$ the obstacle $H_{r2}(\cdot\,;c)$ now has an additional concave region $(\hye(c), \hyc(c))$ (see Lemma \ref{lem:Hconvexity}.(II).(iii)). 
This additional concave region does not fall inside the waiting region $(\hyo(c),\hyt(c))$ in \eqref{eq:Wdef} since, for $c\in(0,c_1)$, we have $\hyt(c) \in (y_c, \hye(c))$ (Proposition \ref{existencegeometricy2}). Thus in view of Proposition \ref{prop:DayKar}, the solution presented in Theorem \ref{existencegeometric} cannot hold. Nevertheless, the one-shot problem will be used in Section~\ref{sec:sf} to obtain bounds on the limiting values of $\Fb(c)$ and $\Gb(c)$ as $c\to 0$ in equation \eqref{olF00} of Ansatz \ref{bigass}.

\begin{Remark} \label{rem:l<la*} 
Note that this concave region $(\hye(c), \hyc(c))$ also exists in the regime $\la \in (\la_*,\la^*]$ with $c \in [0,\bar c]$. 
However, as shown in \cite{KOWZ00}, it then lies {\em inside} the region $(\hyo(c),\hyt(c))$ so that the form of the value function remains consistent with that of Section~\ref{Sec:verification2a}, with moving boundaries characterised as in Remark \ref{rem:noapply}.
\end{Remark}

\subsection{A new waiting region}\label{sec:sf}

Our aim in this section is to motivate the presence of boundaries $\Fb$ and $\Gb$ in Ansatz \ref{bigass}.
For $x, \epsilon > 0$ define the punctured neighbourhood $N_x(\epsilon)$ by
\[
N_x(\epsilon) := \{(u,c) \in \R \times (0,\infty): 0 < |(x,0)-(u,c)| < \epsilon\}.
\]
The next result establishes a useful relationship between the waiting region of the one-shot problem and that of the control problem.

\begin{Proposition} \label{prop:waitbel}
Recalling Ansatz \ref{bigass}, suppose that the point $(x,0)$ is not the limit as $c \to 0$ of any boundary of the solution to the control problem.
Let $\e>0$ be small enough that the neighbourhood $N_x(\epsilon)$ is contained in a single region of the solution to the control problem (waiting, action or stopping). Further, suppose that $N_x(\epsilon)$ is contained in the waiting region of the one-shot problem $\tild V$ of \eqref{eq:defqc}. 
Then $N_x(\epsilon)$ is contained in the waiting region of $Q$ from \eqref{eq:defV}. 
\end{Proposition}

\begin{proof}
Let $c \in (0,\epsilon)$ and consider the following two cases:

(a) {\it $N_x(\epsilon)$ is contained in the stopping region of $Q$.} 
In this case we have $Q(x,c) = \delta x^2$. 
Since $Q \leq \tild V$ (suboptimality), combining this equality with the definition \eqref{eq:defqc} gives 
$$
Q(x,c) \leq \tild V(x,c) \leq \delta x^2.
$$ 
We conclude that $\tild V(x,c) = \delta x^2$, so that $(x,c)$ belongs to the stopping region of $\tild V$. This is a contradiction, as $N_x(\epsilon)$ is contained in the waiting region of $\tild V$.

(b) {\it $N_x(\epsilon)$ is contained in the action region of $Q$.} 
In this case we have $Q(x,c) = \tild V(x-c,0)+c$. Combining this with \eqref{eq:defqc} gives 
$$
Q(x,c) \leq \tild V(x,c) \leq \tild V(x-c,0)+c.
$$ 
We conclude that $\tild V(x,c) = \tild V(x-c,0)+c$, so that $(x,c)$  belongs to the stopping region of $\tild V$. As above, this is a contradiction.

We conclude that $(x,c)$, and hence the whole neighbourhood $N_x(\epsilon)$, is contained in the waiting region of $Q$.
\hfill $\Box$
\end{proof} 

\vspace{3mm}
The following corollary then motivates Ansatz \ref{bigass}, and also provides bounds on the limiting values of $\Fb(c)$ and $\Gb(c)$ as $c\to 0$. 
\begin{Corollary}\label{cor:oset}
The waiting region of the control problem $Q$ contains an open neighbourhood of the set $(f_0, \atl) \times \{0\}$. 
In particular, 
$\lim_{c \to 0}\Fb(c) \leq f_0$ and $\lim_{c \to 0}\Gb(c) \geq \atl.$
\end{Corollary}

\begin{proof}
It follows from Proposition \ref{existencegeometricy4} and \eqref{ys} that the region $\{(x,c): c \in (0,k), \; x \in [f_0+c, \atl + \frac c2]\}$ belongs to the waiting region of $\tild V$. 
Hence for $x \in (f_0, \atl)$ and $\epsilon>0$ sufficiently small, the neighbourhood $N_x(\epsilon)$ belongs to the waiting region of $\tild V$.

Then, by Proposition \ref{prop:waitbel}, the neighbourhood $N_x(\epsilon)$ is contained in the waiting region of the control problem $Q$, and taking the union of these sets establishes the first claim of the corollary. This implies the second claim.
\hfill $\Box$
\end{proof}

\subsection{Additional boundaries $\Fb$, $\Gb$}
\label{sec:topo}

In this section we write down a system of equations for a repelling boundary $\Fb$ and reflecting boundary $\Gb$, which will feature in addition to the boundaries $F$ and $G$ obtained above, and obtain bounds useful for their solution.

In the waiting region, the value function $Q$ is expected to solve the Feynman-Kac equation 
\begin{align}\label{eq:F-K}
(\mathcal{L} - \a) Q(x,c) + \la x^2 = 0,
\end{align}
thus we aim to construct a candidate $\tild Q$ solving \eqref{eq:F-K}, that is,
\begin{equation} \label{eq:VstarRR}
\tild{Q}(x,c) = \tild A(c) e^{x\sta} + \tild B(c) e^{-x\sta} + \fla x^2 + \flas ,
\quad x \in (\bar F(c), \bar G(c)), 
\end{equation}
for differentiable functions $\tild A$ and $\tild B$. 
Since each unit of fuel instantaneously costs one unit (cf.~\eqref{OCP}) and also instantaneously moves the process towards the origin by one unit, the action region should be characterised by the equation $U(x,c) = 1$, where 
\begin{equation}\label{eq:UUdef}
U(x,c):=\Big( \pd{\tild Q}{x}  + \pd{\tild Q}{c} \Big)(x,c) .
\end{equation}

\begin{Remark}\label{rem:sf}
In the sequel we say that a function has {\rm smooth fit} in some neighbourhood of a boundary if it is continuously differentiable in that neighbourhood. 
\end{Remark}

We will construct a candidate value function $\tild Q$ with, in addition to the boundaries $F$ and $G$ introduced and studied in Section~\ref{sec:cpos}, a repelling boundary $\Fb$ and reflecting boundary $\Gb$ satisfying the following smooth-fit conditions between the waiting and action regions:
\begin{enumerate}
\vspace{-2mm}
\item[(i)] smooth fit in the $x$ variable for $\tild Q$ across repelling free boundary points $\ol{F}(c) \in (\frac1{2\d}+\frac c2, f_0+c)$, so that $\tild Q(\Fb(c),c) = \tild V_0(\Fb(c) - c) + c$ and $\pd{\tild Q}{x}(\Fb(c),c) = \tild V_0'(\Fb(c) - c)$;

\vspace{-2mm}
\item[(ii)] smooth fit in the $x$ variable for $U$ across reflecting free boundary points $\ol{G}(c) \in (\atl, \infty)$, so that $U(\Gb(c),c)=1$ and $\pd{U}{x}(\Gb(c),c)=0$.
\end{enumerate}
\vspace{-1mm}
It then follows from conditions (i)--(ii) and straightforward manipulations that (cf.~\eqref{eq:th3} and \eqref{eq:th4})  
\begin{equation}\label{tildeAB}
\tild A(c) = \tild \cH_3(\ol{F}(c),c)
\quad \text{and} \quad 
\tild B(c) = \tild \cH_4(\ol{F}(c),c),
\end{equation}
and that these coefficient functions satisfy the ordinary differential equations 
\begin{equation}\label{tildeA'B'}
\tild A'(c) + \sqrt{2\a} \tild A(c) = h_3(\ol{G}(c))
\quad \text{and} \quad 
\tild B'(c) - \sqrt{2\a} \tild B(c) = - h_4(\ol{G}(c)) , 
\quad c > 0. 
\end{equation}
These four equations will characterise the unknowns $\tild A(c)$ and $\tild B(c)$ of \eqref{eq:VstarRR} as well as $\Fb(c)$ and $\Gb(c)$. We will establish the existence of a unique pair $(\Fb,\Gb)$ of moving boundaries satisfying this characterisation, provide their explicit construction, and prove that they satisfy sufficient range constraints (see Propositions~\ref{prop:olF'}--\ref{prop:olG'}). 
This will ensure both the feasibility of Ansatz \ref{bigass} and the regularity needed for the  verification arguments of Section~\ref{Sec:verification}.

\begin{Remark} \label{rem:AB}
A similar derivation was performed in \cite[Eqs.~(8.14)--(8.17)]{KOWZ00}, but for different coefficient functions $A$ and $B$, which were respectively equal to $\cH_3(G(\cdot),\cdot)$ and $\cH_4(G(\cdot),\cdot)$ defined by \eqref{(10.9)}.
\end{Remark}

The next result identifies a necessary and sufficient condition for the continuity of the candidate value function.

\begin{Proposition} \label{prop:newbdrs} 
Under Ansatz \ref{bigass}, the candidate $\tild Q$ constructed in \eqref{eq:F-K}--\eqref{tildeA'B'} is continuous as $c \to 0$ if and only if the boundary $\Fb$ satisfies 
$\lim_{c \to 0}\Fb(c)= f_0.$
\end{Proposition}

We show below that the new component $(\Fb, \Gb)$ of the waiting region does not communicate with the existing component $(F,G)$. 
That is, if the process starts inside the region $(\Fb, \Gb)$ then at the right boundary $\Gb$ it is reflected back into this region while, at the left boundary $\Fb$, all fuel is expended. Either way, the process never subsequently enters the region $(F, G)$. 
Conversely, however, the existing component $(F,G)$ of the waiting region may or may not communicate with the new component $(\Fb, \Gb)$, and both possibilities are addressed in Section~\ref{Sec:verification3bb} below.

Fixing $c\in(0,K)$, the function $H_{r1}(\cdot;c)$ is strictly convex on $(y_c, \hye(c))$ (Lemma \ref{lem:Hconvexity}.(II).(ii)--(iii)). It follows from the definition \eqref{eq:th3} of $\tild \cH_3$ that 
\begin{equation} \label{dxH3}
\pd{\tild \cH_3}{x}(x,c) = \Psi'(x) \,
\pd {^2 H_{r1}}{y^2}(\Psi(x);c) > 0, \quad  \Psi(x) \in (y_c, \hye(c)),
\end{equation}
i.e.~$\tild \cH_3(\cdot,c)$ is invertible on $(\frac1{2\d}+\frac c2, f_0+c)$. 
Then since $\ol{F}(c) \in (\frac1{2\d}+\frac c2, f_0+c)$ it follows that there can be at most one value $\ol{F}(c)$ satisfying
\eqref{tildeAB}.
Substituting \eqref{tildeAB} into \eqref{tildeA'B'} (and appealing again to \eqref{dxH3}) yields
\begin{align*}
&\pd{\tild \cH_3}{x}(\ol{F}(c),c) \Big(\sqrt{2\a} \tild \cH_4(\ol{F}(c),c) - \pd{\tild \cH_4}{c}(\ol{F}(c),c) - h_4(\ol{G}(c)) \Big) \\ 
&+\pd{\tild \cH_4}{x}(\ol{F}(c),c) \Big(\sqrt{2\a} \tild \cH_3(\ol{F}(c),c) + \pd{\tild \cH_3}{c}(\ol{F}(c),c) - h_3(\ol{G}(c)) \Big) = 0.
\end{align*}
Substituting the identities \eqref{H3H4}--\eqref{eq:hall} into the above equation (and recalling \eqref{eq:qdef}), we have
\begin{align} \label{eq:tildeqdef}
\tild q(\ol{G}(c); \ol{F}(c)) = 0,  
\quad \text{where} \quad 
\tild q(x;z) := q(x;z) - q(z;z) , \quad z \leq x < \infty.
\end{align}

\begin{Lemma} \label{lem:soltildeq}
For each $z \in (\frac1{2\d}, \atl)$, the function $x \mapsto \tild q(x,z)$ of \eqref{eq:tildeqdef} has a unique root $\mathcal{X}(z) \in (z,\infty)$. 
The function $z\mapsto \mathcal{X}(z)$ is of class $C^1(\frac1{2\d}, \atl)$ and satisfies
\begin{align} \label{eq:2ub}
\atl < \mathcal{X}(z) < \min\Big\{ \atl+\frac1{\sqrt{2\a}} , \frac{\a}{\la}-z \Big\}
\qquad \text{and} \qquad 
\pd{\tild q}{x}(\mathcal{X}(z); z) < 0 .
\end{align} 
Moreover, for every $z \in [\atl, \infty)$, the unique solution to $\tild q(x; z) = 0$ is $x = z$.
\end{Lemma}

Recall from \eqref{eq:longineq2} that $\frac 1 {2\d} < \atl$.
Then for each $\Fb(c) \in (\frac 1 {2\d}, \atl)$ we set $\ol{G}(c) := \mathcal{X}(\ol{F}(c))$, obtaining from Lemma \ref{lem:soltildeq} that 
\begin{equation} \label{olF<olG}
\frac 1 {2\d} < \; \ol{F}(c) < \atl < \ol{G}(c) := \mathcal{X}(\ol{F}(c)) < \min\Big\{ \atl+\frac1{\sqrt{2\a}} , \frac{\a}{\la} - \ol{F}(c) \Big\}.
\end{equation}
From the continuity of $\Fb$ and $\Gb$ (Ansatz \ref{bigass}) and in light of Proposition \ref{prop:newbdrs} we define 
$g_0 := \mathcal{X}(f_0)$, so that
\begin{equation} \label{tildeg0}
\frac1{2\d} < \lim_{c\to 0} \Fb(c) =  f_0 < \atl < \lim_{c\to 0} \Gb(c) = g_0 := \mathcal{X}(f_0) < \Big(\atl+\frac1{\sqrt{2\a}}\Big) \wedge \Big(\frac{\a}{\la}-f_0 \Big).
\end{equation}
In view of \eqref{olF<olG}, the fuel level $c_\mathcal{I}$ of Ansatz \ref{bigass} may be written as 
\begin{align} \label{cf}
c_\mathcal{I} &= \inf \Big\{c \in [0,\tild c] : \ol{F}(c) = \atl \Big\} \wedge \tild c, 
\\
\label{tildc}
\quad &\text{where} \quad \tild c := \inf\{ c \in [0, \infty) 
\,:\, G(c) = \ol{F}(c)\} > 0.
\end{align}
\begin{Remark}\label{rem:Gbar}
It will be proved in Section~\ref{Sec:cf} that the boundaries $G$ and $\ol{F}$ do not intersect, so that $\tild{c}$ defined by \eqref{tildc} satisfies $\tild c = \infty$. 
\end{Remark}

We have the following dichotomy:
\begin{equation} \label{cftildc}
\text{either } c_\mathcal{I} = \tild c \quad \text{or} \quad c_\mathcal{I} < \tild c = \infty, 
\end{equation}
It will also be convenient to define the fuel level $\ol{k}$ by 
\begin{equation} \label{olk}
\ol{k}:= \atl - \frac1{2\d} \in (0, K), 
\end{equation}
where $\ol{k} < K$ can be proved to hold true for all $\la\in(\la^*,\la^\dagger)$ via straightforward, lengthy calculations involving the definitions \eqref{lambda*} of $\la^*$ and \eqref{Kk} of $K$. 

\begin{Lemma} \label{cf<infinite}
Recall the definitions \eqref{lem:Gderv} and \eqref{olk} of $\ol{c}$ and $\ol{k}$, respectively. There exists a unique fuel level $c_g$ such that the moving boundary $G$ of \eqref{eq:FGdef} satisfies
\begin{align*}
\frac1{2\d} + c < G(c) < \atl \quad \text{for all } c \in (0,c_g); 
\quad G(c_g) = \atl; \quad \text{and} \quad 
c_g \in (0,\ol{c} \wedge \ol{k}).
\end{align*}
Then, the fuel level $c_\mathcal{I}$ of Ansatz \ref{bigass} is finite and 
$$
c_\mathcal{I} \leq c_g < \ol{c} \wedge \ol{k}.
$$
Further, when $c_\mathcal{I} < \tild{c} = \infty$ in the dichotomy \eqref{cftildc}, we have 
$\ol{F}(c_\mathcal{I}) = \atl = \ol{G}(c_\mathcal{I})$.
\end{Lemma}

Next, we combine the finiteness of the fuel level $c_\mathcal{I}$ in Lemma \ref{cf<infinite} with all the previous results to obtain useful bounds for the moving boundaries. 

\begin{Lemma} \label{lem:olF>12d}
For any $\la \in (\la^*,\la^\dagger)$, the boundaries $\ol{F}$ and $\ol{G}:=\mathcal{X}(\ol F)$ obtained via Lemma \ref{lem:soltildeq} satisfy 
\begin{equation*} 
\ol{F}(c) \in \Big(\frac1{2\d}+c, \atl \Big] 
\quad \text{and} \quad 
\ol{G}(c) \in \Big[\atl, \min\Big\{\atl+\frac1{\sqrt{2\a}}, \frac{\a}{\la} - \ol{F}(c) \Big\}\Big), 
\quad \text{for } c\in (0,c_\mathcal{I}].
\end{equation*}
\end{Lemma}

\subsection{Monotonicity} 

In this section we obtain monotonicity and regularity properties of the boundaries $\Fb$ and $\Gb$. We begin with the following useful technical result.

\begin{Lemma} \label{lem:dtildeqdz}
For every $z \in (\frac1{2\d}, \atl)$, the functions $h_3$, $\tild q$ and $\mathcal{X}$ (cf.\ \eqref{eq:h3}, \eqref{eq:tildeqdef} and Lemma \ref{lem:soltildeq} respectively) satisfy
$$
h_3(\mathcal{X}(z)) < h_3(z) 
\qquad \text{and} \qquad 
\pd{\tild q}{z}(\mathcal{X}(z); z) < 0.
$$
\end{Lemma}

Differentiating \eqref{tildeAB} to compute $\tild A'(c)$ and substituting into the ODE \eqref{tildeA'B'}, the boundary $\Fb$ satisfies the following ODE:
\begin{equation} \label{F'}
\begin{cases}
&\mathfrak{F}'(c) = \dfrac{h_3(\mathcal{X}(\mathfrak{F}(c))) - \sqrt{2\a} \tild \cH_3(\mathfrak{F}(c),c) - \pd{\tild \cH_3}{c}(\mathfrak{F}(c),c)}{\pd{\tild \cH_3}{x}(\mathfrak{F}(c),c)}, \\
&\text{with} \quad \mathfrak{F}(0) = f_0, 
\end{cases}
\end{equation}
where the boundary condition comes from Proposition \ref{prop:newbdrs}. 
Thus $\Fb$ is uniquely specified by its solution,
\begin{align*}
\Fb(c) = \mathfrak{F}(c), \qquad c \in (0,c_\mathcal{I}].
\end{align*}

\begin{Proposition} \label{cor:xf}
The functions $\Fb$ and $\Gb$ are both of class $C^1(0,c_\mathcal{I})$ and $\Fb$ satisfies $\mathcal{X}'(\ol{F}(c)) < 0$ for all $c\in(0,c_\mathcal{I})$. 
\end{Proposition} 

The next result sharpens the bounds on the moving boundary $\Fb$ which were obtained in Lemma \ref{lem:olF>12d}, and establishes bounds on its derivative for all $c\in (0, c_\mathcal{I})$.

\begin{Proposition} \label{prop:olF'}
We have $\ol{F}'(c) \in (0,1)$ for all $c\in(0,c_\mathcal{I})$.  
If also $c_\mathcal{I} < \tild{c} = \infty$ in the dichotomy \eqref{cftildc}, then $\ol{F}'(c_\mathcal{I})=1$. 
We further have  
\begin{equation*} 
\ol{F}(c) \in \Big(\frac1{2\d}+c, \atl \wedge (f_0+c)\Big) 
\quad \text{for } c\in (0,c_\mathcal{I}), \;\; \text{and} \;\; 
\begin{cases}
\ol F(c_\mathcal{I})=\atl, &\text{if } c_\mathcal{I} < \tild{c} = \infty, \\
\ol F(c_\mathcal{I}) = G(c_\mathcal{I}) \leq 
\atl, &\text{if } c_\mathcal{I} = \tild{c}. \\
\end{cases}
\end{equation*}
\end{Proposition}
The following result completes our study of the boundary derivatives, by sharpening the bounds on the moving boundary $\Gb$ which were obtained in Lemma \ref{lem:olF>12d} and establishing bounds on its derivative for all $c\in (0, c_\mathcal{I})$.

\begin{Proposition} \label{prop:olG'}
We have $\ol{G}'(c) <0$ for all $c\in(0,c_\mathcal{I})$, and 
\begin{equation*} 
\ol{G}(c) \in \Big(\atl,  g_0\Big) 
\quad \text{for } c\in (0,c_\mathcal{I}), \quad \text{and} \quad 
\begin{cases}
\ol G(c_\mathcal{I})=\atl, &\text{if } c_\mathcal{I} < \tild{c} = \infty; \\
\ol G(c_\mathcal{I}) \geq  
\atl, &\text{if } c_\mathcal{I} = \tild{c}. \\
\end{cases}
\end{equation*}
\end{Proposition}
Overall, the results obtained in Propositions \ref{prop:olF'} and \ref{prop:olG'} show that the moving boundaries $\ol{F}$ and $\ol{G}$ from \eqref{olF<olG} have derivatives of opposite signs for all fuel levels $c\in(0,c_\mathcal{I})$, where $c_\mathcal{I}$ is defined by \eqref{cf}. 
In particular, Propositions \ref{prop:olF'}--\ref{prop:olG'} give
\begin{align}
\begin{cases}
\ol F(c_\mathcal{I})=\atl=\ol G(c_\mathcal{I}), &\text{if } c_\mathcal{I} < \tild{c} = \infty; \\
\ol F(\tild{c}) = G(\tild{c}) \leq  
\atl \leq  
\ol G(\tild{c}), &\text{if } c_\mathcal{I} = \tild{c}. \\
\end{cases}
\label{sol:olFolGcf}
\end{align}

\subsection{Boundary intersections} 
\label{Sec:cf}

In this section we show that in fact the second case in \eqref{sol:olFolGcf} (that is, $c_\mathcal{I} = \tild{c}$ in the dichotomy \eqref{cftildc}), where the boundaries $\Fb$ and $G$ intersect at $c=\tild c$, never arises. 
This is illustrated in Figure~\ref{fig:boundaries2}. 

For this, note first that on $(0,c_\mathcal{I})$ we have $G(c) < \ol{F}(c) < f_0+c$ (Proposition \ref{prop:olF'} and \eqref{cf}) and $c_\mathcal{I}<\ol{c}$ (Lemma \ref{cf<infinite}), which yield in view of Corollary \ref{cor:c1}  that 
\begin{equation} \label{cI<c0}
c_\mathcal{I} < c_0 = \bar c.
\end{equation}
Next we establish the following useful lemma. 

\begin{Lemma} \label{g0<gd}
For each $c \in (0,c_\mathcal{I})$ let $G_\d(c)$ be the unique root of $q (\cdot; F(c))$.
Then setting $g_\d := \lim_{c\to 0} G_\d(c)$, so that $q (g_\d; \frac1{2\d}) = 0$, we have $g_0 < g_\d$.
\end{Lemma}

We are now ready to resolve the dichotomy in \eqref{cftildc}, by proving that $c_\mathcal{I} < \tild c = \infty$ (and therefore excluding the second case in \eqref{sol:olFolGcf}).

\begin{Proposition} \label{cor:ctild}
For all $\la\in (\la^*, \la^\dagger)$ we have $\tild{c}=\infty$.
\end{Proposition}

In summary, and recalling Lemma \ref{cf<infinite} and Proposition \ref{prop:olF'}, for $(\la^*, \la^\dagger)$ we have $\ol{F}'(c_\mathcal{I}) = 1 < G'(c_\mathcal{I}) $ and
\begin{equation} \label{valuescf}
F(c_\mathcal{I}) < \frac1{2\d} < \frac1{2\d}+ c_\mathcal{I} < G(c_\mathcal{I}) < \ol{F}(c_\mathcal{I}) = \atl = \ol{G}(c_\mathcal{I}) < f_0 + c_\mathcal{I}. 
\end{equation}

\subsection{A candidate value function for the control problem $Q$} \label{sec:cpbar}

In this section we construct a candidate value function $(x,c) \mapsto \tild{Q}(x, c)$ for $(x,c)\in[0,\infty)^2$ and $\la\in (\la^*,\la^\dagger)$.
Just as in Section~\ref{Sec:verification2}, the value function for $c=0$ and $x\geq 0$ is the function $x \mapsto Q(x,0)$ from \eqref{eq:Vtild0}
and the task is to specify our candidate value function for $c>0$. 
For this, we use the boundary $G$ constructed in Section~\ref{sec:cpos} to define the fuel level $c^*$ by
\begin{equation} \label{eq:c*}
\begin{split}
c^* &:= \inf\big\{c \in [c_\mathcal{I},\infty) \,:\, G(c) = D(c) \big\}, \quad \text{where} \quad D(c) := \atl - c_\mathcal{I} + c , \quad c \geq c_\mathcal{I}.
\end{split}
\end{equation}
\begin{Remark}\label{rem:cst}
The meaning of $c^*$ is as follows. For fuel levels $c$ just above $c^*$, if $G(c)$ is repelling then repulsion is not full, but is partial, as the state then jumps to a point on the boundary $\Gb$ (cf. Figure \ref{fig:boundaries2}).
It is straightforward from \eqref{valuescf} that $c_\mathcal{I} < c^*$. 
\end{Remark}
As proven in Proposition \ref{prop:olF'} and illustrated in Figure~\ref{fig:boundaries2}, our candidate solution satisfies $\ol{F}'(c) \in (0,1)$ for all $c\in(0,c_\mathcal{I})$. Thus for all $c \geq c_\mathcal{I}$ we have
\begin{align}\label{eq:someineqs}
D(c) < f_0+c \quad \text{ and } \quad \atl + \frac c 2 < g_0 + c.     
\end{align}
Coupled with the fact that $\lim_{c\to 0}G(c)=\fotd < f_0$ for all $\la \in (\la^*,\a\d)$ and the definition \eqref{eq:c*} of $c^*$, this gives
$c^* < \inf\{c \in (0,\infty): G(c) \geq f_0+c \}$.
Then, in view of \eqref{cI<c0} we see that only the following orderings of fuel levels are possible: 
\begin{align} 
\label{pos:1c*}
&c_\mathcal{I} < c_0 = \bar c \leq c^* < \inf\{c \in (0,\infty): G(c) \geq f_0+c \}, \\
\label{pos:2c*}
&c_\mathcal{I} < c^* < c_0 = \bar c \wedge \inf\{c \in (0,\infty): G(c) \geq f_0+c \}.
\end{align}
In the case \eqref{pos:2c*}, recalling Remark \ref{rem:cst}, the boundary $G$ is repelling just above $c^*$ and so the analysis is more challenging; both cases are covered in Section~\ref{Sec:verification3bb} below. 
The consideration of small and large fuel levels (Sections \ref{Sec:verification3a} and \ref{Sec:verification3c}) is the same in both cases.

\subsubsection{Small fuel: $c \in (0,c_\mathcal{I}]$}
\label{Sec:verification3a}

In view of the properties of $F$, $G$, $\Fb$ and $\Gb$ established in Sections \ref{sec:cpos}--\ref{sec:newsol}, notably \eqref{cI<c0}, we consider the candidate value function $\tild{Q}(x;c)$ for all $(x,c)\in [0,\infty) \times (0,c_\mathcal{I}]$ defined by
\begin{equation*}
\tild{Q}(x;c) := \left\{
\begin{array}{ll}
\delta x^2, & 0 \leq x \leq F(c),\\[+3pt]
A(c) e^{x\sta} + B(c) e^{-x\sta} + \fla x^2 + \flas, & F(c) < x < G(c), \\[+3pt]
\tild{V}_0(x-c) + c, & G(c) \leq x \leq \Fb(c) ,\\[+3pt]
\tild{A}(c) e^{x\sta} + \tild{B}(c) e^{-x\sta} + \fla x^2 + \flas, & \Fb(c) < x \leq \Gb(c), \\[+3pt]
\tild{Q}(x-\zeta(x,c);c-\zeta(x,c)) + \zeta(x,c), & \Gb(c) < x < g_0 +c ,\\[+3pt]
\tild{V}_0(x-c) + c, & x \geq g_0 + c,
\end{array}
\right. 
\end{equation*}
where $A,B$ are defined by \eqref{H3=h1}--\eqref{H4=h2} and
\begin{equation} \label{zeta}
\zeta(x,c) := \inf\{u \in [0,c] \,:\, \Gb(c-u) = x-u \} \wedge c .
\end{equation}

\subsubsection{Intermediate fuel: $c \in (c_\mathcal{I}, \ol{c})$} 
\label{Sec:verification3bb}

We begin the analysis of this case with the following observation. If $c^* < \ol{c}$ and the initial fuel level is $c \in (c^*,\bar c)$ then, referring to Figure~\ref{fig:regionsnew}, we have~$G(c)>D(c)$. 
Since $G'(c) \geq 1$, the process is repelled at the boundary point $(G(c),c)$. 
In particular, if $G(c) \in (D(c), g_0 +c)$ then the process reaches the waiting region $(\Fb,\Gb)$ before expending all fuel and so it is moved from the point $(G(c),c)$ to the point $(G(c) -\zeta(G(c),c), c-\zeta(G(c),c))$ (cf.~\eqref{zeta}) which, by construction, lies on the reflecting boundary $(\Gb(\cdot),\cdot)$. 
In light of this, we define the fuel level $c^\dagger \in (c^*, \infty]$ by 
\begin{equation} \label{eq:cdag}
c^\dagger := \inf\big\{c \in (c^*,\infty) \,:\, G(c) = g_0 +c \big\}, 
\end{equation} 
so that the aforementioned repulsion from the boundary $G$ to the boundary $\Gb$ occurs precisely when $c \in (c^*, c^\dagger)$.

Taking the above into account, the candidate value function $\tild{Q}(\cdot;c)$ for $c \in (c_\mathcal{I}, \ol{c})$ should satisfy $\tild Q(G(c),c) = h_r^*(G(c);c) + \fla G^2(c) + \flas$, where we define
\begin{align} \label{def-G0*}
h_r^*(x;c) &:= \tild Q \big(x-\zeta(x,c), c-\zeta(x,c) \big) + \zeta(x,c) - \fla x^2 - \flas, 
\quad \text{for all } c > c_\mathcal{I}.
\end{align}
The definition \eqref{def-G0*} can be seen as a generalisation of the definition \eqref{def-G0} of $h_r$, since when the new waiting region ($\Fb$,$\Gb$) does not interact with the repulsion (for example, when  $x\not\in(D(c), g_0 +c)$ in the present regime) we have    
\begin{align} \label{c-z=0}
\zeta(x,c) = c 
\quad \text{and} \quad  
h_r^*(x;c) = \tild V_0(x-c) + c - \fla x^2 - \flas \equiv h_r(x;c),
\end{align}
which is precisely the expression used in Section~\ref{sec:cpos} for the construction of the boundary functions $F$ and $G$.
Conversely, for all $x \in (D(c), g_0 + c)$ we have 
\begin{align} \label{c-z<cf}
c - \zeta(x,c) \in (0, c_\mathcal{I}] 
\quad \text{and} \quad  
x-\zeta(x,c) = \Gb(c-\zeta(x,c)), 
\end{align}
in which case optimality and the strong Markov property give $h_r^*(x;c) \not= h_r(x;c)$. 
Then writing $\zeta = \zeta(x,c)$, we have 
\begin{align} \label{dxh*r}
\pd{h_{r}^*}x(x,c) 
&= \pd{\tild Q}{x}(x-\zeta, c-\zeta) - \frac{2\la}{\a} x - \pd{\zeta}x(x,c) \bigg( \Big(\pd{\tild Q}{x} + \pd{\tild Q}{c}\Big)(x-\zeta, c-\zeta) - 1 \bigg)   
\notag\\
&= \pd{\tild Q}{x}(x-\zeta, c-\zeta) - \frac{2\la}{\a} x , 
\end{align} 
where the second equality uses \eqref{c-z<cf} and the fact that $U(\Gb(c),c)=1$ (Section~\ref{sec:topo}),
further yielding $\frac{\partial^2}{\partial x^2} h_r^*(x,c) = \frac{\partial^2}{\partial x^2} \tild Q(x-\zeta, c-\zeta) - \ftla$. 
This, together with the fact that $\Gb(\cdot) \geq \atl$ (Proposition \ref{prop:olG'}), imply the following useful inequality: 
\begin{align} \label{eq:H*pos}
(\cL - \alpha) h^*_r(x;c) 
&=\frac12 \frac{\partial^2}{\partial x^2}\tild Q(x-\zeta, c-\zeta) - \alpha \tild Q(x-\zeta, c-\zeta) - \alpha \zeta + \lambda x^2 \notag\\ 
&= -\lambda(x - \zeta)^2  - \alpha \zeta + \lambda x^2 
= 2 \lambda \zeta \left(\Gb(c-\zeta) - \atl \right) + \lambda \zeta^2 > 0, 
\end{align}
where the second equality follows from the Feynman-Kac equation \eqref{eq:F-K} satisfied by \eqref{eq:VstarRR}.

In the forthcoming analysis we make the following Ansatz. 

\begin{anz}\label{ansatz2}
The candidate solution $\tild Q$ of Section~\ref{Sec:verification3a} is of class $C^1$ on $[0,\infty) \times (0,c_\mathcal{I}]$. 
\end{anz} 

\begin{Remark}
Ansatz \ref{ansatz2} is verified in the proof of Theorem \ref{thm:ver2}.     \end{Remark}

\begin{Lemma} \label{lem:Hr*}
Let $H^*_r(\cdot;c):=\Phi(h^*_r)(\cdot;c)$ be the image of $h_r^*$ from \eqref{def-G0*} under the transformation $\Phi$ of \eqref{def-H}, where $c>c_\mathcal{I}$. Then 
\begin{align} \label{def-H*}
H_r^*(y;c) = \begin{cases}
H_{r1}(y;c), & y_c < y \leq \Psi(D(c)) ,\\
\Phi(h_r^*)(y;c), & \Psi(D(c)) < y < \Psi(g_0 + c) ,\\
H_{r2}(y;c), & y \geq \Psi(g_0 + c),
\end{cases} 
\end{align}
and under Ansatz \ref{ansatz2} the function $H^*_r(\cdot;c)$ is convex and of class $C^1(y_c,\infty)$.  
\end{Lemma}

We conclude from Lemma \ref{lem:Hr*} that the obstacle $H^*_r(\cdot;c)$ is a `convexified' version of the non-convex obstacle $H_r(\cdot;c)$ illustrated in Figure~\ref{fig:Fig1synth2}. 
That is, $H^*_r$ is constructed from $H_r$ by pasting smoothly at $H_r(\Psi(D(c));c)$, where $D(c) \in (x_c, f_0+c)$, an additional convex part that covers the concave region $(y_r(c), y_m(c))$ of $H_r(\cdot;c)$ (whenever it exists), and is then pasted again smoothly at $H_r(\Psi(g_0 +c);c)$, where $g_0 +c > \atl+\frac{c}{2}$ (cf.~\eqref{eq:someineqs}).

In order to extend the boundaries $F$ and $G$ from Section~\ref{Sec:verification3a} to the intermediate fuel levels $(c_\mathcal{I},\ol{c})$, it will be convenient to extend the definition of $H^*_r(x,c)$ to all $c \in (0,\infty)$ by setting $H^*_{r}(\cdot,c) := H_r(\cdot,c)$ for each $c \in (0,c_\mathcal{I}]$ (as in Section~\ref{sec:cpos}), giving $F$ and $G$ a unified definition on $(0,\ol{c})$ in the following result. Note that by construction, the boundaries $F$ and $G$ are then of class $C^1$ at $c=c_\mathcal{I}$.

\begin{Proposition} \label{0}
Suppose that $\la \in (\la^*, \la^\dagger)$. 
Then there exists a unique couple $(y^*_1(c),y^*_2(c))$ with $1 \leq y^*_1(c) < y_c < y^*_2(c)$ solving the system:
\begin{equation} \label{eq:system-y1y20}
\left\{
\begin{array}{lr}
\pd {H_l}{y}(y_1)= \pd{H^*_{r}} y (y_2;c), \\[+4pt]
H_l(y_1) - \pd {H_l} y (y_1)y_1 = H^*_{r} (y_2;c) -\pd {H^*_{r}} y (y_2;c)y_2 ,
\end{array}
\right. 
\end{equation}
for each fixed $c \in (0, c^*_1)$, where $c^*_1 := \inf\{c \in (0,\infty): y^*_1(c) = 1 \text{ or } y^*_2(c) \geq \hye(c) \}$.
Defining $F(c):=\Psi^{-1}(y_1^*(c))$ and $G(c):= \Psi^{-1}(y_2^*(c))$, for $c\in(0,c^*_1)$, and 
\begin{align} \label{eq:FG*def}
\ol{c}&:= \inf \{c \in (0,c^*_1): G'(c) \leq 1\}, 
\end{align}
we have that $\ol{c}>0$ and that 
\begin{align} 
&(y^*_1(c), y^*_2(c)) = (\hyo(c),\hyt(c)), \quad c \in (0,c^*), \label{eq:FGeq}\\
&c_\mathcal{I} < c^* < c^*_1. \label{cf<c1} 
\end{align}
\end{Proposition}

Since $c^* < c^\dagger$ (cf.~\eqref{eq:cdag}) we have the following trichotomy for the constant $\ol{c}$ of  \eqref{eq:FG*def}:
\begin{equation} \label{c-order}
\text{Case 1:} \;\; \bar{c} \leq c^*; \qquad  
\text{Case 2:} \;\; c^* < \bar{c} < c^\dagger; \qquad 
\text{Case 3:} \;\; c^\dagger \leq \bar{c}.
\end{equation}

Consistent with the notation in \eqref{eq:th3} and \eqref{eq:th4}, we write $\cH_3^*(x,c)$ (respectively $\cH_4^*(x,c)$) for the gradient (resp.~intercept at the vertical axis) of the line tangent to $H_r^*(\Psi(x);c)$ at $y=\Psi(x)$. 
Then \eqref{def-H*} gives, in view of the trichotomy in \eqref{c-order}, that 
\begin{align} \label{FGext}
\begin{split}
h_1(F(c)) &= \begin{cases}
\tild \cH_3(G(c),c), &\quad c \in (c_\mathcal{I},c^* \wedge \ol{c}] \;\; \Leftrightarrow \; G(c) \leq D(c), \\
\cH_3^*(G(c),c), &\quad c \in (c^* \wedge \ol{c}, c^\dagger \wedge \ol{c}) \; \Leftrightarrow \; D(c) < G(c) < g_0+c, \\
\cH_3(G(c),c), &\quad c \in [c^\dagger\wedge \ol{c}, \ol{c}) \;\, \Leftrightarrow \; G(c) \geq g_0+c,
\end{cases} \\
h_2(F(c)) &= \begin{cases}
\tild \cH_4(G(c),c), &\quad c \in (c_\mathcal{I},c^*\wedge \ol{c}] \;\; \Leftrightarrow \; G(c) \leq D(c), \\
\cH_4^*(G(c),c), &\quad c \in (c^*\wedge \ol{c}, c^\dagger \wedge \ol{c}) \; \Leftrightarrow \; D(c) < G(c) < g_0+c, \\
\cH_4(G(c),c), &\quad c \in [c^\dagger\wedge \ol{c}, \ol{c}) \;\, \Leftrightarrow \; G(c) \geq g_0+c.
\end{cases}
\end{split}
\end{align}
We now construct the candidate value function $\tild{Q}(x;c)$ for these intermediate fuel levels via the greatest non-positive convex minorant of $H^*:= H_l \wedge H_r^*$, analogous to the construction in Sections \ref{Sec:verification2a} and \ref{Sec:verification2b}. 
Recalling the discretionary stopping, waiting, and action regions illustrated in Figure~\ref{fig:regionsnew}, the fact that $G'(c) > 1$ for all $c \in (0,\bar c)$ by construction (cf.~Proposition \ref{0}), and the bounds on the gradients of $\Fb$ (Proposition \ref{prop:olF'}) and $\Gb$ (Proposition \ref{prop:olG'}), for all $(x,c)\in [0,\infty) \times (c_\mathcal{I}, \bar{c})$ we have
\begin{equation*}
\tild{Q}(x;c) := \left\{
\begin{array}{ll}
\delta x^2, & 0 \leq x \leq F(c),\\[+3pt]
A(c) e^{x\sta} + B(c) e^{-x\sta} + \fla x^2 + \flas, & F(c) < x < G(c), 
\\[+3pt]
\tild{V}_0(x-c) + c, &G(c) \leq x \leq D(c),\\[+3pt]
\tild{Q}(x -\zeta(x,c);c-\zeta(x,c)) + \zeta(x,c), &D(c) \vee G(c) \leq x < g_0+c ,\\[+3pt]
\tild{V}_0(x-c) + c, & (g_0 + c) \vee G(c) \leq x.
\end{array}
\right. 
\end{equation*} 
Here $A(c) = h_1(F(c))$ and $B(c) = h_2(F(c))$, and for each $c \in (c_\mathcal{I}, \bar c)$ the equations \eqref{FGext} define the boundaries $F$ and $G$ via Proposition \ref{0}.
 
\begin{Remark}\label{rem:oldb}
Just as the system \eqref{FGext} for fuel levels $c\in(c_\mathcal{I},c^*]$ is nothing more than \eqref{H3=h1}--\eqref{H4=h2} of Section~\ref{sec:cpos}, so also for fuel levels $c\in(c^\dagger, \ol{c}]$ it is identical to the system \eqref{eq:ifuel}. Thus for fuel levels $c\in(c^\dagger, \ol{c}]$ the boundaries $F$ and $G$ are constructed as in 
Section~\ref{Sec:verification2b}.
\end{Remark}

\subsubsection{Large fuel: $c \in [\ol{c}, \infty)$}
\label{Sec:verification3c}

We begin this section by confirming that the equation $q(G(\bar c), F(\bar c)) = 0$ holds in all three cases of the trichotomy \eqref{c-order}. 

\begin{Lemma} \label{q=0*}
Suppose that $\la \in (\la^*, \la^\dagger)$. Then the boundary functions $F$ and $G$ of \eqref{FGext} satisfy $q(G(c), F(c)) > 0$ for all $c<\ol{c}$ and $q(G(\bar c), F(\bar c)) = 0$. 
\end{Lemma}

It follows from Lemma \ref{q=0*} that the boundaries $F$ and $G$ can be extended to $c\in [\ol{c}, \infty)$ via the equations~\eqref{F'olc}. Thus for $c \in [\ol{c}, \infty)$ the candidate value function $\tild{Q}(\cdot;c)$ is given by \eqref{eq:tildQolc}. 

\begin{Remark}[Verification of Ansatz \ref{bigass}] \label{rem:a1ver}
Since the boundaries $\Fb$ and $\Gb$ are of class $C^1$ with $\fotd < \lim_{c\to 0}\Fb(c) = f_0 < \atl < g_0 = \lim_{c\to 0} \Gb(c)$
(Proposition \ref{cor:xf} and \eqref{tildeg0}), 
they verify \eqref{olF00}. 
Thus Ansatz \ref{bigass} is verified. 
\end{Remark}

\section{Verification}
\label{Sec:verification}

In this section we verify the optimality of the candidate solutions obtained in Sections \ref{sec:cpos} and \ref{sec:newsol}. We begin with a verification theorem, stated in Theorem \ref{th:ver}. 
Naturally, this is a modification of \cite[Theorem 4.1]{KOWZ00}, where the modifications are only to allow for more than two boundaries and remove specific requirements on their monotonicity. 
To address the possible lack of $C^2$-smoothness across the boundaries $\{(b_i(c), c)\}$ indexed by $i$, it suffices to note that under our assumptions, for each $i$ the process $X - b_i(C)$ is a semimartingale. 
Thus the set of times at which $X_t - b_i(C_t)=0$ for any $i$ has Lebesgue measure zero, making it possible to use the mollification argument of \cite{KOWZ00} and references therein.
The optimality conditions stated in Theorem \ref{th:ver} can also be obtained following similar arguments as in \cite[Corollary~4.2]{KOWZ00}. 
We omit its proof for brevity. 

\begin{Theorem}
\label{th:ver}
Consider an evenly symmetric function $\tild Q: \R \times [0,\infty) \to \R$ which is continuous, continuously differentiable on $\R \times (0,\infty)$, and twice continuously differentiable in its first argument with locally bounded second derivatives for all $(x,c) \in \R \times [0,\infty)$ with $x \in \mathcal{R}_c := \R \setminus \{\pm b_i(c): i=1,\ldots, n_b\}$ 
where $n_b \in \mathbb{N}$ and $\{b_i, 1 \leq i \leq n_b\}$ is a finite set of piecewise $C^1$ boundaries. 
Suppose that $\tild Q(\cdot,c)$ satisfies the growth condition
\begin{gather}\label{eq:42}
\big| \tfrac{\partial\tild{Q}}{\partial x} (x,c) \big| \leq K(c)(1+|x|), \qquad \forall \; (x,c) \in \R \times [0,\infty), 
\end{gather}
for some continuous and increasing function $K:[0,\infty) \to (0,\infty)$, together with the following conditions: 
\begin{align}
&\hspace{-2mm}\tild Q(x,c) \leq \delta x^2, \quad x\in\R, c\geq 0,
\label{eq:43} \\
&\hspace{-2mm}\big(\big|\tfrac{\partial\tild{Q}}{\partial x}\big| + \tfrac{\partial\tild{Q}}{\partial c} \big)(x,c) \leq 1, \quad x\in\R, c > 0, 
\label{eq:44} \\ 
&\hspace{-2mm}(\cL - \a) \tild Q(x,c) + \la x^2 \geq 0, \quad x\in \mathcal{R}_c, c\geq 0,
\label{eq:45} \\
&\hspace{-2mm}[\delta x^2 - \tild Q(x,c)] \big[1 - \big|\tfrac{\partial\tild{Q}}{\partial x}\big|(x,c) - \tfrac{\partial \tild{Q}}{\partial c}(x,c) \big] [(\cL - \alpha) \tild Q(x,c) + \la x^2] = 0, \, x\in\mathcal{R}_c, \, c> 0. \label{eq:46}
\end{align}
Then $\tild Q(x,c) \leq Q(x,c)$ for all $(x, c) \in  \R \times [0,\infty)$.
Moreover, if $\tild Q(x,c)$ is the value of a strategy $(\xi,\tau) \in \mathcal{A}(c) \times \cT$ for $(x,c) \in [0,\infty)^2$, such that  
\begin{align} \label{eq:genito}
\left.\begin{cases}
\hspace{-1cm}
&\tild{Q}(X_\tau,C_\tau) = \d X_\tau^2, \qquad 
\int_0^\tau e^{-\alpha t} \big(\left(\cL - \alpha \right) \tild{Q} (X_{t-},C_{t-}) + \la X_t^2 \big) dt = 0 , \quad \\
&\int_{[0,\tau]} e^{-\alpha t} \big(1- \frac{\partial\tild{Q}}{\partial x}(X_{t-},C_{t-})-\frac{\partial\tild{Q}}{\partial c}(X_{t-},C_{t-}) \big) d\xi^-_t = 0, \quad \\
&\sum_{0 \leq t \leq \tau} e^{-\alpha t}\Big(\tild{Q}(X_{t-},C_{t-}) - \tild{Q}(X_{t-} - \triangle C_t, C_{t-} -\triangle C_t) + \triangle C_t \Big) = 0,
\end{cases}\right\} \text{a.s.,}
\end{align}
then $Q(x,c) = \tild Q(x,c)$ for all $(x,c) \in \R \times [0,\infty)$ and $(\xi,\tau)$ is optimal for problem \eqref{eq:defV}.
\end{Theorem}

To solve the control problem for $\la\in (\la^*,\a\d)$, consider the candidate value function $\tild Q$, constructed in Section~\ref{Sec:verification2} under the regime $\la\in [\la^\dagger,\a\d)$, and that constructed in Section~\ref{sec:cpbar} under the regime $\la\in (\la^*,\la^\dagger)$, illustrated in Figures \ref{fig:boundaries4} and \ref{fig:boundaries2} respectively.
Recalling Remark \ref{rem:sf}, in both cases it is straightforward to check that, by the construction of $\tild{Q}$, the following smooth fit conditions hold:
\begin{enumerate}
\vspace{-2mm}
\item[SF1:] \label{SF1} for the function $\tild{Q}(\cdot,c)$ across $F$ and $\Fb$, and also across the restriction of $G(c)$ for $c \leq \ol{c}$, i.e.\ on $\{c \geq 0: G'(c) \geq 1\}$,

\vspace{-2mm}
\item[SF2:] \label{SF2} for the function $(\tfrac{\partial}{\partial x} + \tfrac{\partial}{\partial c})\tild{Q}(\cdot,c)$ across $\bar G(c)$, and also across the restriction of $G(c)$ for $c > \ol{c}$, i.e.\ on $\{c \geq 0: G'(c) < 1\}$.
\end{enumerate}
\vspace{-2mm}
The details are as follows. In the candidate solution of Section \ref{Sec:verification2a},
condition SF1 at both boundaries $F$ and $G$ follows from the construction of the coefficients $h_1(F)$ and $h_2(F)$ used in \eqref{eq:tildQ}. For the same reason, condition SF1 holds at $F$ and $G$ in Section \ref{Sec:verification2b} and at $F$ in Section \ref{Sec:verification2c}. Condition SF2 at $G$ follows in Section~\ref{Sec:verification2c} by direct differentiation of $\tilde Q$. 
Similarly, in the candidate solutions of Sections~\ref{Sec:verification3a}--\ref{Sec:verification3c}, conditions SF1--SF2 follow from the construction of the coefficients $A$, $B$, $\widetilde{A}$ and $\widetilde{B}$, and from direct differentiation of~$\tilde Q$. 

The fact that conditions SF1 and SF2 hold is used in Theorem \ref{thm:ver2} below to uniquely characterise $\tild Q$.

\begin{figure}[tb]
\begin{center}
\vspace{-1.5cm}
\begin{tikzpicture}
\node[anchor=south west,inner sep=0] (image) at (0,0) {\includegraphics[width=\textwidth]{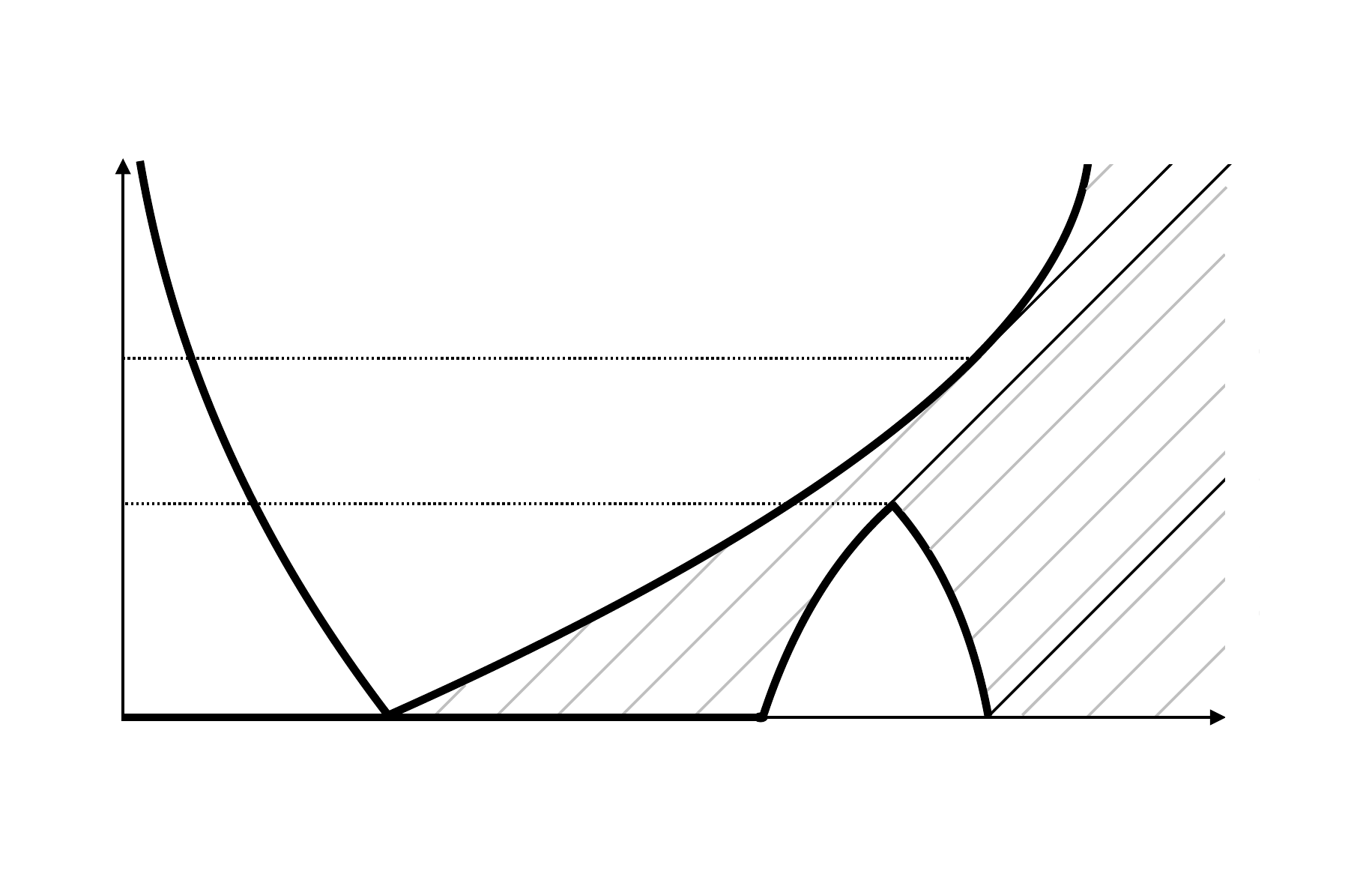}};
\begin{scope}[x={(image.south east)},y={(image.north west)}]
 \node[scale=1.2] at (0.055,0.43) {$c_\mathcal{I}$};
 \node[scale=1.2] at (0.93,0.192) {$x$};
 \node[scale=1.2] at (0.285,0.14) {$\fotd$};
 \node[scale=1.2] at (0.555,0.14) {$f_0$};
 \node[scale=1.2] at (0.65,0.14) {$\atl$};
 \node[scale=1.2] at (0.725,0.14) {$g_0$};
 \node[scale=1.2] at (0.055,0.81) {$c$};
 \node[scale=1] at (0.82,0.815) {IVa};
 \node[scale=1] at (0.815,0.6) {$x=D(c)$};
 \node[scale=1] at (0.8,0.5) {IVb};
 \node[scale=1] at (0.35,0.65) {II};
 \node[scale=1] at (0.645,0.3) {III};
 \node[scale=1] at (0.17,0.3) {I};
 \node[scale=1.2] at (0.055,0.6) {$\bar c$}; 
 \node[scale=1] at (0.85,0.3) {IVc};
 \node[scale=1] at (0.51,0.3) {IVc};
\end{scope}
\end{tikzpicture}
\vspace{-18mm}
\caption{Regions I to IV from the statement and proof of Theorem \ref{thm:ver2}, when $\la\in (\la^*,\la^\dagger)$ and $\bar{c} < c^*$. When $\la\in [\la^\dagger, \a \d)$ we have $c_\mathcal{I}=0$, and regions III and IVb are then both empty.}
\label{fig:regions}
\end{center}
\end{figure}

\begin{figure}[tb]
\begin{center}
\vspace{-1.5cm}
\begin{tikzpicture}
\node[anchor=south west,inner sep=0] (image) at (0.4,0) {\includegraphics[width=\textwidth]{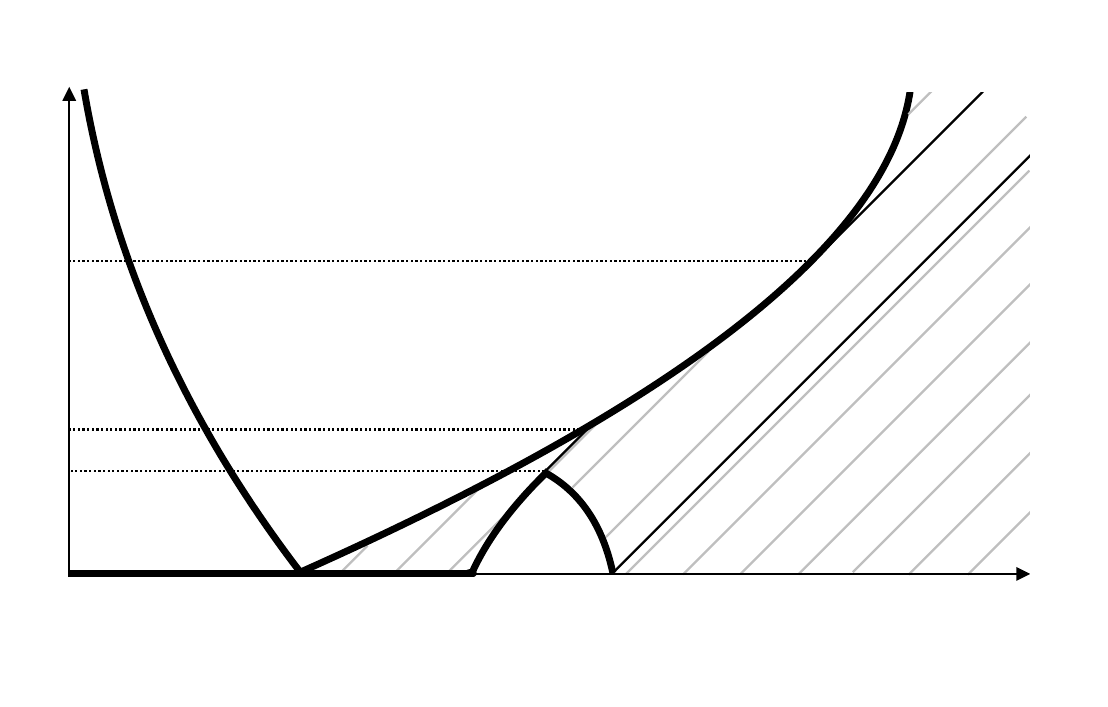}};
\begin{scope}[x={(image.south east)},y={(image.north west)}]
 \node[scale=1.2] at (0.055,0.33) {$c_\mathcal{I}$};
  \node[scale=1.2] at (0.055,0.4) {$c^*$};
 \node[scale=1.2] at (0.945,0.19) {$x$};
 \node[scale=1.2] at (0.285,0.14) {$\fotd$};
 \node[scale=1.2] at (0.435,0.14) {$f_0$};
 \node[scale=1.2] at (0.5,0.14) {$\atl$};
 \node[scale=1.2] at (0.56,0.14) {$g_0$};
 \node[scale=1.2] at (0.04,0.87) {$c$};
 \node[scale=1] at (0.835,0.87) {IVa};
 \node[scale=1] at (0.69,0.5) {IVb};
 \node[scale=1] at (0.35,0.66) {II};
 \node[scale=1] at (0.5,0.23) {III};
 \node[scale=1] at (0.17,0.27) {I};
 \node[scale=1.2] at (0.055,0.635) {$\bar c$};
 \node[scale=1] at (0.8,0.3) {IVc};
 \node[scale=1] at (0.41,0.23) {IVc};
\end{scope}
\end{tikzpicture}
\vspace{-18mm}
\caption{Regions I to IV from the statement and proof of Theorem \ref{thm:ver2}, when $\la\in (\la^*, \la^\dagger)$ and $c^* < \bar c$. When $\la\in [\la^\dagger, \a \d)$ we have $c_\mathcal{I}=0$, and regions III and IVb are then both empty.}
\label{fig:regionsnew}
\end{center}
\end{figure}

\begin{Theorem}\label{thm:ver2} 
Let $c_\mathcal{I} \geq 0$, $\ol{c}>0$ and define the regions
\begin{align*}
    &{\rm I} = \{(x,c) : c>0, \; 0 \leq x \leq F(c)\}, \qquad 
    &&{\rm II} = \{(x,c) : c>0, \; F(c)< x < G(c)\}, \\
    &{\rm III} = \{(x,c) : c \in (0,c_\mathcal{I}), \; \bar F(c)< x < \bar G(c)\}, \qquad
    &&{\rm IV} = \{(x,c) : c>0, \; x \geq G(c)\} \setminus {\rm III},
\end{align*}
whose boundaries 
$F:[0,\infty) \to (0,\infty)$, $G:(0,\infty) \to (0,\infty)$ are of class $C^1((0,\infty)\setminus\{\ol{c}\})$ and 
$\bar{F}, \bar{G}:(0,c_\mathcal{I}] \to (0,\infty)$ are of class $C^1(0,c_\mathcal{I})$, with
\begin{align*}
&F'(c)<0 \; \forall \; c>0, 
\quad G'(c)>1 \; \forall \; c\in(0,\ol{c}), 
\quad G'(c)\in(0,1) \; \forall \; c>\ol{c}, \\ 
&\lim_{c\to 0} F(c) \leq \fotd \leq \lim_{c\to 0}G(c), 
\quad q(G(c);F(c))>0 \; \forall \; c<\ol{c},
\quad G(c) \geq \atl \; \forall \; c>\ol{c}, \\ 
&\bar{F}'(c) \in (0,1), \; \bar{G}'(c) < 0 \; \forall \; c \in (0,c_\mathcal{I}), 
\quad G(c) < \bar{F}(c) \; \forall \; c \in (0,c_\mathcal{I}],  \\
&\lim_{c\to 0}\Fb(c) = f_0,
\quad \atl <  \lim_{c\to 0}\Gb(c) < \atl+\frac1{\sqrt{2\a}},  
\quad \bar{F}(c_\mathcal{I}) = \bar{G}(c_\mathcal{I}) = \atl.
\end{align*}
Suppose that for some fuel-dependent coefficients $A(c), B(c), \tild A(c), \tild B(c)$  we have 
\begin{align}
&\tild Q(x,c) = \tild Q_1(x,c) := \delta x^2, & (x,c) \in {\rm I}, \label{vc1}\\
&\tild Q(x,c) = \tild Q_2(x,c) := A(c) e^{x\sta} + B(c) e^{-x\sta} + \fla x^2 + \flas, & (x,c) \in {\rm II}, \label{vc2}\\
&\tild Q(x,c) = \tild Q_3(x,c) := \tild A(c) e^{x\sta} + \tild B(c) e^{-x\sta} + \fla x^2 + \flas, & (x,c) \in {\rm III}, \label{vc3}\\
&\tild Q(x,c) = \tild Q_4(x,c) := \tild Q\big(x-\zeta(x,c), c-\zeta(x,c)\big) + \zeta(x,c), & (x,c) \in {\rm IV}, \label{vc4}
\end{align}
where $\zeta(x,c):= \inf\{u\in (0,c]: (x-u,c-u) \not\in IV\}$ and $\tild Q(\cdot,0) = Q(\cdot,0)$ from \eqref{eq:Vtild0}. 
Suppose that this candidate $\tild{Q}$ is continuous and satisfies the smooth fit conditions \hyperref[SF1]{SF1}--\hyperref[SF2]{SF2}.
Then the even (in $x$) extension of the function $\tild Q$  to $\R \times [0,\infty)$ and the moving boundaries $F$, $G$, $\bar F$, $\bar G$ satisfy the conditions of Theorem \ref{th:ver}.
\end{Theorem}
    
\begin{proof}
Since this even extension of $\tild Q$ is of class $C^2$ across the boundary $x=0$ (where it is locally quadratic in $x$), we examine the regularity only of the function $\tild Q$. 
Firstly, it may be checked by direct calculation from \eqref{vc1}--\eqref{vc4} that $\tild Q$ is twice continuously differentiable in its first argument with locally bounded second derivatives away from the set 
\begin{align} \label{eq:cbarns}
\begin{split}
&\{(x,c): c>0, \; x \in \{F(c), G(c), \bar F(c), \bar G(c)\}\} \;\cup\; \{(x,c): c \geq \ol{c}, \; x = G(\ol{c}) + c - \ol{c}\} \qquad \\
& \quad \;\cup\; \{(x,c): x = f_0+c \} \;\cup\; \{(x,c): x = g_0 + c\} \;\cup\; \{(x,c): c \geq c_\mathcal{I}, \; x = D(c) \}. 
\end{split}
\end{align}
Then, the smooth fit conditions \hyperref[SF1]{SF1}--\hyperref[SF2]{SF2} ensure that $\tild Q$ is continuously differentiable across the boundaries and line segments making up the set \eqref{eq:cbarns}, as follows. 

\vspace{1mm}
{\it Proof that $\tild Q$ is of class $C^1$ across the boundaries $F, G, \bar F$ and $\bar G$, and across the line segments $c \mapsto G(\ol{c})+c-\ol{c}$, $c \mapsto f_0+c$, $c \mapsto g_0+c$ and $D$ on their corresponding domains.}  
Let $R$ be one of the boundaries $F$, $\Fb$ or $\{(G(c),c): G'(c) \geq 1 \}$. 
Then two surfaces meet along the boundary $R$, let these be $\tild Q_i$, $\tild Q_j$ in the notation of \eqref{vc1}--\eqref{vc4}, for $i \neq j$. 
Then from \hyperref[SF1]{SF1} we have
\begin{align}
\tild Q_i(R(c),c) &= \tild Q_j(R(c),c),  \label{eq:SF01} \\
\pd{\tild Q_i}{x}(R(c),c) &= \pd{\tild Q_j}{x}(R(c),c). \label{eq:SF02}   
\end{align}
Differentiating \eqref{eq:SF01} gives 
\begin{align} \label{eq:SFderiv}
\pd{\tild Q_i}{x}(R(c),c)R'(c) + \pd{\tild Q_i}{c}(R(c),c) = 
\pd{\tild Q_j}{x}(R(c),c)R'(c) + \pd{\tild Q_j}{c}(R(c),c),
\end{align}
then substituting \eqref{eq:SF02} gives $\pd{\tild Q_i}{c}(R(c),c) = \pd{\tild Q_j}{c}(R(c),c)$, so that $\tild Q$ is of class $C^1$ across the boundary $R$.

Alternatively, let $R$ be one of the boundaries $\Gb$ or $\{(G(c),c): G'(c) < 1 \}$, along which two surfaces $\tild Q_i$, $\tild Q_j$ meet. Then from the definition of the regions $I$ to $IV$ we have $i \in \{2,3\}$ and $j=4$, and it follows from \eqref{vc4} that \eqref{eq:SF01} again holds.
From \hyperref[SF2]{SF2} we have
\begin{align} \label{eq:SF03}
\Big(\pd{\tild Q_i}{x} + \pd{\tild Q_i}{c} \Big)(R(c),c) &= 
\Big(\pd{\tild Q_4}{x} + \pd{\tild Q_4}{c} \Big)(R(c),c) = 1, 
\end{align}
and substituting this into \eqref{eq:SFderiv} gives
\begin{align*}
(R'(c) - 1)\pd{\tild Q_i}{x}(R(c),c) = (R'(c) - 1) \pd{\tild Q_4}{x}(R(c),c).
\end{align*}
Then since $R'(c) <1$ we have $\pd{\tild Q_i}{x}(R(c),c) = \pd{\tild Q_4}{x}(R(c),c)$, so that \eqref{eq:SF03} gives $\pd{\tild Q_i}{c}(R(c),c) = \pd{\tild Q_4}{c}(R(c),c)$ and hence $\tild Q$ is of class $C^1$ across $R$.

Finally, the $C^1$ property across all line segments except $D$ can be established as in \cite{KOWZ00}, by appealing to \cite[Theorem 10.3]{KOWZ00}, straightforward differentiation, and the analysis of \cite[Theorem 9.1]{KOWZ00}, in that paper.
Although the line segment $D$ does not arise in the latter paper, it can be dealt with similarly, by noting that $\tild Q$ is of class $C^1$ on $[0,\infty) \times [0, c_\mathcal{I}]$ and then using the identity $\tild Q(x,c_\mathcal{I}+c')=\tild Q(x-c',c_\mathcal{I})+c'$, which is valid across $D$ by construction, at least for sufficiently small $c'>0$. 
For completeness we also recall that the $C^1$ property trivially also holds across these line segments wherever they lie in the interior of the domain of functions $\tild Q_2(\cdot,c)$, $\tild Q_3(\cdot,c)$ in \eqref{vc2} and \eqref{vc3} respectively, since $\tild Q$ is of class $C^2$ there.

\vspace{1mm}
{\it Proof that \eqref{eq:44} holds.} 
On region I, \eqref{vc1} implies that $\left(\pd{}{x} + \pd{}{c}\right) \tild Q(x,c) = 2 \delta x \leq 1$, since in this region we have $x\leq F(c) <\fotd$. 

For $(x,c)$ in the interior of region IV, for sufficiently small $\epsilon>0$ we have from \eqref{vc4}
\[
\frac{\tild Q(x + \epsilon,c + \epsilon) - \tild Q(x,c)}{\epsilon} = 1,
\]
hence $1 = \left(\pd{}{x} + \pd{}{c}\right) \tild Q(x,c) = U(x,c)$ (cf.~\eqref{eq:UUdef}). 

Then since region II shares boundary $F$ with region I and boundary $G$ with region IV, we combine the above observations with the continuous differentiability of $\tild Q$ across these boundaries (established above) to give 
\begin{align}\label{eq:FGU}
U(F(c),c) = 2 \d F(c) \in (0,1) 
\quad \text{and} \quad 
U(G(c),c) = 1 .
\end{align}
We then claim that $x \mapsto \pd{U}{x}(x,c)$ is a strictly concave function with at most one zero in $(F(c),G(c))$ and $\pd{U}{x}(G(c);c) \geq 0$. Together with \eqref{eq:FGU}, this means that $U(\cdot,c)$ is either strictly increasing or first strictly decreasing and then strictly increasing. In both cases \eqref{eq:44} holds.

To prove the above claim about $\pd{U}{x}$, note from {\eqref{vc2}} that 
\begin{align*}
U(x,c) &= \ftla x + C(c)e^{x\sta} + D(c)e^{-x\sta}, \quad x\in(F(c),G(c)).  
\end{align*}
It is straightforward to check from the equations $A(c) = h_1(F(c)), B(c) = h_2(F(c))$ (Section \ref{Sec:verification3a}) that 
\begin{align*}
C(c) &= \sta A(c) + A'(c) < 0 \quad \text{and} \quad D(c) = -\sta B(c) + B'(c) > 0, 
\end{align*}
where the inequalities follow from $F'(\cdot)<0$ and \cite[Lemma 8.1]{KOWZ00}. 
This implies that $\pd{U}{x}(\cdot,c)$ is strictly concave. 
For $c>\ol{c}$, we have from the  smooth fit assumption \hyperref[SF2]{SF2} that 
$\pd{U}{x}(G(c);c) = 0.$
For $c<\ol{c}$ it follows from \eqref{eq:FGU}, the explicit expressions for $A(c)$ and $B(c)$ in Section \ref{Sec:verification3a}, and straightforward manipulations that the pair $(F,G)$ solves
\begin{equation*} 
2 \sta e^{G(c)\sta}q\big(G(c);F(c) \big) = \big( e^{2G(c)\sta}-e^{2F(c)\sta} \big) \, \pd{U}{x}\big(G(c),c\big), 
\end{equation*} 
where $q$ is defined by \eqref{eq:qdef}. 
Given that the left-hand side is strictly positive, we get that $\pd{U}{x}(G(c);c) > 0$.

In region III, we follow similar arguments as for region II to prove that $x \mapsto \pd{U}{x}(x,c)$ is a strictly concave function on $(\bar F(c),\bar G(c))$. 
This follows since condition \hyperref[SF2]{SF2} gives $U(\bar G(c),c) = 1$ and $\pd{U}{x}(\bar G(c);c) = 0$, implying that \eqref{tildeA'B'} holds, and then the definitions of the functions $h_3$ and $h_4$ in \eqref{eq:h3} and \eqref{eq:h4} with the fact that $\bar G(\cdot) \in [\atl, \atl + \frac1{\sqrt{2\a}})$ establish the signs of the quantities in \eqref{tildeA'B'}. 
Combining this strict concavity with the equation $U(\bar F(c),c)=1$ (which follows from \eqref{vc4}, condition \hyperref[SF1]{SF1} and the continuous differentiability of $\tild{Q}$ across $\Fb$), we conclude that $\pd{U}{x}(\cdot,c)$ has exactly one zero in $(\bar F(c),\bar G(c))$. 
Hence $U(\cdot,c)$ is first strictly decreasing and then strictly increasing, thus \eqref{eq:44} holds.

\vspace{1mm}
{\it Proof that \eqref{eq:43} holds.} 
In region I this condition holds with equality. 
In region II, the claim is proved in the same way as in \cite[Theorem 9.1]{KOWZ00}. 
The condition holds in regions III, IVb and IVc since 
\begin{align*}
\tild{Q}(x,c) &= \tild{Q}(x-c,0) + \int_0^c \Big(\frac{\partial\tild{Q}}{\partial x} + \frac{\partial\tild{Q}}{\partial c}\Big)(x-c+s,s) ds \\
&\leq \delta (x-c)^2 + \int_0^c 1 \, ds \\
&\leq \delta (x-c)^2 + \int_0^c \Big(\frac{\partial}{\partial x} + \frac{\partial}{\partial c}\Big)(\delta (x-s)^2)ds = \d x^2,
\end{align*}
where the first line follows by integrating along the line segment in $[0,\infty)^2$ from $(x-c,0)$ to $(x,c)$; 
the second from the fact that \eqref{eq:43} holds for $\tild{Q}(\cdot,0)$ and since \eqref{eq:44} holds; 
and the third from the fact that regions III, IVb and IVc are included in $\{(x,c) \in [0,\infty)^2 : x \geq \fotd + c\}$ (cf.~Figures \ref{fig:regions}--\ref{fig:regionsnew}). 
The proof for region IVa is analogous, except that for $c > \bar c$ we integrate along the line segment from the point $(x - \zeta(x,c), c - \zeta(x,c))$ on the boundary of region II (where $x - \zeta(x,c) = G(c - \zeta(x,c))$, and where we have already established that \eqref{eq:43} holds) to the point $(x,c)$.

\vspace{1mm}
{\it Proof that \eqref{eq:45} holds.} 
By construction the condition holds with equality in regions II and III. In region I direct calculation gives 
\begin{align*}
\frac12 \frac{\partial^2\tild{Q}}{\partial x^2}(x,c) - \alpha \tild Q(x,c) + \lambda x^2  
= \delta + (\lambda - \alpha \delta)x^2 
\geq \delta + (\lambda - \alpha \delta)f_0^2 > 0,
\end{align*}
where the first inequality follows from $x \leq \fotd < f_0$ when $\la\in (\la^*,\a\d)$ and the second from Lemma \ref{lem:Hlprops}.  

Region IV, illustrated in Figures \ref{fig:regions}--\ref{fig:regionsnew} and constructed as in \eqref{vc4}, can be conveniently subdivided into regions IVa--IVc as follows:
\begin{align} \label{partition}
\left. \begin{cases}
\text{IVa} \;:=\; \{(x,c): \tild Q(x,c) = \tild Q(G(c - \zeta), c - \zeta) + \zeta \}, \; \zeta < c,\\
\text{IVb} \;:=\; \{(x,c): \tild Q(x,c) = \tild Q(\Gb(c - \zeta), c - \zeta) + \zeta \}, \; \zeta < c,\\
\text{IVc} \;:=\; \{(x,c): \tild Q(x,c) = \tild Q(x-c,0) + c \}, \; \zeta = c. 
\end{cases}\right\} 
\end{align}
To address region IVa, write $\eta = x - \zeta$, $\theta = c - \zeta$, $G(\theta) = \eta$, and note that $\theta > \ol{c}$ (that is, the point $(\eta,\theta)$ lies on the reflecting part of the boundary $G$). 
We proceed as follows (similar arguments can be found in \cite{KOWZ00}). 
Under the conditions of Theorem \ref{thm:ver2}, the boundary $G$ is of class $C^1(\ol{c}, \infty)$ and we have condition \hyperref[SF2]{SF2}, thus $\tild Q$ is of class $C^2$ across the reflecting portion of the boundary $G$. 
In particular we have 
$\frac{\partial^2}{\partial x^2}\tild Q (x,c) = \frac{\partial^2}{\partial x^2}\tild Q(\eta,\theta)$ and since $G(\theta) \geq \atl$, for $\theta>\ol{c}$, we get
\begin{align*}
&\frac12 \frac{\partial^2}{\partial x^2} \tild Q(x,c) - \alpha \tild Q(x,c) + \lambda x^2 
=\frac12 \frac{\partial^2}{\partial x^2} \tild Q(\eta, \theta) - \alpha \big( \tild Q(\eta, \theta) + \zeta \big) + \lambda x^2 \notag\\
& \qquad = (\cL - \alpha) \tild Q(\eta, \theta) + \lambda x^2 - \alpha \zeta = \lambda \big( x^2 - (x-\zeta)^2 \big) - \alpha \zeta \notag\\
& \qquad = 2 \lambda \zeta \Big((x-\zeta) - \atl \Big) + \lambda \zeta^2
= 2 \lambda \zeta \Big(G(\theta) - \atl \Big) + \lambda \zeta^2 \geq 0.
\end{align*}
The region IVb is addressed in the same way in \eqref{eq:H*pos}.
For $(x,c) \in \text{IVc}$ note that $\tild Q(x,c) = h_r(x;c) + \fla x^2+ \las$, so that
\begin{align} \label{eq:Qhrel}
(\cL - \alpha) \tild Q(x,c) + \lambda x^2  = (\cL - \alpha) h_r(\Psi^{-1}(y);c).
\end{align}
Recalling Lemma \ref{lem:Hconvexity}, it is clear from the locations of the boundaries $G$, $\bar F$ and $\bar G$  (that is, the facts that $\frac 1 {2\d}+\frac c 2 < G(c)$ for all $c\in (0,\ol{c}]$, $G(c) < \bar F(c) < f_0+c$ for all $c\in (0,c_\mathcal{I}]$, and $\lim_{c\to 0}\Gb(c) > \atl$), that the transformed obstacle $y \mapsto H_r(y;c)$ is convex on each connected component of $\Psi\big(([0,\infty) \times \{c\}) \cap \text{IVc} \big)$. 
Lemma \ref{lem:geomlemma} and \eqref{eq:Qhrel} then give that $(\cL - \alpha) \tild Q(x,c) + \la x^2 \geq 0$.

\vspace{1mm}
{\it Proof that \eqref{eq:46} holds.} 
By the above discussion, on each region I--IV one of conditions \eqref{eq:43}--\eqref{eq:45} holds with equality, so condition \eqref{eq:46} holds. 

\vspace{1mm}
{\it Proof that \eqref{eq:42} holds.}
The growth condition \eqref{eq:42} follows from \eqref{vc1}--\eqref{vc4}, the partition \eqref{partition}, $C^1$ regularity and \eqref{eq:Vtild0}.
\hfill $\Box$ 
\end{proof} 

\vspace{3mm}
We may now conclude by confirming that the regularity needed for Theorem \ref{thm:ver2} holds for all $\la\in (\la^*,\a\d)$, i.e.~in both of the regimes that we study. 

\vspace{1mm}
{\it Regime $\la\in [\la^\dagger,\a\d)$}. 
In this case we set $c_\mathcal{I}=0$ so that the boundaries $\Fb$ and $\Gb$ play no role, and take the candidate $\tild{Q}$ with the boundaries $F$ and $G$ defined in Sections \ref{Sec:verification2a}--\ref{Sec:verification2c}. 
The conditions of Theorem \ref{thm:ver2} for $F$ and $G$ are then established as follows. For the small fuel levels of Section~\ref{Sec:verification2a}, their limits at $0$, continuous differentiability, and derivatives are given by 
\eqref{cor:1/2d0}--\eqref{eq:FGdef}, Proposition \ref{cor:1/2d}, and Proposition \ref{prop:monotonicity} respectively.
For the larger fuel levels of Sections \ref{Sec:verification2b} and \ref{Sec:verification2c}, their derivatives and slopes have already been established in \cite{KOWZ00} and their location is given in \eqref{eq:longineq}. 
The positivity property for $q$ follows from Lemma \ref{cor:FGC1}. 
The continuity of $\tild Q$ away from $c=0$ is established in the proof of Theorem \ref{thm:ver2} and its continuity at $c=0$ can easily be deduced from Lemma \ref{lem:minorant}. 

\vspace{1mm}
{\it Regime $\la\in (\la^*,\la^\dagger)$}.
The candidate $\tild{Q}$ in this regime is that described in Sections \ref{Sec:verification3a}--\ref{Sec:verification3c}.
The conditions of Theorem \ref{thm:ver2} required of $F$ and $G$ are established as in the above regime $\la\in [\la^\dagger,\a\d)$, with the addition of Proposition \ref{0}, and the positivity property for $q$ is shown in Lemma \ref{q=0*}. 
The conditions required of $\Fb$ and $\Gb$ are established as follows. 
Their limits at 0, continuous differentiability, and derivatives follow respectively from \eqref{tildeg0}; Proposition \ref{cor:xf}; and Propositions \ref{prop:olF'}--\ref{prop:olG'} together with Proposition \ref{cor:ctild} (see also \eqref{valuescf}). 
The continuity of $\tild Q$ away from $c=0$ is established in the proof of Theorem \ref{thm:ver2} and its continuity at $c=0$ is given by Proposition \ref{prop:newbdrs}.

\vspace{1mm}
We thus may now conclude that problem \eqref{eq:defV} has the value function $\tild Q$ constructed in Section~\ref{Sec:verification2} under the regime $\la\in [\la^\dagger,\a\d)$, and that constructed in Section~\ref{sec:cpbar} under the regime $\la\in (\la^*,\la^\dagger)$. The remaining details, including the associated optimal strategy, are given in Corollary \ref{cor:opts} below.

\begin{Corollary}\label{cor:opts}
If the conditions of Theorem \ref{thm:ver2} hold, then the following strategy $(\xi,\tau) \in \mathcal{A}(c) \times \cT$ is optimal in the control problem \eqref{eq:defV}:
\begin{align} \label{sol:genver}
\left.
\begin{cases}
&\text{{\rm (a)} If the fuel level is $c=0$:} \\
&\quad \text{ {\rm (i)} stop immediately if $0 \leq X_t \leq f_0$}, \\
&\quad \text{ {\rm (ii)} wait if $X_t>f_0$}. \\
&\text{{\rm (b)} If the fuel level is $c>0$:} \\ 
&\quad \text{ {\rm (i)} stop immediately if  $(X_t,C_t) \in \text{I}$}, \\
&\quad \text{ {\rm (ii)} wait without expending fuel if $(X_t,C_t) \in \text{II} \cup \text{III}$}, \\
&\quad \text{ {\rm (iii)} if $(X_t,C_t) \in \text{IV}$ then, if possible, expend just enough fuel to} \\
&\qquad \quad \text{ensure that $(X_\cdot, C_\cdot) \in \text{II} \cup \text{III}$ at almost all times; if impossible }\\
&\qquad \quad \text{then expend all fuel and proceed as in {\rm(a)}.} \\
\end{cases}
\right\}
\end{align}
\end{Corollary}

\begin{proof}
The strategy $(\xi, \tau)$ described by \eqref{sol:genver} is elementary except for the indication in (b).(iii) to `if possible, expend just enough fuel to ensure that $(X_\cdot, C_\cdot) \in \text{II} \cup \text{III}$ at almost all times', which applies to all regions in Figures \ref{fig:regions}--\ref{fig:regionsnew} except region I. 
To make this precise we appeal to \cite[Corollary 5.2]{Dupuis93}, which gives the existence and uniqueness of a strong solution to the Skorokhod reflection problem. 

For initial states $(x,c)$ lying in region IVc of Figures \ref{fig:regions}--\ref{fig:regionsnew}, it is clearly impossible for $(X_t,C_t)$ to lie in $\text{II} \cup \text{III}$ at almost all times, so all fuel is immediately expended. 

For initial states $(x,c)$ lying in region III $\cup$ IVb of Figures \ref{fig:regions}--\ref{fig:regionsnew} we have 
$$
x>\bar F(c) {\bf 1}_{\{c\leq c_\mathcal{I}\}} + D(c){\bf 1}_{\{c>c_\mathcal{I}\}}, 
\quad \text{where $D(c)$ is defined by \eqref{eq:c*}}.
$$ 
Take $(X_t,C_t)$ to be the solution to the Skorokhod reflection problem for the Brownian motion process $(x + W_t,c)$ with boundary $\bar{G}$ and reflection in the direction $(-1,-1)$. 
This reflected process immediately enters region III and, almost surely, subsequently it either expends all of its fuel by repeated reflection in $\Gb$, or it hits the boundary $\bar F$. 
In the latter case, at that time it becomes impossible to remain in region III and so all remaining fuel is immediately expended.

For initial states $(x,c)$ lying in region IVa of Figures \ref{fig:regions}--\ref{fig:regionsnew} we take $(X_t,C_t)$ to be the solution to the Skorokhod reflection problem for the Brownian motion process $(x + W_t,c)$ with the restriction of the boundary $G$ to $\{c \geq 0: G'(c) < 1 \}$ and reflection in the direction $(-1,-1)$. 
This reflected process immediately enters region II and, almost surely, subsequently hits either $F$ or the restriction of the boundary $G$ to $\{c \geq 0: G'(c) \geq 1 \}$. 
If the former occurs first then, according to (b).(i) of \eqref{sol:genver}, discretionary stopping happens at that time; 
if the latter occurs first then it becomes impossible to remain in region II and so all remaining fuel is immediately expended.
        
Then by the strong Markov property we have:
\begin{enumerate}
\vspace{-2mm}
\item the time spent by $(X_t,C_t)$ outside II $\cup$ III has Lebesgue measure 0, and in II $\cup$ III the expressions \eqref{vc2}--\eqref{vc3} satisfy $(\cL - \alpha) \tild Q(x,c) + \lambda x^2 = 0$;

\vspace{-2mm}
\item fuel is expended only when $(X_t,C_t) \in$ IV (acting on the process $(X_t,C_t)$ in the direction $(-1,-1)$);
    
\vspace{-2mm}
\item whenever $\xi^-$ increases via a jump, the value of $\tild Q$ decreases via a jump of equal magnitude, so that the final line of \eqref{eq:genito} holds.
\end{enumerate}
\vspace{-2mm}
This implies that the conditions in \eqref{eq:genito} hold almost surely, and, in view of Theorem \ref{th:ver}, the strategy $(\xi, \tau)$ of \eqref{sol:genver} is optimal for the control problem \eqref{eq:defV}. 
\hfill $\Box$ 
\end{proof}

\begin{appendix}

\section{Proofs of results in Sections \ref{sec:methods}--\ref{sec:newsol}}
\label{AppA}

We first recall the following useful result (see, e.g.\ \cite[Appendix A.1]{DeAFeMo15b}).

\begin{Lemma} \label{lem:geomlemma}
Let $g: \R \times [0, \infty) \to \R$. For values of $y$ and $c$ such that the following derivatives exist, we have: 
\begin{itemize}
\vspace{-1mm}
\item[\rm (i)] $\pd{\Phi(g)}{y}(y;c)$ has the same sign as $\pd{}{x}\big(\frac{g(x;c)}{\phi(x)} \big) \big|_{x=\Psi^{-1}(y)}$;

\vspace{-2mm}
\item[\rm (ii)] $\pd{^2 \Phi(g)}{y^2}(y;c)$ has the same sign as $(\cL - \alpha) g(\Psi^{-1}(y);c)$,
\end{itemize}
\vspace{-1mm}
where the sign of $0$ is defined to be $0$.
\end{Lemma}

\subsection{Proofs of results in  Section~\ref{sec:methods}}
\label{AppA3}

\noindent {\bf Proof of Lemma \ref{lem:Hlprops}.}
Part $\rm (i)$ follows from \eqref{eq:Hl}, and part $\rm (ii)$ from Lemma \ref{lem:geomlemma} and the straightforward calculation
\begin{align*} 
(\cL - \alpha) h_l(x) 
&= \d - (\a \d - \la)x^2.
\end{align*}
Note that this quadratic in $x$ satisfies $(\cL - \alpha) h_l(0) = \d > 0$ and 
\begin{align*} 
(\cL - \alpha) h_l(f_0) &= (\a \d -\la) \Big(\frac 1 \a+2\frac{f_0}{\sta}\Big)>0.
\end{align*}
Similarly, part $\rm (iii)$ follows from 
\begin{equation*} 
\pd{}{x} \Big(\frac{h_l(x)}{\phi(x)}\Big) \frac {e^{-x\sta}} {\sta (\d-\frac{\la}{\a})}=
\rho(x),
\end{equation*}
with $\rho(x)$ as in \eqref{eq:defrho}. 
\hfill$\Box$

\vspace{3mm}
\noindent {\bf Proof of Lemma \ref{lem:littleh}.}
The continuous differentiability of $x \mapsto h_r(x;c)$ follows from that of $V_0$, and \eqref{def-G1} is obtained from \eqref{eq:Vtild0}. To determine the minimum $h_l(x) \wedge h_r(x;c)$ in \eqref{eq:hdef}, it is easily checked from \eqref{def-G1} that the continuously differentiable function $x \mapsto h_r(x;c) - h_l(x)$ is concave. Also for $ x < f_0+c$ we have
\begin{equation*}
h_r(x;c) - h_l(x) = h_{r1}(x;c) - h_l(x) 
= c(1+\d(c-2x)), 
\end{equation*}
which is strictly decreasing in $x$ with root $x_c$ defined by \eqref{x_c}. 
\hfill$\Box$ 

\vspace{3mm}
\noindent {\bf Proof of Lemma \ref{lem:Hconvexity}.}
Fixing $c>0$, the first claim on the smoothness of $H_r$ follows from the smoothness of the transformation $\Phi$ of \eqref{def-H}. 
We have
\begin{align}
(\cL - \alpha) h_{r1}(x;c) &=
\d - (\a \d - \la)x^2 +  c \a ( \d (2 x -c ) - 1), \label{hr1cc} \\
(\cL - \alpha) h_{r2}(x;c) &= c(\la (2 x -c ) - \a). \label{hr2cc}
\end{align}
The expression \eqref{hr1cc} is concave with maximum value attained at 
$$x_{\rm max} \equiv x_{\rm max}(c) = \frac{\a\d c}{\a\d-\la}.$$ 
Using \eqref{eq:defrho}--\eqref{eq:fodef} and the fact that $f_0(\la^*) = \fotd$, we can show that $(\cL - \alpha) h_{r1}(x_{max};c)>0$ and $(\cL - \alpha) h_{r1}(f_0+c;c)>0$ for all $c>0$ and all $\la\in(\la^*,\a\d)$, while 
$$ (\cL - \alpha) h_{r1}(x_c;c) \begin{cases} 
> 0 , & \forall \; 0<c<\oc , \\
< 0 , & \forall \; c >\oc . \end{cases}
$$
Hence, in view of Lemma \ref{lem:geomlemma} and the above calculations, we conclude that:

(i) if $c>\oc$, there exists a unique $x=x_v(c)$ in $[0, f_0+c)$ such that the function $y \mapsto H_r(y;c)$ is strictly convex on $(y_v(c) \vee y_c,\hye(c))$, where $y_v(c) := \Psi(x_v(c))$ (and strictly concave on $(y_c, y_v(c))$ if $y_c < y_v(c)$); 

(ii) if $c\leq\oc$, then the function $y \mapsto H_r(y;c)$ is strictly convex on the whole of $(y_c,\hye(c))$.
\noindent 
We therefore define
$$y_v(c):= \begin{cases} \Psi(x_v(c)) , &\text{for all } c > K , \\ 
y_c , &\text{for all } c \leq K . \end{cases} $$
To complete the proof, we examine the geometry of $(\cL - \alpha) h_{r2}(\cdot;c)$:
 
In case $\rm (I)$, it is straightforward (using \eqref{eq:longineq1}) to see that \eqref{hr2cc} is strictly positive for all $x \geq f_0+c$, hence the required results follow from Lemma \ref{lem:geomlemma}.

In case $\rm (II)$ we use the fact that  
$$ (\cL - \alpha) h_{r2}(f_0+c;c) \begin{cases} 
> 0 , & \forall \; c > \uc ,\\
< 0 , & \forall \; 0<c<\uc,  \end{cases} 
$$
the fact that $(\cL - \alpha) h_{r2}(\cdot;c)$ is linearly increasing, and Lemma \ref{lem:geomlemma} to draw the required conclusion.
\hfill$\Box$

\subsection{Proofs of results in  Section~\ref{sec:cpos}}
\label{AppA4}

\noindent {\bf Proof of Lemma \ref{lem:minorant}.}
The tangent $r_{y_u}(\cdot;0)$ has strictly negative gradient. Also, from Figure~\ref{fig:1}, there exists $\epsilon > 0$ such that 
\begin{align} \label{eq:uncon}
r_{y_u}(y;0) < W_0(y) - \epsilon, \quad \forall \; y \in [y_w, \infty).    
\end{align}
We first restrict attention to a finite interval, and appeal to the uniform convergence of the functions $H(\cdot;c)$ to the function $W_0(\cdot)$ as $c \to 0$. 
For this, denote respectively by $H^I(\cdot;c)$, $W_0^I(\cdot)$ the restrictions of these functions to the compact domain $I:=[y_w, y_z]$, where $y_z:=\Psi(\max\{f_0, \atl\}+1)$. 
Then for any $c'>0$, the functions $\{H^I(\cdot;c): c \in [0,c']\}$ are equicontinuous, and they also converge pointwise to $W_0^I(\cdot)$ as $c \to 0$. 
This convergence is therefore uniform on $I$. It follows from \eqref{eq:uncon} that for sufficiently small $c$, we have $r_{y_u}(\cdot;0) < H(\cdot;c)$ on the interval $I$, and it remains only to prove this on the interval $J:=(y_z, \infty)$. 

To that end, we construct a straight line $\tild l$ with 
\begin{align} \label{eq:sepli}
r_{y_u}(\cdot;0) < \tild l < H(\cdot;c) \quad \text{on } J,
\end{align} 
as follows. By the above (uniform) convergence, for sufficiently small $c$ we have 
$$
H(y_z;c) > W_0(y_z) - \epsilon > r_{y_u}(y_z;0).
$$
For small $c$, by the definition of $H_{r2}$ we have $H(y_z;c) = H_{r2}(y_z;c)$. Also, it is straightforward to check from Lemma \ref{lem:Hconvexity} that, for sufficiently small $c$, the function $H(\cdot;c)$ is convex on $J$. It is also easy to check from \eqref{def-H} and \eqref{def-G1} that 
\[\pd{}{y} H_{r2}(y_z;c) \geq -c \, \frac{\la}{\a\sta y_z}.\]
Taking $\tild l$ to be the tangent to $H_{r2}(\cdot; c)$ at $y=y_z$, it follows that \eqref{eq:sepli} holds for sufficiently small $c$.
\hfill$\Box$

\vspace{3mm}
\noindent {\bf Proof of Proposition \ref{existencegeometricy2}.}
The proof is split into several steps.  
\vspace{1mm}

{\it \ul{Step 1:} Conditions for the existence of a unique couple $(\hyo(c),\hyt(c))$ for $0 \vee \uc \leq c \leq \oc$.}
From Lemma \ref{lem:Hconvexity}, we have that on $(y_c, \infty)$, $H_r(\cdot;c)$ is strictly convex and so $P_r(\cdot;c)$ is strictly decreasing. For all $c>\oc$, similarly $P_r(\cdot;c)$ increases strictly on $(y_c,y_v(c))$ and decreases strictly on $(y_v(c),\infty)$.
From the continuity of the function $y \mapsto P_r(y;c)$ and \eqref{eq:Pyc}--\eqref{eq:Pinfty} we conclude that there exists a unique point $\hat{y}_2(c) \in (y_c \vee y_v(c),\infty)$ satisfying $P_r(\hat{y}_2(c);c)=0$. By the convexity of $H_l$ from Lemma \ref{lem:Hlprops}, the supremum in \eqref{def-P_r} must be uniquely achieved at some point $\hyo(c) \in [1,y_c)$. 

\vspace{1mm}
{\it \ul{Step 2:} Existence of at least one couple $(\hyo(c),\hyt(c))$ for $\la \in (\la^*,\la^\dagger)$, $c \in(0, \uc)$}. 
In this case we have from Lemma \ref{lem:Hconvexity} that $P_r(\cdot\,;c)$ is strictly decreasing on
$(y_c, \hye(c)) \cup (\hyc(c), \infty)$ and strictly increasing on $(\hye(c), \hyc(c))$.
We conclude as in step 1, except that now there exist at least one and at most three roots of $P_r(\cdot;c)$, with the sets $(y_c,  \hye(c))$, $[\hye(c), \hyc(c)]$ and $(\hyc(c), \infty)$ each containing at most one root. 

\vspace{1mm}
{\it \ul{Step 3:} Tangent between $H_l$ and $H_{r1}$ for sufficiently small $c>0$ and its uniqueness}. 
The above construction gives
\[
\limsup_{c \to 0} \hyo(c) \leq \Psi \Big( \frac 1 {2\d} \Big) \leq \liminf_{c \to 0} \hyt(c).
\] 

If $\liminf_{c \to 0} \hyo(c) = y^* < \Psi \lb \frac 1 {2\d} \rb$ then take $y_u \in (y^*,\Psi(\fotd))$ and $y_w=\Psi(\fotd)$ in Lemma \ref{lem:minorant}.  For sufficiently small $c$, the tangent $r_{y_u}$ to $H_l$ at $y_u$ is then a separating line between the tangent $r_{\hyo(c)}$ to $H_l$ at $\hyo(c)$ and the obstacle $H_r$. Thus the  tangent $r_{\hyo(c)}$ cannot also be tangent to $H_r$, which is a contradiction.

Similarly, if $\limsup_{c \to 0} \hyt(c) = y^\circ > \Psi \lb \frac 1 {2\d} \rb$ then take $y_u$, $y_w$ in Lemma \ref{lem:minorant} with $\Psi(\fotd) < y_u < y_w < \min\{y^\circ,\Psi(f_0)\}$.
Then for sufficiently small $c$, the tangent $r_{y_u}$ to $H_l$ at $y_u$ is a separating line between the tangent $r_{\hyo(c)}$ to $H_l$ at $\hyo(c)$ and the point $(\hyt(c), H_r(\hyt(c);c))$. Thus the tangent $r_{\hyo(c)}$ cannot also be the tangent to $H_r$ at $y=\hyt(c)$, which is a contradiction.

We have thus established the limits \eqref{cor:1/2d0}. In particular, $\hyo(c)>1$ for sufficiently small $c>0$, so that $c_m>0$ in \eqref{eq:cm}. 
Similarly, $\hyt(c) < \hye(c)$ for sufficiently small $c>0$, giving $\hat{c}>0$ in \eqref{eq:c2defb} and hence $c_1 > 0$. 

\vspace{1mm}
{\it \ul{Step 4:} Proof of $\hyo(c)\leq \Psi(f_0)$.}
Recall also from Lemma \ref{lem:Hlprops} that $H_l(y)$ is increasing on $(\Psi(f_0),\infty)$. 
If the tangency point $\hyo(c)$ lay in that interval then, as $y \to \infty$, the common tangent would increase to infinity, whereas $H_r$ decreases without bound thanks to \eqref{eq:Hneglim}. 
This contradiction implies that $\hyo(c)\leq \Psi(f_0)$.
\hfill $\Box$

\vspace{3mm}
\noindent {\bf Proof of Proposition \ref{cor:1/2d}.}
It is sufficient to apply the Implicit Function Theorem to system \eqref{eq:system-y1y2} with respect to the independent variable $c$ as in \cite[Appendix A]{DeAFeMo14}. Here the corresponding matrix of derivatives  
has determinant 
$$
D(\hyo(c),\hyt(c),c) 
= (\hyt(c)-\hyo(c)) \, \pd{^2H_l}{y^2}(\hyo(c)) \, \pd{^2H_r} {y^2} (\hyt(c);c).
$$ 
Recalling Proposition \ref{existencegeometricy2}, this determinant is strictly positive for $c \in (0,c_1)$.
The functions $H_l(\cdot)$, $H_r(\cdot;c)$ ($=H_{r1}(\cdot;c)$) are twice differentiable and strictly convex on these respective intervals. 
The limit established in \eqref{cor:1/2d0} completes the argument.
\hfill $\Box$

\vspace{3mm}
\noindent {\bf Proof of Lemma \ref{derivL}.}
Combining \eqref{LG=0} with the continuity of $(x,c) \mapsto \tild \cH_3(x,c)$, we conclude that the map $(x,c) \mapsto L(x,c)$ is well-defined in an open neighbourhood of $(G(c),c)$. 

Since $h_1$ is the gradient of the tangent to $H_l(\cdot)$ at $y=\Psi(x)$, its derivative $\pd {h_1}{x}(x)$ is strictly positive when $y=\Psi(x)$ lies in a strictly convex region of $H_l(\cdot)$. 
From Lemma \ref{lem:Hlprops} and \eqref{eq:hidefs1}, it has an inverse $h_1^{-1}:[-\frac{\la}{2\a^2},0] \to [0,f_0]$ which is differentiable on $(-\frac{\la}{2\a^2},0)$. 
Hence, the map $(x,c) \mapsto L(x,c)$ is also differentiable in an open neighbourhood of $(G(c),c)$. 

It is easily checked that 
$$
h_1'(x) = -e^{-2x\sta}h_2'(x) ,
\quad \text{so that} \quad 
( h_2 \circ h_1^{-1})'(w)=-e^{2h_1^{-1}(w)\sta}.$$
Using this and \eqref{H3H4} after differentiating \eqref{eq:Ldot}, we obtain \eqref{eq:lx}.

Then \eqref{eq:hall} gives
\begin{equation*} 
\Big( \pd {L}{x}+\pd {L}{c} \Big)(x,c)= \sta \big( \tild \cH_4 (x,c) - \tild \cH_3 (x,c) e^{2z \sta} \big) + h_3(x) e^{2z \sta} - h_4(x),
\end{equation*}
and by setting $x=G(c)$ and recalling \eqref{H3=h1}, we obtain \eqref{eq:hall3} as required.
\hfill$\Box$

\vspace{3mm}
\noindent {\bf Proof of Lemma \ref{cor:FGC1}.}
Recall that $\tild \cH_3(G(c),c)$ is the gradient at the point $y=\hyt(c)= \Psi(G(c))$ of the function $y \mapsto H_{r1}(y;c)$, which is strictly convex at that point (cf.~Proposition \ref{existencegeometricy2}).
Thus we have 
$$
\pd{\tild \cH_3}{x} (G(c),c) > 0 
\quad \overset{\eqref{eq:lx}}{\Rightarrow} \quad 
\pd{L}{x}(G(c),c)< 0 ,
$$ 
since $z=h_1^{-1}(\tild \cH_3(G(c),c)) = F(c) < G(c)$. Differentiation with respect to $c$ of the identity \eqref{LG=0} and using \eqref{eq:hall3} then gives \eqref{eq:Gderiv}.
\hfill$\Box$

\vspace{3mm}
\noindent {\bf Proof of Lemma \ref{lem:Fcsmall}.}
We again argue for a contradiction, assuming that $\fotd \leq F(c)$ with $c \in (0,c_1)$ (note that the relevant geometry is illustrated in Figures \ref{fig:Fig1synth} and \ref{fig:Fig1synth2}). 
Then on the domain $y \in (\hyo(c),\infty)$, due to the convexity of $H_l$ on $[1, \Psi\left(f_0\right)]$ (Lemma \ref{lem:Hlprops}), the tangent $y \mapsto r_{\Psi(\fotd)}(y;c)$ lies below the common tangent of $H_l$ and $H_{r1}$.

Next we claim that the transformed obstacle $H(y;c)$ lies above the common tangent line of $H_l$ and $H_{r1}$. When $\la\in [\la^\dagger,\a\d)$, this was proved in Theorem \ref{existencegeometric} above (cf.\ Figure~\ref{fig:Fig1synth}). 
When $\la\in (\la^*, \la^\dagger)$ it follows from Proposition \ref{existencegeometricy4} (cf.\ Figure~\ref{fig:Fig1synth2}), which is proved in Appendix~\ref{AppA5} but which may be read now and completes the claim.
Thus in both cases, the greatest non-positive convex minorant $W(\cdot;c)$ of $H(\cdot;c)$ also lies above this common tangent.

Putting together the above observations, since $\hyo(c) < y_c$ (Proposition \ref{existencegeometricy2}), then taking $x=\frac1{2\d}+c$ gives $W(\Psi(x)) \geq r_{\Psi(1/(2\d))}(\Psi(x);c)$. Next, recall from Proposition \ref{prop:DayKar} that the optimal stopping value function $V$ of \eqref{def-V} is given by $V(\cdot;c) = \Phi^{-1}(W)(\cdot;c)$ (where $\Phi$ is the transformation of \eqref{def-H}).
Then the equations \eqref{eq:hidefs1}--\eqref{eq:hidefs2} for the tangent at $x=\frac1{2\d}$ give
\begin{align*}
V \Big(\fotd+c; c \Big) &\geq h_1 \Big(\fotd \Big) \, e^{( \fotd+c ) \sta} + h_2 \Big(\fotd \Big) \, e^{- ( \fotd+c ) \sta} .
\end{align*}
Since $x_c = \fotd+\frac{c}2 < \fotd+c < f_0+c$, we also have 
$$
h\Big( \fotd+c; c \Big) 
= h_{r1}\Big( \fotd+c; c \Big) 
= - \flas - \fla \, \Big( \fotd+c \Big)^2 + c + \frac 1 {4\d},
$$
which eventually gives 
\begin{align*}
&V \Big(\fotd+c; c \Big) - h \Big(\fotd+c; c \Big) \\
&\geq h_1\Big(\fotd \Big) e^{(\fotd+c)\sta} + h_2 \Big(\fotd \Big) e^{-(\fotd+c)\sta} + \flas + \fla \Big( \fotd+c \Big)^2 - c - \frac 1 {4\d}\\ 
&= \frac{\a\d-\la}{4\d^2} \, c^2 + O(c^3) , \quad \text{ as } c \to 0.
\end{align*}
For sufficiently small $c$ this inequality contradicts the fact that 
$$
V \Big(\fotd+c; c \Big) \leq h \Big(\fotd+c; c \Big) ,
$$ 
which is a direct consequence of the definition \eqref{def-V} of $V$. This completes the proof.
\hfill $\Box$

\vspace{3mm}
\noindent {\bf Proof of Proposition \ref{prop:monotonicity}.}
In view of \eqref{eq:longineq1}--\eqref{eq:longineq2} and the limits \eqref{cor:1/2d0}, as well as Lemma \ref{lem:Fcsmall}, we have established that for sufficiently small $c$ we have 
\[
F(c)<\fotd  \quad  \text{ and } \quad G(c)<\atl .
\]
Taking $z=F(c)$ with sufficiently small fixed $c$, we have from \eqref{eq:qprop1}--\eqref{eq:qprop2} that
$$
q(F(c),F(c))>0 \quad \text{ and } \quad 
\pd{q}{x}(x;F(c))>0 , 
\quad \text{for all } x \in (F(c), G(c)).
$$
Then $q(G(c);F(c))>0$, so that Lemma \ref{cor:FGC1} implies that $G'(c)>1$, for sufficiently small $c$. 
We thus conclude that $c_0>0$ and $G'(\cdot)>1$ on $(0,c_0)$. 

Finally, the monotonicity of $F(\cdot)$ follows from that of $\hyo(\cdot)$. 
More precisely, fix the pair $(\tild y, \tild c)$ by choosing $\tild c \in (0,c_0)$ and setting $\tild y:=\hyt(\tild c)$. 
By the common tangency property we also have $r_{\hyo(\tild c)}(\tild y;\tild c) = H_{r1} (\tild y, \tild c)$.
Differentiating \eqref{def-G1} and \eqref{def-H} gives
\begin{align} \label{eq:Hr1dec}
\pd {h_{r1}} c (G(\tild c); \tild c)=1-2\d(G(\tild c)-\tild c)<0 
\quad \overset{\eqref{def-H}}{\Rightarrow} \quad 
\pd {H_{r1}} c (\tild y; \tild c)<0 ,
\end{align}
where we have used the fact that $G(\tild c) > \frac 1 {2\d}+\tild c$ since $G'(\cdot)>1$ on $(0,\tild c)$.
Suppose for a contradiction that $\hyo'(\tild c) \geq 0$. 
Setting $a(c):=r_{\hyo(c)}(\tild y;c)$, we would then have $a'(\tild c) \geq 0$ 
since $H_l$ is convex and does not depend on $c$. Therefore recalling \eqref{eq:Hr1dec}, there exists $\e>0$ such that for $c \in (\tild c, \tild c + \e)$ we have
$$
H_{r1} (\tild y,c) < H_{r1} (\tild y, \tild c)=r_{\hyo(\tild c)}(\tild y;\tild c) \leq  r_{\hyo(c)}(\tild y;c),
$$
contradicting the fact that the common tangent lies below the transformed obstacle.
\hfill$\Box$

\vspace{3mm}
\noindent {\bf Proof of Corollary \ref{cor:c1}.}
Definition \eqref{eq:defqc} gives $\tild V \geq 0$ (since all its terms are non-negative) and $\tild V(0;c)=0$ (as it is optimal to stop immediately when $x=0$).
Thus no point $(0,c)$ can lie in any waiting region. 

Suppose for a contradiction that for some $c' \in (0, \bar c \wedge \hat{c})$ we have $\hyo(c') = 1$, which in view of \eqref{eq:FGdef} is equivalent to $F(c')=0$. Then by the proof of Proposition \ref{prop:monotonicity}, $F$ is strictly decreasing at $c'$. This contradicts the above observation that no point $(0,c)$ can be part of the waiting region, establishing the first equality for $c_0$ in the statement. The second equality is a matter of definitions (namely \eqref{ys}, \eqref{eq:c2def} and \eqref{eq:FGdef}), completing the proof.
\hfill$\Box$

\subsection{Proofs of results in  Section~\ref{sec:newsol}}
\label{AppA5}

\noindent {\bf Proof of Proposition \ref{existencegeometricy4}.}
The proof is split into the following steps:

\vspace{1mm}
{\it \ul{Step 1:} Existence of the unique couple $(\hyoo(c),\hytt(c))$ for sufficiently small $c>0$.}
We begin by acknowledging a dichotomy. It is not difficult to see from the geometry of $H_l$ (Lemma \ref{lem:Hlprops}) and of $H_r$ (Lemma \ref{lem:Hconvexity}.(II).(iii)) that for sufficiently small $c>0$ there are two possibilities. 
One possibility is that the \gcm \, has the geometry of Figure~\ref{fig:Fig1synth} with one linear portion, which is tangent to $H_l$ and $H_{r2}$. This corresponds to the waiting region in problem $V$ of \eqref{def-V} having a single connected component. A second possibility is that it has the geometry of Figure~\ref{fig:Fig1synth2}, with two distinct linear portions: one which is tangent to $H_l$ and $H_{r1}$, and another which is tangent to $H_{r1}$ and $H_{r2}$, and this latter case corresponds to the waiting region for $V$ being disconnected, with two connected components. The former case would resemble the regime $\la \in (\la_*,\la^*]$ (cf.~Remark \ref{rem:l<la*}), in that the concave portion of $H_r$ would then lie inside the single connected component of the waiting region for $V$. 
In contrast, in the latter case the concave portion of $H_r$ creates the two connected components of the waiting region for $V$. 

We claim, in order to resolve this dichotomy, that for sufficiently small $c>0$ the geometry is that of Figure~\ref{fig:Fig1synth2}.
Recalling the line $r_{y}(\cdot\,;c)$ and 
signed distance $P_r(\cdot\,;c)$ of Definition \ref{def:tangents}, define
\begin{equation} \label{yb}
y_b(c) := \sup \{ y \geq y_c :\, P_r(y;c) = 0 \},
\end{equation}
(it was established in Step 1 of Proposition \ref{existencegeometricy2} that the set on the right-hand side of \eqref{yb} is not empty). 
If $y_b(c) < y_r(c)$ there is no common tangent between $H_l$ and $H_{r2}$, so the first case of the above dichotomy cannot hold. 
Let us therefore suppose that  $y_b(c) \geq y_r(c)$, which means that $y_b(c)$ is the unique zero of $P_r(\cdot;c)$ on the strictly convex region of $H_{r2}$.
Writing $y_a(c)$ for the unique point of tangency on $H_l$ with $r_{y_b(c)}(\cdot\,;c)$, we see that the geometry is that of Figure~\ref{fig:Fig1synth2} if there exists a point $y^* \in (y_a(c),y_b(c))$ at which $H_r(y^*;c) < r_{y_b(c)}(y^*;c)$; otherwise the geometry is that of Figure~\ref{fig:Fig1synth}.

As $c \to 0$ we have $y_a(c) \to \Psi(f_0)$ by Lemma \ref{lem:minorant}, and $\hyt(c) \to \Psi(\fotd)$ by \eqref{cor:1/2d0}. Let $c>0$ be sufficiently small that $y_a(c) > \hyt(c)$ holds. Then recalling \eqref{eq:hbigf0}--\eqref{x_c}, we obtain
$H_r(y_a(c)\,;c) < r_{y_b(c)}(y_a(c)\,;c) = H_l(y_a(c))$, establishing the claim.

With this claim established, it follows from the convexity of $H_{r2}(\cdot\,;c)$ on $(y_b(c), \infty)$ and \eqref{eq:Hneglim} that we may uniquely define $\hytt(c)$ by
$$
\hytt(c) := \inf \big\{u \geq y_b(c) :\; r_{u}(y;c) \leq H_r(y;c) , \;\forall \; y>1 \big\}.
$$ 
Then by strict convexity $r_{\hytt(c)}(\cdot\,;c)$ is tangent to $H_{r1}$ at a unique point $\hyoo(c) \in (y_c,\hye(c))$ (Lemma \ref{lem:Hconvexity}.(II).(iii)).

\vspace{1mm}
{\it \ul{Step 2:} The bound $\hyt(c) < \hyoo(c)$.}
Having established in Step 1 the existence of $y^* \in (y_a(c),y_b(c))$ with $H_r(y^*;c) < r_{y_b(c)}(y^*;c)$, this bound is now obvious from the geometry established in Lemma \ref{lem:Hconvexity}.(II).(iii) and illustrated in Figure~\ref{fig:Fig1synth2}.

\vspace{1mm}
{\it \ul{Step 3:} Limit as $c \downarrow 0$.} 
From Lemma \ref{lem:minorant},  $\liminf_{c \to 0} \hyoo(c) \geq \Psi(f_0)$. 
From Step 1 we have $\limsup_{c \to 0} \hyoo(c) \leq \lim_{c \to 0} y_r(c) = \Psi(f_0)$, giving $\lim_{c \to 0} \hyoo(c) = \Psi(f_0)$.
\hfill $\Box$

\vspace{3mm}
\noindent {\bf Proof of Proposition \ref{prop:newbdrs}.} 
We first address necessity.
If the control problem \eqref{eq:defV} has continuous boundaries $F,G,\Fb,\Gb$ satisfying \eqref{olF00} then Corollary \ref{cor:oset} gives $\lim_{c \to 0}\Fb(c) \leq f_0 < \atl \leq \lim_{c \to 0}\Gb(c)$, and \eqref{eq:VstarRR} gives 
\begin{equation*}
\lim_{c \downarrow 0} \tild{Q}(x,c) = \fla x^2 + \flas + \tild A(0) e^{x\sta} + \tild B(0) e^{-x\sta} ,
\quad x \in \big(\lim_{c \to 0}\Fb(c), \lim_{c \to 0}\Gb(c) \big),
\end{equation*}
where $\tild A(0) := \tild \cH_3(\lim_{c \to 0}\Fb(c),0)$ and $\tild B(0) := \tild \cH_4(\lim_{c \to 0}\Fb(c),0)$ due to \eqref{tildeAB}. 
Further, if the candidate $\tild{Q}(x,\cdot)$ is continuous then (recalling \eqref{eq:Vtild0}) we have 
\begin{equation} \label{eq:fbgblim}
\lim_{c \downarrow 0} \tild{Q}(x,c) = Q(x,0) 
= \tild V_0(x) = \begin{cases}
\d x^2, & 0 \leq x \leq f_0, \\[+1mm]
\fla x^2 + \flas + B_0 e^{-x\sta}, & x > f_0,
\end{cases}
\end{equation}
giving $\tild{A}(0) = 0$. 
An inspection of \eqref{eq:th3} with $c=0$ shows that this is equivalent to 
$\lim_{c \to 0}\Fb(c)=f_0$ (and then $\tild{B}(0) = \tild \cH_4(f_0,0) = B_0$ {from \eqref{eq:B0}}), completing the claim.

To establish sufficiency, suppose that $\lim_{c \to 0}\Fb(c)= f_0$. Then \eqref{eq:fbgblim} gives continuity at $c=0$ for $x \in \big(\lim_{c \to 0}\Fb(c), \lim_{c \to 0}\Gb(c) \big)$. 
For $x<f_0$ and $x > \lim_{c \to 0}\Gb(c)$, the continuity of $\tild Q$ at $c=0$ follows by the construction of the strategy in Ansatz \ref{bigass}.
\hfill $\Box$

\vspace{3mm}
\noindent {\bf Proof of Lemma \ref
{lem:soltildeq}.}
Considering the function $\tild q(\cdot; z)$ on $[z,\infty)$, we have
\begin{align*}
\pd{\tild q}{x}(x;z) 
\overset{\eqref{eq:qprop2}}{=} \la \sqrt{\frac 2 \a}\Big( \atl - x \Big) e^{x\sta} \big(1-e^{2(z-x)\sta} \big), \quad z \leq x < \infty,
\end{align*}
and so 
\begin{align*}
x \mapsto \tild q(x;z) \text{ is } \begin{cases}
\text{strictly increasing on } (z,\atl), \\
\text{strictly decreasing on } (\atl, \infty). 
\end{cases}
\end{align*} 
By definition $\tild q(z; z) = 0$ and, clearly, $\tild q(\infty; z) = - \infty$ for all $z > \frac1{2\d}$. Hence for fixed $z \in (\frac1{2\d}, \atl)$, there exists a unique $\mathcal{X}(z) > z \vee\atl$ satisfying $\tild q(\mathcal{X}(z);z)= 0$, such that $\pd{\tild q}{x}(\mathcal{X}(z); z) < 0$, and $\mathcal{X}$ is of class $C^1$ on $(\frac1{2\d}, \atl)$ by the implicit function theorem.

The first upper bound for $\mathcal{X}(z)$ in \eqref{eq:2ub} follows from the observation that 
$$
\tild q\Big(\atl + \frac1{\sqrt{2\a}}; z \Big) 
= \ftla e^{z\sta} \Big[z - \atl - \frac1\sta e^{(z-\atl)\sta - 1} \Big] < 0 , 
\quad \text{for all } z\in \Big(\frac1{2\d}, \atl \Big).
$$ 
The second upper bound follows from
$$
\tild q\Big(\frac\a\la - z; z \Big) 
= \big(1+ e^{2(z-\atl)\sta} \big) \phi(z), 
\quad \text{where} \quad 
\phi(z):= h_4(z) - h_3(z) e^{2\sta\atl} ,
$$
since it is easily checked that
\[
\phi'(z) = \sta \fla \Big(z - \atl\Big) \Big(e^{z\sta} - e^{\sta(\frac{\a}{\la}-z)}\Big) > 0, \quad z\in \Big(\frac1{2\d}, \atl \Big),
\]
giving
$$
\tild q\Big(\frac\a\la - z; z \Big) 
< \big(1+ e^{2(z-\atl)\sta} \big) \phi \Big(\atl \Big) = 0 , 
\quad \text{for all } z\in \Big(\frac1{2\d}, \atl \Big),
$$
which completes the proof.
\hfill $\Box$

\vspace{3mm}
\noindent {\bf Proof of Lemma \ref{cf<infinite}.}
Since $G(0)=\frac 1{2\d}$, $G'(0+)>1$ and $G'(c)>1$ for all $c\in(0,\ol{c})$
(see Section~\ref{sec:cpos} for details), we have
\begin{equation} \label{G>12d}
\frac1{2\d} + c < G(c), \quad \text{for all } c \in (0,\ol{c}]. 
\end{equation}
Then from \eqref{eq:Gderiv}, since $G'(\ol{c})=1$ we have 
$q(G(\ol{c}), F(\ol{c})) = 0$. 
Given that $F(\ol{c}) < \fotd$ due to Proposition \ref{prop:monotonicity}, the proof of \cite[Lemma 8.2]{KOWZ00} gives $\atl < G(\ol{c})$ 
and we conclude that 
$$
\exists \, ! \; c_g \in (0,\ol{c}) \quad \text{such that} \quad G(c) < \atl \quad \text{for all } c \in (0,c_g), \quad \text{and} \quad G(c_g) = \atl.
$$
Combining this with the inequality in \eqref{G>12d}, we observe that $c_g$ also satisfies 
$c_g < \ol{c} \wedge \ol{k}$. 

Then, we have by construction that
$$
\ol{F}(c) \in \Big( G(c), \atl \Big) , \quad \text{for all } c \in (0, c_\mathcal{I}) \quad \text{and} \quad 
\begin{cases}
\ol{F}(\tild{c}) = G(\tild{c}) < \atl, &\text{if } c_\mathcal{I} = \tild{c}, \\
G(c_\mathcal{I}) < \ol{F}(c_\mathcal{I}) = \atl, &\text{if } c_\mathcal{I} < \tild{c}.
\end{cases}
$$
Thus by definition $c_\mathcal{I} \leq c_g < \ol{c} \wedge \ol{k}$. The final assertion of the lemma then follows from \eqref{eq:tildeqdef} and Lemma \ref{lem:soltildeq}.
\hfill $\Box$

\vspace{3mm}
\noindent {\bf Proof of Lemma \ref{lem:olF>12d}.}
The bounds for $\ol G$ follow directly from \eqref{olF<olG}. The upper bound for $\ol F$ follows by the definition \eqref{cf} of $c_\mathcal{I}$, while the lower bound follows from Ansatz \ref{bigass} and Lemma \ref{cf<infinite}.
\hfill $\Box$

\vspace{3mm}
\noindent {\bf Proof of Lemma \ref{lem:dtildeqdz}.}
Fix $z \in (\frac1{2\d}, \atl)$. Recalling \eqref{eq:qdef}--\eqref{eq:qprop1} and differentiating $\tild{q}(x; z)$ with respect to $z$, we obtain 
\begin{align*}
\pd{\tild q}{z}(x;z) 
= \sqrt{2\a} \ftla e^{z \sta} \Big[ 
\Big(\atl - \frac1\sta - x \Big) e^{(z-x) \sta} 
+ z - \atl + \frac1\sta \Big] . 
\end{align*}
Thus
\begin{align}
\pd{\tild q}{z}(\mathcal{X}(z); z) < 0
\quad & \Leftrightarrow \quad 
\Big(\atl - \frac1\sta - \mathcal{X}(z) \Big) e^{-\mathcal{X}(z) \sta} 
< \Big(\atl - \frac1\sta - z \Big) e^{- z \sta} \nonumber \\
\label{h3xz}
& \Leftrightarrow \quad h_3(\mathcal{X}(z)) < h_3(z). 
\end{align}
Since $z < \atl < \frac\a\la - z$, with the values $z$ and $\frac\a\la - z$   equidistant from $\atl$, the explicit expression
\begin{align*}
h_3'(x) 
= \sta \fla \Big( x-\atl \Big) e^{-x\sta} \quad \begin{cases}
<0, &\text{for } x < \atl \\
>0, &\text{for } x > \atl ,
\end{cases}
\end{align*}
gives 
$$
h_3 \Big(\frac\a\la - z\Big) < h_3(z) , \quad 
\text{for all } z \in \Big(\frac1{2\d}, \atl \Big).
$$
Combining all of the above with the inequality 
$z < \atl < \mathcal{X}(z) < \frac\a\la - z$ from Lemma \ref{lem:soltildeq} gives
$$
h_3(\mathcal{X}(z)) < h_3 \Big(\frac\a\la - z\Big) < h_3(z) , \quad 
\text{for all } z \in \Big(\frac1{2\d}, \atl \Big),
$$
and \eqref{h3xz} completes the proof.
\hfill $\Box$

\vspace{3mm}
\noindent {\bf Proof of Proposition \ref{cor:xf}.}
Fixing $c\in(0,c_\mathcal{I})$, $\Fb$ is of class $C^1$ as it is the solution to the ODE \eqref{F'}. Then $\Gb = \mathcal{X} \circ \Fb$ is also of class $C^1$ (Lemma \ref{lem:soltildeq}).
Differentiation of the equation $\tild q(\mathcal{X}(z);z) = 0$ (Lemma \ref{lem:soltildeq}) with respect to $z$ then yields
\begin{equation*}
\mathcal{X}'(\ol{F}(c)) 
= - \Big(\pd{\tild q}{z}(\mathcal{X}(\ol{F}(c));\ol{F}(c)) \Big) \Big(\pd{\tild q}{x}(\mathcal{X}(\ol{F}(c));\ol{F}(c)) \Big)^{-1} , 
\end{equation*}
and Lemma \ref{lem:soltildeq} gives
\begin{equation*} 
\pd{\tild q}{x}\big(\mathcal{X}(\ol{F}(c));\ol{F}(c) \big) < 0, 
\end{equation*}
since $\ol{F}(c) \in \big(\frac1{2\d}+c, \atl \big)$ due to Lemma \ref{lem:olF>12d}. 
Then, Lemma \ref{lem:dtildeqdz} completes the proof.
\hfill $\Box$

\vspace{3mm}
\noindent {\bf Proof of Proposition \ref{prop:olF'}.}
Recall from \eqref{dxH3} that
\begin{align} \label{dH3>0}
\pd{\tild \cH_3}{x}(\ol{F}(c), c) > 0 
\quad \text{for} \quad 
\ol{F}(c) \in \Big(\frac1{2\d}+\frac c2, f_0+c \Big). 
\end{align}
The proof has four parts. 

\vspace{1mm}
{\it(i) Proof that $\ol{F}(c) < f_0 + c$ for sufficiently small $c$.} 
Observe from \eqref{eq:th3} that 
\begin{align*}
\tild \cH_3(x,0)
&= e^{-x\sta} \Big( \frac12 \dmla x^2 + \osta \dmla x - \fltas \Big), 
\end{align*}
with derivative
\begin{align*}
\pd{\tild \cH_3}{x}(x,0)
&= \frac{\a\d-\la}{\sqrt{2\a}} e^{-x\sta} \Big( \frac\delta{\a\d-\la} - x^2 \Big) ,
\end{align*}
which in view of \eqref{eq:defrho} and Lemma \ref{lem:Hlprops} imply that 
\begin{align} \label{H3f0}
\tild \cH_3(f_0,0) = 0 
\qquad \text{and} \qquad 
\pd{\tild \cH_3}{x}(f_0,0) > 0.
\end{align}
From \eqref{F'} and \eqref{eq:hall} we have
\begin{align} \label{eq:Fbrefac}
\ol{F}'(c) = 1 + \dfrac{h_3(\mathcal{X}(\ol{F}(c))) - h_3(\ol{F}(c)) }{\pd{\tild \cH_3}{x}(\ol{F}(c),c)}.
\end{align}
Since $f_0 \in (\frac1{2\d}, \atl)$ (cf.~\eqref{eq:longineq2}), we have from Lemma \ref{lem:dtildeqdz} that 
$h_3(\mathcal{X}(f_0)) < h_3(f_0)$. 
Hence \eqref{H3f0}--\eqref{eq:Fbrefac} and continuity give $\ol{F}'(0+) < 1$ and thus
\begin{align*}
\ol F(c)<f_0+c, \quad \text{for sufficiently small $c>0$}.
\end{align*}

\vspace{1mm}
{\it (ii) Proof of the bounds for $\ol{F}$, and proof that $\ol{F}'(c) < 1$ for all $c\in(0,c_\mathcal{I})$.}  
Note first (from \eqref{dH3>0} and Lemma \ref{lem:olF>12d}) that for $c \in (0,c_\mathcal{I})$, the denominator in the last term of \eqref{eq:Fbrefac} is strictly positive provided that $\ol{F}(c) < f_0 + c$. 
By part $(i)$ this is true for any sufficiently small $c^*>0$; and as $c$ increases beyond $c^*$ it will remain true, provided that $\ol{F}'(c) < 1$ also remains true. 
The latter condition indeed remains true, since the numerator in the last term of \eqref{eq:Fbrefac} is strictly negative for all $c \in (0,c_\mathcal{I})$. Thus $\ol{F}'(c) < 1$ on the whole interval $(0,c_\mathcal{I})$.
We conclude in summary that the solution $\Fb$ to the ODE \eqref{F'} on $(0,c_\mathcal{I})$ satisfies
$$
\ol{F}(c) \in \Big(\frac1{2\d}+c, \atl \wedge (f_0+c)\Big) \quad \text{and} \quad \ol{F}'(c) < 1, 
\quad \text{for all } c\in(0,c_\mathcal{I}).
$$

\vspace{1mm}
{\it (iii) Proof that $\ol{F}'(c_\mathcal{I}) = 1$ if $c_\mathcal{I} < \tild{c}$}.  
It follows from part $(ii)$ that $\Fb(c_\mathcal{I}) < f_0 + c_\mathcal{I}$. Then recalling \eqref{dH3>0}, the limit as $c \to c_\mathcal{I}$ in \eqref{eq:Fbrefac} is well defined; by Lemma \ref{lem:soltildeq}, if $c_\mathcal{I} < \tild{c}$ then $\ol F(c_\mathcal{I}) = \atl = \mathcal{X}(\ol{F}(c_\mathcal{I}))$ and the limit is indeed 1.

\vspace{1mm}
{\it (iv) Proof that $\ol{F}'(c) >0$ for all $c\in(0,c_\mathcal{I})$.}  
Continuing the above analysis, for $c\in(0,c_\mathcal{I})$ we have from \eqref{F'} that
\begin{align} \label{F'>0}
\ol{F}'(c) > 0 
\quad &\Leftrightarrow \quad 
h_3(\mathcal{X}(\ol{F}(c))) > \sqrt{2\a} \tild \cH_3(\ol{F}(c),c) + \pd{\tild \cH_3}{c}(\ol{F}(c),c). 
\end{align}
For this, note first from \eqref{eq:th3} and \eqref{eq:qdef} that 
\begin{align*}
\frac{\partial}{\partial c} \Big(\pd {\tild \cH_3} c + \sta \tild \cH_3 \Big)(x,c) 
= \sta \d e^{-x\sta} \Big(\frac1{2\d} + c - x \Big) < 0, 
\quad \text{for} \quad x > \frac1{2\d} + c.
\end{align*}
Then if $c \in (0,c_\mathcal{I})$ we have $\ol{F}(c) > \frac1{2\d}+c$ (Lemma \ref{lem:olF>12d}) and
\begin{align*}
\Big(\pd {\tild \cH_3} c + \sta \tild \cH_3 \Big)(\ol F(c),0) 
> \Big(\pd {\tild \cH_3} c + \sta \tild \cH_3 \Big)(\ol F(c),c).
\end{align*}
Thus in order to prove the right-hand side inequality \eqref{F'>0} it suffices to show that 
\begin{align} \label{F'>02}
h_3(\mathcal{X}(\ol{F}(c))) > \sqrt{2\a} \tild \cH_3(\ol{F}(c),0) + \pd{\tild \cH_3}{c}(\ol{F}(c),0) , \quad c\in(0,c_\mathcal{I}). 
\end{align}
To that end, it follows from \eqref{eq:h3}, \eqref{eq:th3} and \eqref{eq:defrho} that 
\begin{align*}
h_3(\mathcal{X}(\ol{F}(c))) 
&= e^{- \mathcal{X}(\ol{F}(c)) \sta}\Big( \ha - \fla \Big( \mathcal{X}(\ol{F}(c)) + \frac 1 \sta \Big) \Big)  ,\\
{\tild \cH_3}(\ol F(c),0)
&= e^{- \ol F(c) \sta} \ha\Big(\d - \frac\la\a \Big) 
\Big(\Fb^2(c) + \frac{2 \Fb(c)}{\sta} - \frac{\la}{\a(\a\d-\la)}\Big),\\
\pd{\tild \cH_3}{c}(\ol F(c),0)
&= e^{- \ol F(c) \sta} \Big( \ha - \delta \Big( \ol F(c) + \osta \Big) \Big). 
\end{align*}
Then defining the function 
$$
u:(0,\infty)^2 \to \R, \quad 
u(x,s) := \Big(\frac12 - s \Big(x+\frac1\sta \Big) \Big) e^{-x \sta}, 
$$
the desired inequality in \eqref{F'>02} becomes 
\begin{align} \label{F'>03}
u\Big(\mathcal{X}(\ol{F}(c)), \frac\la\a \Big) - u(\ol{F}(c),\d) 
&> \frac{\sqrt{2\a}}{2}\Big(\d - \frac\la\a \Big) e^{- \ol F(c) \sta} \rho(\ol F(c)) , \quad c\in(0,c_\mathcal{I}). 
\end{align}
To show that this holds, we treat the two sides of the inequality separately. 
The left-hand side of \eqref{F'>03} satisfies
\begin{align} \label{F'>03lb}
u\Big(\mathcal{X}(\ol{F}(c)), \frac\la\a \Big) - u(\ol{F}(c),\d) 
&> u\Big(\atl, \frac\la\a \Big) - u\Big(\atl,\d \Big) \notag\\
&= \Big(\d - \frac\la\a \Big) \Big(\atl+\frac1\sta \Big) e^{-\atl \sta},
\end{align}
due to the monotonicity of $u(\cdot,s)$ derived from 
$$
\pd{u}{x}(x, s) 
= \sta s \Big(x - \frac1{2s} \Big) e^{-x \sta},
$$
and the fact that $\frac1{2\d} < \ol{F}(c) < \atl < \mathcal{X}(\ol{F}(c))$ for all $c\in(0,c_\mathcal{I})$ due to Lemma \ref{lem:olF>12d}. On the other hand, it follows from \eqref{eq:defrho} that 
\begin{align*}
\frac{d}{dx}\Big( e^{- x \sta} \rho(x) \Big) 
&= \sta e^{- x \sta} \Big(\frac{\d}{\a\d-\la} -x^2 \Big), 
\end{align*}
so its maximum value is attained at $x = \sqrt{\frac{\d}{\a\d-\la}}$.
Thus the right-hand side of \eqref{F'>03} satisfies
\begin{align} \label{F'>03ub}
\frac{\sqrt{2\a}}{2}\Big(\d - \frac\la\a \Big) e^{- \ol F(c) \sta} \rho(\ol F(c)) 
&< \frac{\sqrt{2\a}}{2}\Big(\d - \frac\la\a \Big) e^{- \sqrt{\frac{\d}{\a\d-\la}} \sta} \rho\Big( \sqrt{\frac{\d}{\a\d-\la}} \Big) \notag\\
&= \Big(\d - \frac\la\a \Big) \Big(\sqrt{\frac{\d}{\a\d-\la}} + \frac1{\sqrt{2\a}}\Big) e^{- \sqrt{\frac{\d}{\a\d-\la}} \sta},
\end{align}
Then \eqref{F'>03lb} and \eqref{F'>03ub} establish \eqref{F'>03}, since 
\begin{align*}
\Big(\atl+\frac1\sta \Big) e^{-\atl \sta}
\geq \Big(\sqrt{\frac{\d}{\a\d-\la}} + \frac1{\sqrt{2\a}}\Big) e^{- \sqrt{\frac{\d}{\a\d-\la}} \sta}, 
\end{align*}
which follows from the facts that 
\begin{align*}
x\mapsto \Big(x+\frac1\sta \Big) e^{-x \sta} \text{ is strictly decreasing on $[0,\infty)$, and } \sqrt{\frac{\d}{\a\d-\la}} > \atl
\end{align*}
for all $\la\in(\la^*,\la^\dagger)$, which follows from the fact that
\begin{align*}
\la \mapsto \sqrt{\frac{\d}{\a\d-\la}} - \atl 
\text{ is strictly increasing when $\la < \a\d$, and } \sqrt{\frac{\d}{\a\d-\la^*}} > \frac\a{2\la^*},
\end{align*}
thanks to the explicit expression \eqref{lambda*} for $\la^*$. 
This completes the proof.~\hfill $\Box$

\vspace{3mm}
\noindent {\bf Proof of Proposition \ref{prop:olG'}.}
Let $c \in (0,c_\mathcal{I})$. The first claim follows since
\begin{equation*}
\ol{G}'(c) = \mathcal{X}'(\ol{F}(c)) \ol{F}'(c), \quad c\in(0,c_\mathcal{I}),
\end{equation*}
while $\ol{F}'(c)>0$ (Proposition \ref{prop:olF'}), and $\mathcal{X}'(\ol{F}(c)) < 0$ (Proposition \ref{cor:xf}).

The bounds then follow immediately from the above monotonicity of $\ol{G}$, the definition \eqref{tildeg0} of $g_0$, the definition \eqref{cf} of $c_\mathcal{I}$ and the fact that $\ol F(c_\mathcal{I}) = \atl = \mathcal{X}(\ol{F}(c_\mathcal{I}))$ if $c_\mathcal{I} < \tild{c}$ (cf.~Lemma \ref{lem:soltildeq} and Lemma \ref{lem:olF>12d}). 
\hfill $\Box$

\vspace{3mm}
\noindent {\bf Proof of Lemma \ref{g0<gd}.}
Fix $\la\in (\la^*, \la^\dagger)$ and recall that $F(c) \in (0,\fotd) \subset (0,f_0)$ for all $c\in (0,c_\mathcal{I}) \subset (0,c_0)$ (this follows from Proposition \ref{prop:monotonicity}, \eqref{eq:longineq2} and \eqref{cI<c0}). 
Hence, the uniqueness of the root of $q(\cdot; F(c))=0$ follows from \cite[Lemma 8.2]{KOWZ00}. 
The fact that $q (g_\d; \frac1{2\d}) = 0$ is a consequence of \eqref{cor:1/2d0}, since $\lim_{c\to 0} F(c) = \fotd$.

Given that $\lim_{c\to 0} \Fb(c) = f_0$ (Proposition \ref{prop:newbdrs}), it will suffice to prove that the functions $q(\cdot; \fotd)$ and $\tild{q}(\cdot; f_0)$ satisfy
\begin{align} \label{qtildq}
q\Big(x; \frac1{2\d}\Big) > \tild{q}\big(x; f_0 \big), 
\quad \text{for all } x \in \big(f_0, g_0 \big).
\end{align}
This is because as $x$ increases from $f_0$ to $g_0$, both functions are non-negative and increasing until $x=\atl$ and decreasing afterwards (cf.\ \cite[Lemma 8.2]{KOWZ00} and Lemma \ref{lem:soltildeq}), so \eqref{qtildq} implies the required ordering on their respective roots.

We observe from the definitions of \eqref{eq:qdef} and \eqref{eq:tildeqdef} of $q$ and $\tild{q}$ that \eqref{qtildq} is equivalent~to
\begin{align} \label{qtildq2}
&\fla e^{x\sta} \Big(\atl -x -\frac 1 \sta \Big) e^{2(\frac1{2\d}-x)\sta} - \frac{\a\d-\la}{\a\d} e^{ \frac1{2\d} \sta} \notag\\
&> \fla e^{x\sta} \Big(\atl -x -\frac 1 \sta \Big) e^{2(f_0(\la)-x)\sta} -\Big(1- \frac{2\la}{\a} f_0(\la)\Big) e^{ f_0(\la) \sta} \notag\\
\Leftrightarrow \quad 
&\Big(x - \atl + \frac 1 \sta \Big) e^{2(\frac1{2\d}-x)\sta} + 2 \Big(\frac{\a}{2\la} - \frac{1}{2\d} \Big) e^{ (\frac1{2\d}-x) \sta} \notag\\
&<\Big(x - \atl + \frac 1 \sta \Big) e^{2(f_0(\la)-x)\sta} + 2 \Big(\frac{\a}{2\la} - f_0(\la)\Big) e^{(f_0(\la)-x)\sta} .
\end{align}
Then, we define the function $\varphi(\cdot;x)$ on $(\frac1{2\d}, \atl)$ by
\begin{align*} 
\varphi(z;x) 
:= \Big(x - \atl + \frac 1 \sta \Big) e^{2(z-x)\sta} + 2 \Big(\frac{\a}{2\la} - z \Big) e^{ (z-x) \sta} 
\end{align*}
and observe after straightforward calculations that 
\begin{align*} 
\varphi'(z;x) 
&= 2\sta e^{(2z-x)\sta} \Big[ \Big(x - \atl + \frac 1 \sta \Big) e^{-x\sta} - \Big(z - \frac{\a}{2\la} + \frac1\sta \Big) e^{-z\sta} \Big] \\
&= - \frac{2\a\sta}\la e^{(2z-x)\sta} \big(h_3(x) - h_3(z)\big), 
\quad \text{for } z \in \Big(\frac1{2\d}, \atl \Big),  
\end{align*}
where $h_3$ is defined by \eqref{eq:h3}.
Combining the fact that $\varphi'(z;z) = 0$ with the derivatives of $h_3$, namely
\begin{align*} 
h_3'(x) 
= \sta \fla \Big( x-\atl \Big) e^{-x\sta} \begin{cases}
<0, &\text{for } z < x < \atl ,\\
>0, &\text{for } x > \atl ,
\end{cases}
\end{align*}
and the results of Lemma \ref{lem:dtildeqdz}, namely that for every $z \in (\frac1{2\d}, \atl)$ we have $h_3(\mathcal{X}(z)) < h_3(z)$, we conclude that   
\begin{align*}
\varphi'(z;x) > 0 , 
\quad \text{for } z \in \Big(\frac1{2\d}, \atl \Big) \text{ and } x \in \big(z, \mathcal{X}(z)\big).   
\end{align*}
Using this result we see that \eqref{qtildq2} holds true for all $x \in (z, \mathcal{X}(z))$, since $\fotd < f_0 < \atl < \mathcal{X}(z)$. 
Then taking $z = f_0 \in (\frac1{2\d}, \atl)$, we have $\mathcal{X}(z) = g_0$ by definition (cf.~\eqref{tildeg0}). 
Thus \eqref{qtildq} holds, completing the proof.
\hfill $\Box$

\vspace{3mm}
\noindent {\bf Proof of Proposition \ref{cor:ctild}.} 
Aiming for a contradiction, suppose that in the dichotomy \eqref{cftildc} we have $c_\mathcal{I} = \tild c \leq c_g < \ol{c} \wedge \ol{k}$ (cf.~Lemma \ref{cf<infinite}) and start the process $(X,C)$ with the fuel level $C_0 = \tild c$.
Then we have from Proposition \ref{prop:monotonicity} and Lemma \ref{lem:olF>12d} that 
$$
F(\tild{c}) < \frac1{2\d} < \frac1{2\d}+\tild{c} 
< G(\tild{c}) =\ol{F}(\tild{c}) \leq \atl \leq \ol{G}(\tild{c}). 
$$
Thus in this case the waiting regions $[F(\tild{c}),G(\tild{c})]$ and $[\ol{F}(\tild{c}),\ol{G}(\tild{c})]$ have a common boundary point $x = G(\tild{c}) =\ol{F}(\tild{c})$. 
Given that the candidate value function $x \mapsto \tild{Q}(x,c)$ is of class $C^1$ across $G$ and $\Fb$ (by construction), we have 
\begin{align*}
&\tild{Q}(G(\tild{c}),\tild{c}) 
= \tild{V}_0(G(\tild{c}) - \tild{c}) + \tild{c} 
= \tild{V}_0(\Fb(\tild{c}) - \tild{c}) + \tild{c} 
= \tild{Q}(\Fb(\tild{c}),\tild{c}), \\
&\pd{\tild{Q}}{x}(G(\tild{c}),\tild{c}) 
= \tild{V}_0'(G(\tild{c}) - \tild{c}) 
= \tild{V}_0'(\Fb(\tild{c}) - \tild{c}) 
= \pd{\tild{Q}}{x}(\Fb(\tild{c}),\tild{c}). 
\end{align*}
Hence, comparing \eqref{eq:tildQ} and \eqref{eq:VstarRR} (since $c_\mathcal{I} < c_0$ by \eqref{cI<c0}) gives $A(\tild c)= \tild A(\tild c)$ and $B(\tild c)= \tild B(\tild c)$.
In this case our candidate solution therefore satisfies
\begin{equation*}
\tild{Q}(x,\tild c) := \fla x^2 + \flas + A(\tild c) e^{x\sta} + B(\tild c) e^{-x\sta} ,
\quad x \in (F(\tild c), \bar G(\tild c)), 
\end{equation*}
with boundary conditions at $F(\tild c)$ and $\Gb(\tild c)$ corresponding respectively to discretionary stopping and reflection. 
This system of equations was analysed in \cite[Section 8]{KOWZ00}, and its solution is characterised by the equation
$q(\ol{G}(\tild{c});F(\tild{c})) = 0$. 

Then by definition (cf.~Lemma \ref{g0<gd}) we have $\ol{G}(\tild{c}) = G_\d (\tild{c})$, which is impossible, as follows. 
Since $F$ is decreasing due to Proposition \ref{prop:monotonicity}, \cite[Lemma 8.3]{KOWZ00} gives that $G_\d(\cdot)$ is increasing on $(0,c_\mathcal{I})$. 
However $\ol{G}(\cdot)$ is decreasing on $(0,c_\mathcal{I})$ by Proposition \ref{prop:olG'}, hence, given that $g_0 < g_\d$ thanks to Lemma \ref{g0<gd}, we have the desired contradiction.
\hfill $\Box$

\vspace{3mm}
\noindent {\bf Proof of Lemma \ref{lem:Hr*}.}
Fixing $c>c_\mathcal{I}$, the expression \eqref{def-H*} is an immediate consequence of \eqref{c-z=0}, and next we establish the convexity of $H^*_r(\cdot;c)$. 
Firstly, the obstacle $H^*_r(\cdot;c)$ is strictly convex on $(\Psi(D(c)), \Psi(g_0 + c))$ due to Lemma \ref{lem:geomlemma} and the inequality \eqref{eq:H*pos}. 
It is also strictly convex on $(y_c, \Psi(D(c)))$ and on $(\Psi(g_0 + c),\infty)$ by Lemma \ref{lem:Hconvexity} and the bounds established in \eqref{eq:someineqs}.
Thus to establish that $H^*_r(\cdot;c)$ is strictly convex on the whole of $(y_c,\infty)$, it suffices to show that it is of class $C^1$. 
In view of the smoothness of the transformation $\Phi$ of \eqref{def-H}, it would suffice if the function $h_r^*(\cdot;c)$ was of class $C^1$. 
The latter follows from the Ansatz \ref{ansatz2}.   
\hfill $\Box$

\vspace{3mm}
\noindent {\bf Proof of Proposition \ref{0}.}
Recalling the boundaries $F$ and $G$ of Section~\ref{sec:cpos}, note first from Figure~\ref{fig:regionsnew} that $G'(c^*) \geq 1$, and that $G(c) < D(c) < f_0+c$ on $(c_\mathcal{I},c^*)$ (the final inequality comes from \eqref{eq:someineqs}). Thus on $(c_\mathcal{I},c^*)$, the situation is that of Section~\ref{Sec:verification2a}: namely, the boundary $G$ is of repelling type and its equations are characterised by double tangency between $H_l$ and $H_{r1}$ (Proposition \ref{existencegeometricy2}).
Then since $H_r$ and $H_r^*$ coincide for $y \in (y_c, \Psi(D(c))]$ (cf.~\eqref{def-H*}), they coincide in particular at $y=\Psi(G(c))$ for $c \in (0,c^*)$. It follows that for $c \in (0,c^*)$, the pair $(\Psi(F(c)),\Psi(G(c)))$ is a solution to \eqref{eq:system-y1y20}. Further, the strict convexity of $H_r^*(\cdot,c)$ on $(y_c,\infty)$ (Lemma \ref{lem:Hr*}) ensures that this solution is unique, and we have established \eqref{eq:FGeq}.

Recall that the boundary $F$ of Section~\ref{sec:cpos} satisfies $F(c)>0$ for all $c > 0$, and that $D(c) < f_0+c$ from  \eqref{eq:someineqs}. 
Thus \eqref{eq:FGeq} gives $c^* < c^*_1$. Together with \eqref{pos:1c*}--\eqref{pos:2c*}, which imply that $c_\mathcal{I} < c^*$, we obtain \eqref{cf<c1}.

For $c \in [c^*,c^*_1)$, the existence of a unique solution to \eqref{eq:system-y1y20} follows as in Step 1 of Proposition \ref{existencegeometricy2}, after again recalling the strict convexity of $H^*_r(\cdot,c)$ on $(y_c,\infty)$. Finally, the strict positivity of $\ol{c}$ is established just as in Section~\ref{sec:cpos}. 
\hfill $\Box$

\vspace{3mm}
\noindent {\bf Proof of Lemma \ref{q=0*}.}
We prove the desired statement in each case of \eqref{c-order} separately. 

\vspace{1mm}
{\it Case 1: $\bar{c} \leq c^*$}. In this case the boundaries $F$ and $G$ are simply those of Section~\ref{sec:cpos} (see Remark \ref{rem:oldb} or \eqref{eq:FGeq}), and the claim is established as in Section~\ref{Sec:verification2c}.

\vspace{1mm}
{\it Case 2: $c^* < \bar{c} < c^\dagger$.}
In this case we have $G(c) < g_0+c$ for all $c \leq \bar c$ and $G(\ol c) \in (D(\ol c), g_0+\ol c)$. 
Analogously to \eqref{eq:th3} and \eqref{eq:th4}, we have 
\begin{align} \label{eq:th3*}
\begin{split}
\cH_3^*(x,c)
&:= \pd {H_{r}^*}y (y;c)\Big|_{y=\Psi(x)} 
=\frac{e^{-x\sta}}{2\sta}\Big( \pd {h_{r}^*}x (x;c) + \sta {h_{r}^*}(x;c)\Big) ,\\
\cH_4^*(x,c)
&:= \Big( H_{r}^* - y \pd {H_{r}^*}y \Big) (y;c)\Big|_{y=\Psi(x)} 
= \frac{e^{x\sta}}{2\sta}\Big( - \pd {h_{r}^*}x (x;c) + \sta {h_{r}^*}(x;c)\Big) .
\end{split}
\end{align}
Then, we may use in \eqref{eq:th3*} the equality \eqref{dxh*r} to conclude exactly as in Lemma \ref{derivL} (Section~\ref{sec:FG}) that the function $L^*(\cdot,\cdot)$ defined by 
\begin{equation*}
L^*(x,c):=
\cH_4^*(x,c) - h_2\circ h_1^{-1} \big(\cH_3^*(x,c)\big),
\end{equation*}
satisfies
\begin{align*}
&\pd {L^*}{x} (x,c) = \big(e^{2z\sta}-e^{2x\sta} \big) \, \pd {\cH_3^*}x (x,c), \quad \text{where } z:=h_1^{-1}\big(\cH_3^*(x,c) \big), \\
&\Big(\pd {L^*}{x}+\pd {L^*}{c} \Big)(G(c),c) = q\big(G(c);F(c)\big), \quad \text{where $q$ is defined by \eqref{eq:qdef}.} 
\end{align*} 
Hence, exactly as in Lemma \ref{cor:FGC1}, we obtain for all $c \in (c^*,c^\dagger)$ that $\pd{L^*}{x}(G(c),c)<0$ and 
\begin{equation*}
G'(c) = 1 - \frac{q(G(c);F(c))}{\pd{L^*}{x}(G(c),c)}.
\end{equation*}
For $c<\ol{c}$ we have $G'(c)>1$ by definition, giving $q(G(c), F(c)) >0$. Further, $G'(\ol c)=1$ is equivalent to $q(G(\bar c), F(\bar c)) = 0$.

\vspace{1mm}
{\it Case 3: $c^\dagger \leq \bar{c}$.}
Again recalling Remark \ref{rem:oldb}, in this case for $c \in (c^\dagger, \bar{c}]$, the boundaries $F(c)$ and $G(c)$ are simply those of \cite[Section 10]{KOWZ00}, and the claim follows from equation (10.13) of the latter paper.
\hfill $\Box$

\newpage
\section{List of figures}

\begin{table}[htp]
\centering
\begin{tabular}[!ht]{c|p{11cm}|c}
Figure~\ref{fig:boundaries0} &
Moving boundaries of the control problem when $\la \geq \a \d$ 
(obtained in \cite{KOWZ00})
& p\pageref{fig:boundaries0} \\    
Figure~\ref{fig:boundaries4} & 
Moving boundaries of the control problem when $\fotd < \atl \leq f_0$ and $\la \in [\la^\dagger, \a \d)$ &  p\pageref{fig:boundaries4}  \\
Figure~\ref{fig:boundaries2} & 
Moving boundaries of the control problem when $\fotd < f_0 < \atl$ and $\la \in (\la^*, \la^\dagger)$ &  p\pageref{fig:boundaries2}  \\
Figure~\ref{fig:boundaries1} & 
Moving boundaries of the control problem when $f_0 \leq \fotd$ and $\la \in (\la_*,\la^*]$ (obtained in \cite{KOWZ00}) 
&  p\pageref{fig:boundaries1}  \\
Figure~\ref{fig:1} & 
A geometric representation of the optimal stopping problem $V_0(\cdot)$ of \eqref{eq:protonofuel} when $\lambda \in (\lambda^*, \a \d)$
&  p\pageref{fig:1}  \\
Figure~\ref{fig:Fig1synth} & 
A geometric representation of the optimal stopping problem $V(\cdot;c)$ of \eqref{def-V} for $\la \in [\la^\dagger, \a\d)$ and $c>0$ fixed and sufficiently small &  p\pageref{fig:Fig1synth}  \\
Figure~\ref{fig:Fig1synth2} & 
A geometric representation of the optimal stopping problem $V(\cdot;c)$ of \eqref{def-V} for $\la \in (\la^*,\la^\dagger)$ and $c>0$ fixed and sufficiently small &  p\pageref{fig:Fig1synth2}  \\
Figure~\ref{fig:regions} & 
Regions I to IV from the statement and proof of Theorem \ref{thm:ver2}, when $\la\in (\la^*, \la^\dagger)$ and $\bar{c} < c^*$ &  p\pageref{fig:regions}  \\
Figure~\ref{fig:regionsnew} & 
Regions I to IV from the statement and proof of Theorem \ref{thm:ver2}, when $\la\in (\la^*, \la^\dagger)$ and $c^* < \bar c$ &  p\pageref{fig:regionsnew}  
\end{tabular}
\label{tab:my_label}
\end{table}

\newpage
\section{List of symbols}

\begin{table}[htp]
\centering
\begin{tabular}[!ht]{c|c|c}
Symbol & Definition & Reference\\ \hline 
$\tild \cH_3(x,c)$ & $\frac{e^{-x\sta}}{2\sta} \big( \pd{h_{r1}}x (x;c) + \sta {h_{r1}}(x;c) \big)$ & \eqref{eq:th3} \\
$\tild \cH_4(x,c)$ & $\frac{e^{x\sta}}{2\sta} \big( -\pd {h_{r1}}x (x;c) + \sta {h_{r1}}(x;c) \big)$ & \eqref{eq:th4} \\
$B_0$ & $- \frac{2f_0}{\a \sta}(\a \d - \la)e^{f_0\sta}$ & \eqref{eq:B0} \\
$\rho(x)$ & $x^2+\frac{ 2 x}{\sqrt{2\a}} -\frac{\la}{\a(\a\d-\la)}$ & \eqref{eq:defrho} \\
$f_0$ & $\frac 1 \sta \big(\sqrt{\frac{\a \d + \la}{\a \d - \la}}-1 \big)$ & \eqref{eq:fodef} \\
$x_c$ & $\frac 1 {2\d}+\frac c 2$ & \eqref{x_c} \\
$y_c$ & $\Psi(x_c)$ & \eqref{ys} \\
$\hye(c)$ & $\Psi(f_0+c)$ & \eqref{ys} \\
$\hyc(c)$ & $\Psi\big(\atl + \frac c2\big)$ & \eqref{ys} \\
$y_v(c)$ &  Lemma \ref{lem:Hconvexity} I.(i), II.(i) & Lemma \ref{lem:Hconvexity} \\
$\oc$ & $2 \ \big(\sqrt{\frac{\d}{\a\d-\la}} - \frac 1{2\d} \big)$ & \eqref{Kk} \\
$\uc$ & $2 \ \big(\atl - f_0 \big)$ & \eqref{Kk} \\
$r_y(z;c)$, $P_r(y;c)$ & Definition \ref{def:tangents} & Definition \ref{def:tangents} \\
$\hat{c}$ & $\inf\{c \in (0,\infty): \hyt(c) \geq \hye(c) \}$ & \eqref{eq:c2defb} \\
$c_m$ & $\inf\{c \in (0,\infty): \hyo(c) = 1 \}$ & \eqref{eq:cm} \\
$c_1$ & $\hat{c} \wedge c_m$ & \eqref{eq:c2def} \\
$\ol{c}$ & $\inf \{c \in (0,c_1): G'(c) \leq 1\}$ & \eqref{lem:Gderv} \\
$c_0$ & $c_0 = c_1 \wedge \ol{c} = \hat{c} \wedge \bar c$ & \eqref{lem:Gderv}, Corollary \ref{cor:c1} \\
$\tild c$ & $\inf\{ c \in [0,\infty) \,:\, G(c) = \ol{F}(c)\}$ & \eqref{tildc} \\
$g_0$ & $\mathcal{X}(f_0)$ & \eqref{tildeg0} \\
$c_\mathcal{I}$ & $\inf \{c \in [0,\tild c] : \ol{F}(c) = \atl \} \wedge \tild c$ & \eqref{cf} \\ 
$\mathcal{X}$ & Lemma \ref{lem:soltildeq} & Lemma \ref{lem:soltildeq} \\
$c_g$ & $G(c_g) = \atl$ & Lemma \ref{cf<infinite} \\
$D(c)$ & $\atl - c_\mathcal{I} + c , \quad c \geq c_\mathcal{I}$ & \eqref{eq:c*} \\
$c^*$ & $\inf\big\{c \in [c_\mathcal{I},\infty) \,:\, G(c) = D(c) \big\}$ & \eqref{eq:c*} \\
$c^\dagger$ & $\inf\big\{c \in (c^*,\infty) \,:\, G(c) = g_0 + c \big\}$ & \eqref{eq:cdag} \\
$h_1(x)$ & $\frac{\a \d -\la}{2\a}e^{-x\sta}\rho(x)$ & \eqref{eq:hidefs1} \\
$h_2(x)$ & $\frac{\a \d -\la}{2\a}e^{x\sta}\big(\rho(x) - \frac{4x}{\sta} \big)$ & \eqref{eq:hidefs2} \\
$h_3(x)$ & $\fla \big( \atl - \big( x+\frac 1 \sta \big) \big) e^{-x\sta}$ & \eqref{eq:h3} \\
$h_4(x)$ & $\fla \big( x - \big( \atl +\frac 1 \sta \big) \big) e^{x\sta}$ & \eqref{eq:h4} \\
$q(x;z)$ & $\sta \big(h_2(z)-h_1(z)e^{2z\sta}\big) + h_3(x)e^{2z\sta} - h_4(x)$ & \eqref{eq:qdef} \\
$F(c)$ & $\Psi^{-1}(\hyo(c))$ & \eqref{eq:FGdef}, \text{Proposition \ref{0}} \\
$G(c)$ & $\Psi^{-1}(\hyt(c))$ & \eqref{eq:FGdef}, \text{Proposition \ref{0}} \\
\end{tabular}
\label{tab:my_label}
\end{table}

\end{appendix}


\begin{thebibliography}{99}
\bibitem{Bather67}J. A. Bather and H. Chernoff (1967). Sequential decisions in the control of a spaceship. In {\em Fifth Berkeley Symposium on Mathematical Statistics and Probability} Vol. 3, pp. 181-207. Berkeley: University of California Press.
\bibitem{BayraktarEgami08} E. Bayraktar and M. Egami (2008). Optimizing venture capital investments in a jump diffusion model. {\em Math. Methods Oper. Res.} 67(1), 21-42.
\bibitem{ChenYi12} X. Chen and F. Yi (2012). A problem of singular stochastic control with optimal stopping in finite horizon. {\em SIAM J.~Control Optim.} 50(4), 2151-2172.
\bibitem{DavisNorman90}M. H. Davis and A. R. Norman (1990). Portfolio selection with transaction costs. {\em Math. Oper. Res.} 15(4), 676-713.
\bibitem{DavisZervos94} M.H. Davis and M. Zervos (1994). A problem of singular stochastic control with discretionary stopping. {\em Ann. Appl. Probab.} 4(1), 226-240.
\bibitem{Dayanik2003} S.~Dayanik and I.~Karatzas (2003). On the optimal stopping problem for one-dimensional diffusions.
{\em Stoch. Proc. Appl.} 107(2), 173-212.
\bibitem{DeAFeMo15b} T. De Angelis, G. Ferrari and J. Moriarty (2018). Nash equilibria of threshold type for two-player nonzero-sum games of stopping. {\em Ann. Appl. Probab.} 28(1), 112-147.
\bibitem{DeAFeMo14} T. De Angelis, G. Ferrari and J. Moriarty (2019). A solvable two-dimensional degenerate singular stochastic control problem with nonconvex costs. {\em Math. Oper. Res.} 44(2), 512-531.
\bibitem{Dupuis93} P. Dupuis and H. Ishii (1993). SDEs with oblique reflection on nonsmooth domains. {\em Ann. Probab.} 21(1), 554–580. 
\bibitem{KOWZ00}I. Karatzas, D. Ocone, H. Wang and M. Zervos (2000). Finite-fuel singular control with discretionary stopping. {\em Stochastics} 71(1-2), 1-50.
\bibitem{LonZervos11} P. Lon and M. Zervos (2011). A model for optimally advertising and launching a product. {\em Math. Oper. Res.} 36(2), 363-376.
\bibitem{Morimoto10} H. Morimoto (2010). A singular control problem with discretionary stopping for geometric Brownian motions. {\em SIAM J.~Control Optim.}~48(6), 3781-3804.
\end{thebibliography}
\end{document}